\documentclass[11pt]{article}

\usepackage{amssymb,amsmath,mathtools}
 \usepackage{amsthm} 
 \usepackage{mathrsfs}

\usepackage[OT2, T1]{fontenc} 
\usepackage[utf8]{inputenc}

\usepackage[italian, english]{babel}

\usepackage[a4paper,textwidth=180true mm,lines=43]{geometry}

\usepackage{color}
\usepackage{bm}

\usepackage{braket}
\usepackage{enumitem}

\allowdisplaybreaks 

\usepackage
[draft=false,setpagesize=false,pdfstartview=FitH,hidelinks,bookmarks=true]
{hyperref}

\usepackage{aliascnt}
\usepackage[nameinlink,capitalise]{cleveref}

\usepackage[noadjust]{cite} 

\makeatletter
\renewenvironment{proof}[1][\proofname]{\par
    \pushQED{\qed}%
    \normalfont \topsep6\p@\@plus6\p@\relax
    \trivlist
    \item\relax
    {\bfseries
        #1\@addpunct{.}}\hspace\labelsep\ignorespaces
}{%
    \popQED\endtrivlist\@endpefalse
}
\makeatother

\theoremstyle{theorem}
\newtheorem{Theorem}{Theorem}[section]
\newtheorem{Proposition}[Theorem]{Proposition}
\newtheorem{Corollary}[Theorem]{Corollary}
\newtheorem{Lemma}[Theorem]{Lemma}
\newtheorem{Remark}[Theorem]{Remark} 

\theoremstyle{definition}
\newtheorem{Definition}[Theorem]{Definition}

\crefname{Theorem}{Theorem}{Theorems}
\crefname{Lemma}{Lemma}{Lemmas}
\crefname{Proposition}{Proposition}{Propositions}
\crefname{Corollary}{Corollary}{Corollaries}
\crefname{Definition}{Definition}{Definitions}
\crefname{Remark}{Remark}{Remarks}
\crefname{Example}{Example}{Examples}
\crefname{Section}{Section}{Sections}
\crefname{Subsection}{Section}{Sections}

\numberwithin{equation}{section}

\newcommand{\R}{\mathbb{R}}
\newcommand{\N}{\mathbb{N}}

\def\abs#1{|{#1}|}
\def\pabs#1{\left|{#1}\right|}
\def\norm#1{\|{#1}\|}
\def\pnorm#1{\left\|{#1}\right\|}

\def\meas{\mathop{\rm meas}\nolimits}
\def\esssup{\mathop{\rm ess\, sup}}
\def\wlimit{\rightharpoonup}

\def\half{{\frac{1}{2}}}

\def\calC{{\cal C}} 
\def\calD{{\cal D}}
\def\calE{{\cal E}}
\def\calI{{\cal I}}
\def\calM{{\cal M}}
\def\calO{{\cal O}}
\def\calS{{\cal S}}

\def\scrC{\mathscr{C}}
\def\scrP{\mathscr{P}}
\def\wscrC{\widetilde{\scrC}}

\def\epsilon{\varepsilon}
\def\sgn{\hbox{sgn}}

\def\RE{\R\times E}
\def\RRE{\R\times\R\times E}

\def\intRN{\int_{\R^N}}

\def\distM{{\rm dist}_M}
\def\distRE{{\rm dist}_{\RE}}
\def\distE{{\rm dist}_E}

\def\wKb{\widetilde K_b}
\def\wN{\widetilde N}

\def\ob{\overline b}
\def\ub{\underline b}

\def\oGamma{\overline \Gamma}
\def\uGamma{\underline\Gamma}

\def\orgweta{{\widetilde \eta}}
\def\weta{\mathchoice{\orgweta}{\orgweta}{\orgweta}{{\scriptstyle\tilde\eta}}}

\def\wzeta{{\widetilde\zeta}}

\def\wgamma{{\widetilde\gamma}}

\def\oepsilon{{\overline\epsilon}}

\def\wheta{{\widehat \eta}}

\def\whetan{\widehat \eta_n}
\def\ABS{\Upsilon}

\begin{document}

\title
{Normalized solutions for nonlinear Schr\"odinger equations\\
with $L^2$-critical nonlinearity}

\author{
        Silvia Cingolani
        \\ \normalsize{Dipartimento di Matematica, Universit\`{a} degli Studi di Bari Aldo Moro}
        \\ \normalsize{Via E. Orabona 4, 70125, Bari, Italy}
        \\ 
        \\Marco Gallo
        \\ \normalsize{Dipartimento di Matematica e Fisica, Universit\`{a} Cattolica del Sacro Cuore}
        \\ \normalsize{Via della Garzetta 48, 25133 Brescia, Italy}
        \\
        \\Norihisa Ikoma
        \\ \normalsize{Department of Mathematics, Faculty of Science and Technology, Keio University}
        \\ \normalsize{Yagami Campus, 3-14-1 Hiyoshi, Kohoku-ku, Yokohama, Kanagawa 223-8522, Japan}
        \\
        \\Kazunaga Tanaka 
        \\ \normalsize{Department of Mathematics, School of Science and Engineering, Waseda University}
        \\ \normalsize{3-4-1 Ohkubo, Shijuku-ku, Tokyo 169-8555, Japan}
        }   

\date{}

\maketitle

\abstract{
We study the following nonlinear Schr\"odinger equation and
we look for normalized solutions $(\mu,u)\in (0,\infty)\times H^1(\R^N)$ 
for a given $m>0$ and $N\geq 2$
    \[	-\Delta u + \mu u = g(u)\quad \hbox{in}\  \R^N, \qquad
        \frac{1}{2}\intRN u^2 dx = m.
    \]
We assume that $g$ has an $L^2$-critical growth, both at the origin and at infinity.
That is, for $p=1+\frac{4}{N}$, 
$g(s)=\abs{s}^{p-1}s +h(s)$, $h(s)=o(|s|^p)$ as $s\sim 0$ and $s\sim\infty$.
The $L^2$-critical exponent $p$ is very special for this problem;
in the power case $g(s) = \abs{s}^{p-1}s$ a solution exists only for the specific mass
$m=m_1$, where $m_1=\frac{1}{2}\intRN\omega_1^2\, dx$ is the mass of a least energy solution 
$\omega_1$ of $-\Delta \omega+\omega=\omega^p$ in $\R^N$.

We prove the existence of a positive solution for $m=m_1$ when $h$ has a sublinear 
growth at infinity, i.e., $h(s)=o(s)$ as $s\sim\infty$. 
In contrast, we show non-existence results for $h(s)\not=o(s)$ ($s\sim 0$)
under a suitable monotonicity condition.
}

\medskip

\noindent
\textbf{MSC2020:} 
35A01, 
35B33, 
35B38, 
35J20, 
35J91, 
35Q40, 
35Q55, 
47F10, 
47J30, 
49J35.

\noindent
\textbf{Key words:} 
Nonliear Schr\"odinger equations, 
nonlinear elliptic PDEs, 
%
normalized solutions, 
prescribed mass problem, 
$L^2$-critical exponent, 
Lagrangian approach.

\tableofcontents

\section{Introduction}

\subsection{Motivations and state of art}
\label{Section:1.1}
In this paper, we 
study the existence of $L^2$-normalized solutions to the following equation: 
    \begin{equation}\label{1.1}
    \begin{dcases}
        -\Delta u + \mu u = g(u) & \hbox{in $\R^N$}, 
        \\
        \frac{1}{2} \intRN u^2 \, dx = m,&
    \end{dcases}
    \end{equation}
where $N\geq 2$, $g:\,\R\to\R$ is a nonlinear function, that is not necessarily
of the power type, and the mass $m>0$ are given, while a pair $(\mu,u)\in\R\times H^1(\R^N)$
is unknown.

Problem \eqref{1.1} appears when searching for standing waves 
to the following nonlinear Schr\"odinger equation
    \begin{equation}\label{1.2}
        i \hbar \partial_t \psi = -\hbar^{2} \Delta \psi 
        - g(\psi), \quad (t,x) \in (0,T) \times \R^N,
    \end{equation}
i.e. solutions to \eqref{1.2} of the form 
$\psi (t,x) = e^{i \mu t } u(x)$ for some $\mu \in \R$. 
The constraint $\frac{1}{2}\intRN u^2 \, dx = m$ comes from the fact that, 
under suitable assumptions on $g$, 
a solution $\psi$ of \eqref{1.2} preserves such a quantity in time 
$\intRN \abs{\psi(t,x)}^2 \, dx = \intRN \abs{\psi(0,x)}^2 \, dx $ 
for all $t \in (0,T)$.

In the last decades the existence of $L^2$-normalized solutions has
been studied actively, and the literature is quite extended. 
The $L^2$-critical exponent $p$ defined by
    \begin{equation}\label{1.3}
        p=1+{\frac{4}{N}}
    \end{equation}
plays an
important role to analyze \eqref{1.1}, as well as stability properties of 
\eqref{1.2}, see \cite{Ca0}. Notice that $p \in (1, 2^*-1)$, where as usual 
$2^*= \frac{2N}{N-2}$ for $N\geq 3$ and $2^*=\infty$ for $N=2$.

\smallskip

To see the importance of $L^2$-critical exponent $p$, for $q\in (1,2^*-1)$ and $\mu>0$
we recall that the equation 
    \begin{equation*} 
        -\Delta w+\mu w= \abs{w}^{q-1} w \quad \text{in}\ \R^N 
    \end{equation*}
has a unique positive solution up to translations (c.f. Kwong \cite{Kwo1}), and 
we denote its unique positive radial solution by $w_{q,\mu}$.
We can easily observe that 
    \begin{equation*}
        w_{q,\mu} = \mu^{\frac{1}{q-1}} w_{q,1} (\mu^{1/2}\cdot),
    \end{equation*}
where $w_{q,1}$ is the unique positive radial solution of 
$-\Delta w+w=\abs w^{q-1}w$ in $\R^N$. 
Setting the mass $\calM(u)=\half\intRN u^2\,dx$ and the energy
$\calE_q(u)=\half\intRN\abs{\nabla u}^2\, dx -{\frac{1}{q+1}}\intRN \abs u^{q+1}\, dx$,
we infer from the Pohozaev identity that 
    \begin{align}
    &\calM(w_{q,\mu})=\mu^{{\frac{N}{2(q-1)}}(p-q)} \calM(w_{q,1}), \label{1.4}\\
    &\calE_q(w_{q,\mu})= \mu^{{\frac{2}{q-1}}-{\frac{N}{2}}+1} \calE_q(w_{q,1}), \quad
        \calE_q(w_{q,1})= {\frac{N(q-p)}{(N+2)-(N-2)q}} \calM(w_{q,1}).  \label{1.5}
    \end{align}
Thus we see that $p=1+{\frac{4}{N}}$ plays a special role. 
Indeed, considering three cases
    \[ \text{(1)}\ q\in(1,p), \qquad \text{(2)}\ q\in (p,2^*-1), \qquad \text{(3)}\ q=p,
    \]
we observe that
\begin{itemize}
\item\ [existence of a solution] it follows from \eqref{1.4} that
  \begin{itemize}
  \item for Cases (1) and (2), we have $q\not=p$ and thus for any $m>0$ problem \eqref{1.1} 
  has a solution $(\mu,u)$, $u=w_{q,\mu}$, where $\mu$ satisfies 
  $\mu^{{\frac{N}{2(q-1)}}(p-q)}\calM(w_{q,1})=m$;
  \item for Case (3), $\calM(w_{p,\mu})$ does not depend on $\mu$ and \eqref{1.1} has a 
  solution if and only if $m=\calM(w_{p,1})$, which is equivalent to $m=m_1$ for $m_1$
  given in \eqref{1.8} below.
  Moreover, the set of solutions of \eqref{1.1} is given by $\{ (\mu,w_{p,\mu})\,|\, \mu>0\}$,
  which is non-compact;
  \end{itemize}
\item\ [energy of a solution] by \eqref{1.5}, the energy $\calE_q(u)$ 
of a solution $(\mu,u)$ for \eqref{1.1} satisfies
    \[  \calE_q(u) \begin{cases} <0 &\text{for Case (1)}, \\
                                >0 &\text{for Case (2)}, \\
                                =0 &\text{for Case (3)}. 
                    \end{cases}
    \]
\end{itemize}


We consider now the situations which generalize (1)--(3).
We call $g$ \emph{$L^2$-subcritical at $0$} (resp. \emph{$L^2$-critical}, 
\emph{$L^2$-supercritical}) if 
    \[  \lim_{s \to 0} \frac{g(s)}{\abs{s}^{p-1} s} = \infty \quad 
        \left(\text{resp. $\lim_{s \to 0} \frac{g(s)}{\abs{s}^{p-1} s} \in (0,\infty)$, 
        \ $\lim_{s \to 0} \frac{g(s)}{\abs{s}^{p-1}s} = 0$}\right),
    \]
while it is \emph{$L^2$-subcritical at $\pm \infty$} (resp. \emph{$L^2$-critical}, 
\emph{$L^2$-supercritical}) if 
    \[  \lim_{\abs{s} \to \infty } \frac{g(s)}{\abs{s}^{p-1}s} = 0 \quad 
        \left(\text{resp. $\lim_{\abs{s} \to \infty} 
        \frac{g(s)}{\abs{s}^{p-1} s} \in (0,\infty)$, 
        \ $\lim_{ \abs{s} \to \infty } \frac{g(s)}{\abs{s}^{p-1}s} = \infty$}\right).
    \]

Existence and non-existence of solutions have been studied
extensively for $L^2$-subcritical 
and $L^2$-supercritical problems, as well as mixed cases.
We confine here to mention only few key papers.

For $L^2$-subcritical problems at $0$ and $\pm\infty$, the pioneering works 
are given by Stuart \cite{Stu3, Stu4} and Cazenave and Lions \cite{CaLi0}.
Afterwards, Shibata \cite{Shi0} studies general $L^2$-subcritical
problems via minimizing method on the $L^2$-sphere, that is by solving
 $ \inf_{u \in \calS_m} \calI(u)$
with
    \begin{equation}\label{1.6}
        \calS_m=\Set{ u\in H^1(\R^N)\,|\, \half\intRN u^2\, dx=m}
    \end{equation}
and
    \begin{equation}\label{1.7}
    \calI(u) = \intRN \frac{1}{2} \abs{\nabla u}^2 - G(u) \, dx, 
    \quad G(s) = \int_0^s g(t) \, dt. 
    \end{equation}
In \cite{HT0}, Hirata and Tanaka take an approach using the Lagrange formulation 
(see \cref{Section:1.3} for details) together with a deformation 
argument under
the Palais-Smale-Pohozaev condition; see also Ikoma and Tanaka \cite{IT0}. 
We mention also \cite{CGT1,CGT2,CGT4,CGT5,CT1,CT2,GalT,YaTaCh21} for related results 
on nonlocal problems.

For $L^2$-supercritical problems at $0$ and $\pm\infty$, we refer to Jeanjean \cite{Jea0} 
where the mountain pass argument is applied to $\calI$ on $\calS_m$ to find a solution;
see also Bartsch and de Valerola \cite{BV0}
for the multiplicity of solutions. We also refer to Soave \cite{So20} 
for the study of nonlinearities
$g(s)=\pm \abs s^{q_1-1}s +\abs s^{q_2-1}s$ with 
$1<q_1\leq p \leq q_2<2^*-1$, $q_1\not=q_2$,
which can be regarded as $L^2$-subcritical at $0$ and $L^2$-supercritical at $\pm\infty$,
and to Bieganowski and Schino \cite{BiSc24} for some recent generalizations.
See also \cite{ChTa,JJTV,Soa2} where $q_2=2^*-1$ is considered.
For further results and developments we also refer to 
\cite{AW,AJM,BS1,BM0,DST0,JL-1,JL0,JL1,JL2,JTL,LRZ,MeSc0,MeSc1,NTV1,NTV2,WW0} 
and references therein.

\smallskip

In this paper we 
deal with $g$ which is $L^2$-critical both at $0$ and $\pm\infty$. 
Being the growth fixed, hereafter we write, for the sake of convenience,
$\omega_1 = w_{p,1}$, $\omega_\mu = w_{p,\mu}$ and 
$m_1= \mathcal{M}(w_{p,1})$, namely
    \begin{align}
        &m_1 = \frac{1}{2} \intRN \omega_1^2 \, dx, \label{1.8}\\
        &\omega_\mu =
        \mu^{N/4} \omega_1 (\mu^{1/2} \cdot). \label{1.9}
    \end{align}
We consider the situation \eqref{1.3}, that is
    \[  g(s)\sim a_0 \abs s^{p-1}s \; \text{ as}\ s\sim 0, \qquad 
        g(s)\sim a_\infty \abs s^{p-1}s \; \text{ as}\ s\sim 
        \infty,
    \]
and we focus on the case $a_0=a_{\infty}$; this problem is particularly 
challenging and still open in the literature.


The case of different constants $a_0\not= a_\infty$, to the best of the authors' knowledge, 
has been tackled in only two papers. 
Indeed, Schino \cite{Sch0} and Jeanjean, Zhang and Zhang \cite{JeZhZh}
 consider the following situations among other results; \cite{Sch0} deals with
systems of nonlinear Schr\"odinger equations and \cite{JeZhZh} studies various problems 
including $L^2$-subcritical and $L^2$-supercritical nonlinearities.

Schino \cite{Sch0} considers the case $a_0>a_\infty$. More precisely, setting
    \[  \widetilde a_0 =\liminf_{s\to 0} {\frac{G(s)}{{\frac{1}{p+1}}\abs s^{p+1}}},
        \quad \widetilde a_\infty 
        =\limsup_{s\to \pm\infty} {\frac{G(s)}{{\frac{1}{p+1}}\abs s^{p+1}}}
    \]
and assuming\footnote{
In \cite{Sch0}, \eqref{1.10} is given as ${\frac{2}{p+1}}\widetilde 
a_\infty C_{GN} (2m)^{2/N} < 1  < {\frac{2}{p+1}}\widetilde a_0 C_{GN} (2m)^{2/N}$, 
where $C_{GN}$ is the best constant of the Gagliardo-Nirenberg inequality; 
indeed, it is known that $C_{GN}=\half (p+1)(2m_1)^{-2/N}$ 
(c.f. \cref{Proposition:2.1}). 
}
    \begin{equation}\label{1.10}
        \widetilde a_0^{-N/2} m_1 < m < \widetilde a_\infty^{-N/2} m_1, 
    \end{equation}
he finds a solution of \eqref{1.1} as a minimizer of $\calI|_{\calS_m}$; 
here $\mu \in \R$ is obtained as a Lagrange multiplier.  We note that
\cite{Sch0} also studies the situation $\widetilde a_0=\infty$.

We next state the result by Jeanjean, Zhang and Zhong \cite{JeZhZh}. 
They assume that $g \in C^1( [0,\infty) )$ satisfies $g>0$ in $(0,\infty)$, 
    \begin{equation*}
        b_0 = \lim_{s \to +0} \frac{g'(s)}{(p-1)s^{p-1}} \in (0,\infty), \quad 
        b_\infty = \lim_{s \to \infty} \frac{g'(s)}{(p-1)s^{p-1}} \in (0,\infty)
    \end{equation*}
and moreover
    \begin{equation}\label{1.11}
        \text{the equation $-\Delta u = g(u)$ in $\R^N$ has no positive, 
        radial decreasing, classical solution}.
    \end{equation}
Under these assumptions authors in 
\cite{JeZhZh} prove that if $m$ satisfies 
    \begin{equation}\label{1.12}
        b_{\rm max}^{ -N/2 } m_1 < m < b_{\rm min}^{-N/2} m_1, \quad 
        b_{\rm min} = \min \{b_0,b_\infty\}, \quad 
        b_{\rm max} = \max \{b_0, b_\infty\},
    \end{equation}
then \eqref{1.1} admits a positive solution. Their approach is different 
from variational arguments and 
they show the existence of positive solutions via a continuation method; 
in their study the asymptotic behaviors of positive solutions to 
    \begin{equation}\label{1.13}
        - \Delta u + \mu u = g(u) \quad \text{in} \ \R^N
    \end{equation}
as $\mu \to 0^+$ and $\mu \to + \infty$ are important. 
We also mention that the non-existence of positive solutions to \eqref{1.1} 
is also obtained in \cite{JeZhZh} provided 
$0 < m \ll 1$ or $m \gg 1$.

We notice that, under suitable assumptions, $a_0=\tilde{a}_0=b_0$ 
and $a_{\infty}=\tilde{a}_{\infty}=b_{\infty}$. 
In the above mentioned \cite{Sch0,JeZhZh} we remark that 
the assumption $a_0\neq a_{\infty}$ 
is crucial, 
since it allows to find an open nonempty interval for $m$ to get the existence 
(see \eqref{1.10} and \eqref{1.12}). 
We highlight that existence of solutions is commonly obtained on 
open intervals of this type, 
even in the $L^2$ subcritical and supercritical settings. 
    The case $a_0=a_{\infty}$ 
    is thus not studied and
the typical case $g(s)=\abs s^{p-1}s$, $m=m_1$ is excluded 
(in particular, $g(s)=s^3$, $m=m_1$ for $N=2$ is excluded).

\smallskip

The aim of this paper is to deal with the case $a_0=a_\infty$ and 
generalize the power case.

\subsection{Main results}
\label{Section:1.2}

    To state our results, we assume the following conditions on $g$: 
\begin{enumerate}[label={\rm (g\arabic*)}]
    \setcounter{enumi}{-1}
    \item \label{(g0)} $g\in C([0,\infty),\R)$,
    \item \label{(g1)}
    $g(s)\sim \abs s^{p-1}s$ as $s\sim 0$ and $s\sim\infty$ in the following
    sense:
    \begin{equation*}
        g(s) = \abs s^{p-1}s +h(s),
    \end{equation*}
    where $h\in C([0,\infty),\R)$ satisfies
    \begin{equation} \label{1.14}
        \lim_{s\to 0}{\frac{h(s)}{\abs s^{p-1}s}} =0, \quad
        \lim_{s\to\infty} {\frac{h(s)}{s}}=0.       
    \end{equation}
\end{enumerate}
In what follows, we write $H(s)=\int_0^s h(t)\, dt$.

In \ref{(g1)}, we assume that $h$ has sublinear growth at infinity. 
It is a strong assumption as a perturbation
of $\abs s^{p-1}s$; however, it turns out later in \cref{Theorem:1.7} 
that this condition is necessary for the
existence of a solution, which is in contrast with the results 
of \cite{Sch0, JeZhZh}.

We need an additional assumption 
on the following zero-mass problem, as in \cite{JeZhZh}, for the limit equation
as $\mu \to 0$: 
    \begin{equation} \label{1.15}
        -\Delta u = g(u) \quad \text{in}\ \R^N.
    \end{equation}
Under \ref{(g1)}, this problem is studied in the natural space
    \begin{equation} \label{1.16}
        F= \big\{u \in L^{p+1} (\R^N)  \mid \nabla u \in L^2(\R^{N}), \
        u(x) = u(\abs{x})\big\}
    \end{equation}
with the norm 
    \begin{equation} \label{1.17}
        \norm u_F = \norm{\nabla u}_2 + \norm u_{p+1}. 
    \end{equation}
For \eqref{1.15}, we assume
\begin{enumerate}[label={\rm (g\arabic*)}]
    \setcounter{enumi}{1}
\item \label{(g2)}
if $u_0\in F$ satisfies \eqref{1.15} and $u_0\geq 0$ in $\R^N$, 
then $u_0 \equiv 0$.
\end{enumerate}
We note that any weak solution $u\in F$ (in particular, radially symmetric) is
of class $C^2$ and thus classical.

\smallskip

Our main result deals with the existence of a positive solution. 

\begin{Theorem}[Existence] \label{Theorem:1.1}
Let $N\geq 2$ and $m_1=\half\intRN \omega_1^2\, dx>0$ be given in \eqref{1.8}.
Moreover assume \ref{(g0)}--\ref{(g2)}.
Then \eqref{1.1} with $m=m_1$ has a positive radially symmetric solution
$(\mu,u)\in (0,\infty)\times H^1
(\R^N)$.
\end{Theorem}

The following proposition shows \ref{(g2)} holds under \ref{(g0)}--\ref{(g1)}
for $N=2,3,4$.  However for $N\geq 5$, condition \ref{(g2)} is not easy to verify 
for general $g$. 
Possible sufficient conditions to ensure \ref{(g2)} is also given 
in the following proposition also for $N\geq 5$

\begin{Proposition} \label{Proposition:1.2}
Assume that $g$ satisfies \ref{(g0)} and \ref{(g1)}.  Then we have the following
\begin{itemize}
\item[(i)] When $N=2,3,4$, condition \ref{(g2)} holds;
\item[(ii)] When $N\geq 5$, assume
\begin{enumerate}[label={\rm (g\arabic*)}]
    \setcounter{enumi}{2}
\item \label{(g3)} $NG(s) - {\frac{N-2}{2}}g(s)s \geq 0$ for all $s\geq 0$.
\end{enumerate}
Then condition \ref{(g2)} holds.
\end{itemize}
\end{Proposition}

As a corollary to \cref{Theorem:1.1} and \cref{Proposition:1.2}, we have

\begin{Corollary}\label{Corollary:1.3}
Let $N\geq 2$ and $m_1>0$ be given in \eqref{1.8} and assume \ref{(g0)}, \ref{(g1)}.
Moreover assume \ref{(g3)} when $N\geq 5$.
Then \eqref{1.1} with $m=m_1$ has a positive radially symmetric solution
$(\mu,u)\in (0,\infty)\times H^1(\R^N)$.
\end{Corollary}

\begin{Remark} \label{Remark:1.4}
\begin{enumerate}[label={\rm (\roman*)}]
\item Condition \ref{(g3)} can be also found in \cite{JeZhZh,MeSc0} 
and, when $N\geq 3$, it is equivalent to 
    \begin{equation*}
        s \mapsto {\frac{G(s)}{\abs{s}^{2^*}}} \; 
        \hbox{is non-increasing in $(0,+\infty)$}.
    \end{equation*}
Note that condition \ref{(g1)} implies 
$\displaystyle \lim_{s\to 0} \frac{G(s)}{\abs{s}^{2^*}} =+\infty$ and 
$\displaystyle \lim_{s \to \infty} \frac{G(s)}{\abs{s}^{2^*}} =0$.
Hence
    \[  G(s)>0 \quad \text{for all}\ s>0.
    \]
\item 
To show \ref{(g2)} for $N=2,3,4$, we use a Liouville type result due to
Armstrong and Sirakov \cite{AS} and 
Alarc\'on, Garc\'{\i}a-Meli\'an and Quaas \cite{AGQ16} 
for a positive supersolution to \eqref{1.15}, see \cref{Section:2.4}. 
%
\end{enumerate}
\end{Remark}

In some situation, we can show the existence without assumption \ref{(g2)}.
From the argument for \cref{Theorem:1.1}, we are able to show

\begin{Corollary}\label{Corollary:1.5}
Assume \ref{(g0)}--\ref{(g1)} and $g(s)\geq \abs s^{p-1}s$ for $s\geq 0$.
Then \eqref{1.1} with $m=m_1$ has a positive radially symmetric solution
in $(\mu,u)\in (0,\infty)\times H^1(\R^N)$.
\end{Corollary}

\begin{Remark}\label{Remark:1.6}
If one is interested in finding solutions without information on their sign, it
is possible under $(g0\#)$--$(g2\#)$ in \cref{Appendix:A.3}
and the arguments in the proof get easier.  
We expect that results in \cref{Appendix:A.3} will be useful to study existence
and multiplicity of possibly sign-changing solutions.
\end{Remark}

\smallskip

We next observe that the sublinear growth of $h$ at infinity in \ref{(g1)} 
cannot be removed in the existence result.
To this aim, we introduce $ \rho:\, (0,+\infty) \to \R$ by 
    \begin{equation*}
        \rho(s) = {\frac{H(s)}{s^2/2}}.
    \end{equation*}
Notice that if we suppose \ref{(g1)}, then $\rho$ satisfies 
$\lim_{s \to \infty} \rho(s)=0$ 
and 
    \begin{equation} \label{1.18}
        \rho(s)=o(\abs s^{p-1}) \quad \text{as}\ s\sim 0.
    \end{equation}
Now we study the situation
    \begin{equation}\label{1.19}
        \alpha = \lim_{s \to \infty} \rho(s) \in [-\infty,0).
    \end{equation}

\begin{Theorem}[Nonexistence] \label{Theorem:1.7}
Let $N\geq 2$ and assume \ref{(g1)}, \eqref{1.18} and 
\begin{enumerate}[label={\rm ($\rho$\arabic*)}]
\item \label{(rho1)}
there exists $s_0>0$ such that
    \begin{equation*}
        \begin{aligned}
            &s \mapsto \rho(s) \; 
            \hbox{is non-increasing in $(0,+\infty)$,} \\
            &s \mapsto \rho(s) \; 
            \hbox{is strictly decreasing in $(0,s_0]$. }
        \end{aligned}
    \end{equation*}
\end{enumerate}
Then no positive solution of \eqref{1.1} with $m=m_1$ exists. 
\end{Theorem}

It is easy to give examples of $h$ for which $g(s)=\abs s^{p-1}s+h(s)$ satisfies 
\ref{(g3)} and
the corresponding $\rho$ satisfies the assumptions in \cref{Theorem:1.7}.

\begin{Remark}\label{Remark:1.8}
\begin{enumerate}[label={\rm (\roman*)}]

\item Condition \ref{(rho1)} implies \eqref{1.19}, 
which should be compared with the fact that $\alpha=0$ 
under \ref{(g1)}.
We also note that $\rho\in C([0,\infty),\R)\cap C^1((0,\infty),\R)$ and 
$H(s)-\half h(s)s = -{\frac{1}{4}}\rho'(s)s^3$.  Thus
the first condition in \ref{(rho1)} is equivalent to
    \[ 2H(s)- h(s) s \geq 0 \quad \text{for each $s \in [0,\infty)$}, 
    \]
while the second condition is equivalent to 
    \[  2H(s)-h(s) s > 0 \quad\text{for each $s \in D$},
    \]
where $D$ is a dense subset of $[0,s_0]$.

\item A special case of \cref{Theorem:1.7} is given by Soave \cite{So20}.
In \cite{So20}, non-existence of solutions is shown for $g(s)=\abs s^{p-1}
-\theta \abs s^{q-1}s$, where $1<q<p$ and $\theta>0$.  We note that $g$
is $L^2$-subcrtical at $s\sim\pm\infty$ and the corresponding 
$\rho(s)=-{\frac{2\theta}{q+1}}\abs s^{q-1}$ satisfies the assumptions of 
\cref{Theorem:1.7} with $\alpha=-\infty$.

\item 
The non-existence result for positive solutions in \cref{Theorem:1.7} 
can be extended to possibly sign-changing solutions.
More precisely, let $g$ be a continuous function on $\R$ and define
$\rho$ by $\rho(s) = {\frac{H(s)}{s^2/2}}$ also for $s<0$.
If we assume that in addition to \ref{(rho1)}
    \begin{equation*}
        s\mapsto \rho(s) \ \text{is non-decreasing in $(-\infty,0)$
        and strictly increasing in $[-s_0,0)$,}
    \end{equation*}
then there is no solution of \eqref{1.1} with $m=m_1$. 
A proof of this fact can be given in a similar way to the proof of
\cref{Theorem:1.7}.
\end{enumerate}
\end{Remark}

%
\subsection{Variational setting and strategy}
\label{Section:1.3}
%

To prove \cref{Theorem:1.1}, first we take an odd extension of $g$, that is, 
we extend $g(s)$ for $s<0$ by
    $$  g(s) = -g(-s)
    $$
and make use of variational methods.  As we recalled in \cref{Section:1.1},
there are two relevant variational approaches. 

One approach is to find critical points of $\calI:\, \calS_m\to\R$, where $\calI$ and
$\calS_m$ are defined in \eqref{1.6}--\eqref{1.7}.
In particular, one can consider the minimizing problem 
    \begin{equation}\label{1.20}
        d(m) = \inf \Set{ \calI(u) | u \in \calS_m}
    \end{equation}
and in this case $\mu$ is found as a Lagrange multiplier.
Another approach is the Lagrangian one, which is introduced in Hirata and Tanaka \cite{HT0}.
For $m=m_1$ it can be stated as an approach to find critical points of the following functional
    \begin{equation*}
        I(\lambda,u) = \intRN \frac{1}{2} \abs{\nabla u}^2 - G(u) \, dx 
            + e^\lambda \left( \frac{1}{2} \norm{u}_{2}^2 - m_1 \right) 
        \in C^1 \left( \R \times H^1_r(\R^N) , \R \right);
    \end{equation*}
here $H^1_r(\R^N)$ stands for the subspace of radially symmetric functions.
We note that the Lagrangian approach have been 
also successfully applied to problems with nonlocal 
nonlinearities or quasilinear Schr\"odinger equations, see for instance
 \cite{CGT1,CGT2,CGT4,CGT5,CT1,YaTaCh21}.

Our $L^2$-critical problem \eqref{1.1} is a delicate problem; when we study \eqref{1.1}
through the minimizing problem \eqref{1.20}, the infinimum $d(m)$ is not a continuous
function of $m$ in general.  Indeed, for $g(s)=\abs s^{p-1}s$ we have
    \begin{equation} \label{1.21}
        d(m) = \begin{cases}
                0   &\text{for}\ m\in (0,m_1], \\
                -\infty   &\text{for}\ m\in (m_1,\infty).
        \end{cases}
    \end{equation}
See the end of \cref{Section:2.2} for a proof. 
We also note that $d(m)=0$ is not attained for $m\in (0, m_1)$.
On contrary, for $L^2$-subcritical and $L^2$-supercritical problems, $d(m)$ (or a suitable
minimax value for $\calI$) depends on $m$ continuously and one can find critical points
for $m$ in some open intervals of $\R$.

In this paper, we take the Lagrangian approach to our $L^2$-critical problem.
It gives a natural way to see the fine topological properties of the problem \eqref{1.1} through
the following mountain pass minimax values; for each $\lambda\in\R$ the restricted
functional $u\mapsto I(\lambda,u);\, H_r^1(\R^N)\to\R$ satisfies the conditions of
Berestycki and Lions \cite{BeL,BeGaKa83}.  Thus it has a mountain pass geometry
(c.f. \cite{JT0}) and it enables us to define the mountain pass value $b(\lambda)$ 
for each $\lambda\in\R$ (see \cref{Section:3.1}). 
Under \ref{(g1)}, we will observe that $\lambda\mapsto b(\lambda)$ is continuous and satisfies
    \begin{equation}\label{1.22}
        b(\lambda)\to 0 \quad \text{as}\ \lambda\to\pm\infty.
    \end{equation}
Thus we expect 
$\inf_{\lambda\in\R} b(\lambda)$ or $\sup_{\lambda\in\R} b(\lambda)$
to be critical values of $I$. 
However it is difficult to see such a property directly, since the differentiability of 
$\lambda\mapsto b(\lambda)$ is not clear.  Thus for our purpose, we introduce two minimax values 
$\ub$ and $\ob$ 
inspired by Bahri and Li \cite{BaL}, Bahri and Lions \cite{BaLio97} and Tanaka \cite{Tan0}
(as a matter of facts, in \cref{Section:7}
we will show that $\inf_{\lambda\in\R} b(\lambda)$ is a critical value 
if $\inf_{\lambda\in\R} b(\lambda)<0$).
In what follows, we use notation:
    \begin{equation*}
        \calS_0 = \left\{ u \in H^1_r(\R^N) \ \Big| \ 
            \half\intRN 
            u^2\,dx = m_1 \right\}.
    \end{equation*}
First we observe that $I$ has a mountain pass type geometry in the product
space $\R\times H_r^1(\R^N)$:
\begin{itemize}
\item[(1)] Let $\R\times\calS_0
=\{ (\lambda,u) \mid \half\intRN \abs u^2\, dx =m_1\}$ 
be a \emph{cylinder} in $\R\times H_r^1(\R^N)$.  Then $I$ is bounded from below on 
$\R\times\calS_0$.
\item[(2)] The functional $I$ is unbounded from below inside 
$\R\times \{ u  \mid \half \norm{u}_2^2 < m_1\}$ and outside 
$\R \times \{u  \mid \half \norm{u}_2^2 > m_1\}$
of the cylinder $\R\times\calS_0$. 
Actually, we have
    \begin{equation*}
            \lim_{\lambda\to\infty} I(\lambda,0) =-\infty, \quad
        \lim_{t\to\infty} I(\lambda_0,t\varphi_0)=-\infty,
    \end{equation*}
for any $\lambda_0\in\R$ and $\varphi_0\in H_r^1(\R^N)\setminus \{ 0\}$.
\end{itemize}
We introduce two minimax values $\ub$ and $\ob$ corresponding to this geometry. 
We give here an idea of their definition:
    \begin{equation*}
        \ub = \inf_{\gamma \in \uGamma} \max_{0 \leq t \leq 1} I(\gamma(t)), \qquad 
        \ob = \inf_{\gamma \in \oGamma} \sup_{ (t,\lambda) \in [0,1] \times \R } 
        I(\gamma(t, \lambda)),
    \end{equation*}
where 
    \begin{align*}
            &\uGamma = \Set{ \gamma \in C([0,1] , \R \times H^1_r(\R^N)) | 
            \gamma(0)\in \R\times\{ 0\},\ I(\gamma(0)) \ll -1,\ 
            \gamma(1) \in \gamma_0(\R\times \{1\}) },       
            \\
            &\oGamma = 
            \Set{ \gamma \in C( [0,1] \times \R , \R \times H^1_r(\R^N) ) | 
            \gamma(t,\lambda) = \gamma_0(t,\lambda) \ 
            \text{if $t\sim 0$ or $t\sim 1$ or $\abs{\lambda}\gg 1$}}. 
    \end{align*}
Here $\gamma_0\in C([0,1]\times\R,\R\times H_r^1(\R^N))$ 
is of a form $\gamma_0(t,\lambda)=(\lambda,t\zeta_0(\lambda))$ with
$\zeta_0\in C(\R, H_r^1(\R^N))$ and satisfies
\begin{itemize}
\item[(1)] $\gamma_0(0,\lambda)=(\lambda,0)$, 
$I(\gamma_0(0,\lambda))=I(\lambda,0)\leq 0$ 
for all $\lambda\in\R$;
\item[(2)] $\gamma_0(1,\lambda)=(\lambda,\zeta_0(\lambda))\in 
\R \times \{u  \mid \half \norm{u}_2^2 > m_1 \}$ 
and $I(\gamma_0(1,\lambda))\ll 0$ for all $\lambda\in\R$;
\item[(3)] $\max_{t\in [0,1]} I(\gamma_0(t,\lambda))\to 0$ 
as $\lambda\to\pm\infty$.
\end{itemize}
Notice that $\uGamma$ is a class of paths joining two points 
$\gamma(0)$ and $\gamma(1)$, 
which are inside and
outside the cylinder $\R\times\calS_0$ with $I(\gamma(0))$, $I(\gamma(1))\ll 0$.  
On the other hand, $\oGamma$ is a class of
2-dimensional paths at whose boundary (including $\pm\infty$) $I$ is non-positive and 
$\gamma([0,1]\times\R)$ links with $\{\lambda\}\times \calS_0$ for each $\lambda\in\R$.
See \cref{Section:3} for precise definitions.

For $b(\cdot)$, $\ub$ and $\ob$, in \cref{Proposition:4.2} and \cref{Corollary:7.5} 
we will prove 
    \begin{equation}\label{1.23} 
        \ub \leq b(\lambda) \leq \ob \quad \text{for all}\ \lambda\in\R
    \end{equation}
and
    \begin{equation}\label{1.24}
        \ub = \inf_{\lambda\in\R} b(\lambda). 
    \end{equation}
Combining \eqref{1.23} with \eqref{1.22}, we have $ \ub \leq 0 \leq \ob$ and 
there are three cases to consider: 
    \[  \hbox{(1)\ $\ub < 0$}, \qquad \hbox{(2)\ $0 < \ob$}, \qquad 
        \hbox{(3)\ $\ub = 0 = \ob$}. 
    \]
For case (1) (resp. case (2)), we prove that $\ub$ (resp. $\ob$) is a critical 
value of $I$; notice that, if $\ub<0<\ob$, actually existence of two positive 
solutions occur. 
On the other hand, for case (3), being $b(\lambda)=0$ for all $\lambda \in \R$, 
we will show that 
any mountain pass solution to \eqref{1.13} 
with $\mu=e^\lambda$ satisfies \eqref{1.1} with $m=m_1$. In particular, 
\eqref{1.1} admits uncountably many solutions 
(see \cref{Proposition:4.6}).

\begin{Remark}
A typical example, for which 
(3) takes place, is $g(s)=\abs s^{p-1}s$, 
in which case all these solutions are related by scaling.
We conjecture that this is the only possible case of $g$ for which (3) holds. 
This seems to be an interesting open problem.
\end{Remark}

\medskip

To run the above scheme, we need to introduce suitable compactness conditions. 
In particular, we introduce
\emph{Palais-Smale-Pohozaev-Cerami sequences} 
and the \emph{Palais-Smale-Pohozaev-Cerami condition} ($(PSPC)$ 
sequences and $(PSPC)$ condition in short) 
which are inspired by the works \cite{Ce0,BarBenF,HT0}. 
Here $(\lambda_j,u_j)_{j=1}^\infty$ is called a $(PSPC)_b$ sequence if 
    \begin{equation*}
        I(\lambda_j,u_j) \to b, \quad (1+ \norm{u}_{H^1} ) 
        \norm{ \partial_u I (\lambda_j, u_j) }_{(H^1)^\ast} \to 0, \quad 
        \partial_\lambda I(\lambda_j,u_j) \to 0, \quad P(\lambda_j,u_j) \to 0,
    \end{equation*}
where $P(\lambda,u)$ is defined in \eqref{2.3} and is related to the Pohozaev identity. 
The functional $I$ is said to satisfy the $(PSPC)_b$ condition 
provided any $(PSPC)_b$ sequence 
contains a strongly convergent subsequence. 
To find a \emph{positive} solution under \ref{(g2)}, 
we introduce another version $(PSPC)_b^*$ of the $(PSPC)$ 
condition in \cref{Definition:5.3}.

A notion related to $(PSPC)_b$ has been introduced in \cite{HT0}; namely, 
Palais-Smale-Pohozaev sequences and the 
Palais-Smale-Pohozaev condition are introduced by replacing \\
$(1+\norm{u_j}_{H^1})\norm{\partial_u I(\lambda_j,u_j)}_{(H^1)^*}\to 0$ by
$\norm{\partial_u I(\lambda_j,u_j)}_{(H^1)^*}\to 0$.  See also \cite{IY} for the
Palais-Smale-Cerami sequences and condition.
We remark that our compactness 
condition is weaker than the one 
in \cite{HT0} and
different from the one in \cite{IY}.

In our argument a verification of the $(PSPC)_b$ and $(PSPC)_b^*$ conditions and
the corresponding generation of deformation flows 
are important.
We remark that these $(PSPC)$ type conditions fail 
when $b=0$: 
in fact, when $g(s)=\abs s^{p-1}s$,
the set of solutions of \eqref{1.1} with $m=m_1$ is not compact as we saw in the
beginning of the introduction 
(see also \cref{Remark:5.5}).

For the verification of such conditions 
 for $b\not=0$, the
study of the behavior of $(PSPC)_b$ and $(PSPC)_b^*$ sequences 
$(\lambda_j,u_j)_{j=1}^\infty$ is important and the case $\lambda_j\to-\infty$ 
is the most delicate.
In this case, a solution $u\in F$ of \eqref{1.15}
with 
    \begin{equation} \label{1.25}
        \half\norm{\nabla u}_2^2 -\intRN G(u)\,dx=b     
    \end{equation}
appears as a limit profile of such sequences. 
To run into a contradiction, we thus need to show that no solution to 
\eqref{1.15} with \eqref{1.25} exists for $b\neq 0$. While this is always the case 
for $b<0$ (see \cref{Lemma:2.5}), when $b>0$ we need to take advantage of 
an additional assumption.  
Since \ref{(g2)}
only ensures the non-existence of non-negative solutions, it seems difficult to show the
$(PSPC)_b$ condition, 
thus we need to introduce the $(PSPC)_b^*$ condition, which eventually applies also 
for $b>0$ under \ref{(g2)}.

After checking the compactness condition, 
we develop deformation
arguments.
When $\ub<0$, we find a deformation flow in $\RE$ under the $(PSPC)_{\ub}$ condition, 
which enables us to find a solution of \eqref{1.1}.  
See \cref{Section:7}.

When $\ob>0$, we develop a new deformation argument under the $(PSPC)_{\ob}^*$ 
condition. 
First we  find a deformation flow in a small neighborhood of the set of
non-negative functions (not in the whole $\R\times H_r^1(\R^N)$).  
However such a local deformation is not 
enough to show the existence of a critical point under the Cerami type condition 
(c.f. \cref{Remark:8.2}).
In \cref{Section:8}, we develop an iteration argument involving the absolute 
operator $(\lambda,u(\cdot))\mapsto (\lambda,\abs{u(\cdot)})$ 
to provide a new flow which 
enables us to show the existence of a positive solution.

\begin{Remark}\label{Remark:1.10}
In order to find a positive solution, the truncation argument is one of the
standard approaches; that is, we introduce a truncation of $g$ (equals to 
zero in $(-\infty,0)$) and try to find a solution of \eqref{1.1} with
$m=m_1$.
However this approach is not straightforward.  Indeed in our argument
a priori estimates in $L^{p+1}$ for $(PSPC)_b$-sequences play a crucial
role in the verification of the $(PSPC)$-condition.  But it seems difficult
for the truncated functional.  See \cref{Remark:5.10}.

In this paper, we develop a new approach using an odd extension of $g$, in 
which we introduce the new $(PSPC)_b^*$-condition and the corresponding deformation
argument.  With our new approach we obtain necessary estimates in $L^{p+1}$.
Since our iteration procedure contains jumps due to the absolute operator, our 
deformation flow looks different from the standard one.  However
our flow enjoys the standard properties of flows. For example, our procedure
gives a continuous deformation of the initial paths.  See \cref{Remark:8.7}.
\end{Remark}

\medskip

\textbf{Outline of the paper.} The paper is organized as follows. In \cref{Section:2} 
we introduce the functional 
setting and recall some known results, together with some preliminary properties. 
In \cref{Section:3,Section:4}, 
we study the topological structure of the functional $I$ and introduce 
two minimax values
$\ub$ and $\ob$, and their relation with $b(\lambda)$. 
In \cref{Section:5} we discuss the compactness of the problem through 
the $(PSPC)$ condition.  
In \cref{Section:6,Section:7,Section:8}, we give deformation arguments and
prove \cref{Theorem:1.1}; we highlight that \cref{Section:6} is showed
 in a much more general setting and it could be applied in other frameworks. 
At the end of the paper, in \cref{Section:9} we show that the sublinearity 
condition \ref{(g1)} is essential, through some counterexamples and proving 
\cref{Theorem:1.7}. 
Finally, in \cref{Appendix:A} we collect some technical results.

\section{Preliminaries} \label{Section:2}

Throughout this paper, $g$ is extended as an odd function on $\R$ and we study 
the existence of a positive solution of \eqref{1.1} with $m=m_1$.  
That is, we assume
\begin{enumerate}[label={\rm (g\arabic*')}]
\setcounter{enumi}{-1}
\item \label{(g0')} $g\in C(\R,\R)$ satisfies $g(-s)=-g(s)$ for all $s\in\R$.
\end{enumerate}
We also assume \ref{(g1)} throughout this paper except for \cref{Section:9}.

We use notation:
    \begin{equation*}
        B_R= \Set{ x\in\R^N | \abs x\leq R} \quad \text{for}\ R>0.
    \end{equation*}

\subsection{Variational formulation and function spaces} \label{Section:2.1}
We look for radially symmetric solutions of \eqref{1.1} and work in the
space of radially symmetric functions:
    \begin{equation*}
        E=H_r^1(\R^N)= \Set{ u\in H^1(\R^N) |u(x)=u(\abs x) }.
    \end{equation*}
We use the following notation for $u\in E$
    \begin{equation*}
        \begin{aligned}
        &\norm u_p = \left( \intRN \abs{u(x)}^p\, dx\right)^{1/p} \quad 
            \text{for}\ p\in [1,\infty),\\
        &\norm u_\infty = \esssup_{x\in\R^N} \abs{u(x)}, \\
        &\norm u_E = (\norm{\nabla u}_2^2 +\norm u_2^2)^{1/2}.
        \end{aligned}
    \end{equation*}
We also write for $u$, $v\in L^2(\R^N)$
    \begin{equation*}
        (u,v)_2 = \intRN uv\, dx.
    \end{equation*}

By \ref{(g0')} and \ref{(g1)} there exists a constant $C_1>0$ such that
    \begin{equation*}
        \abs{G(s)} \leq C_1\abs s^{p+1}, \quad 
        \abs{s g(s)} \leq C_1 \abs s^{p+1} \quad \text{for all}\ s\in\R,
    \end{equation*}
where $\displaystyle G(s)=\int_0^s g(\tau)\,d\tau$.
We also set $\displaystyle H(s)=\int_0^s h(\tau)\, d\tau$. 
From \eqref{1.14}, there exists $A>0$ such that
    \begin{equation} \label{2.1}
        \abs{h(s)} \leq A\abs s, \quad 
        \abs{H(s)} \leq As^2 \quad \text{for all}\ s\in\R.      
    \end{equation}
These inequalities will be used repeatedly in this paper. 

We take a variational approach to \eqref{1.1} with $m=m_1$ and will find critical 
points of the following functional:
    \begin{equation} \label{2.2}
        I(\lambda,u) =\half\norm{\nabla u}_2^2 -\intRN G(u)\, dx
            +e^\lambda \left(\half \norm u_2^2-m_1\right)
            \in C^1(\RE,\R).        
    \end{equation}
Note that 
    \begin{align*}  
        &\partial_\lambda I(\lambda,u) = e^\lambda\left(\half \norm u_2^2-m_1\right), \\
        &\partial_u I(\lambda,u)\varphi 
            = (\nabla u,\nabla\varphi)_2 +e^\lambda(u,\varphi)_2
             -\intRN g(u)\varphi\, dx \quad \text{for all}\ \varphi\in E.
    \end{align*}
Thus $(\lambda,u)\in\RE$ is a critical point of $I$ if and only if
$(\mu,u)$ with $\mu=e^\lambda$ is a solution of \eqref{1.1}.
We also introduce 
    \begin{equation} \label{2.3}
        P(\lambda,u)={\frac{N-2}{2}}\norm u_2^2 
            + N\left\{ {\frac{e^\lambda}{2}}\norm u_2^2-\intRN G(u)\, dx\right\}
        \in C^1(\RE,\R),            
    \end{equation}
which corresponds to the Pohozaev identity. 
We note that solutions of \eqref{1.1} with $\mu=e^\lambda$ satisfy the Pohozaev identity 
$P(\lambda,u)=0$ (see \cite{BeL}), that is,
    \begin{equation*}
        P(\lambda,u)=0 \quad \text{if}\ \partial_u I(\lambda,u)=0.  
    \end{equation*}
Thus, if $(\mu,u)$ with $\mu=e^\lambda$ is a solution of \eqref{1.1}, then 
    \begin{equation*}
        \half\norm{\nabla u}_2^2 +{\frac{N}{2}}\intRN G(u)-\half g(u)u\, dx 
        ={\frac{N}{4}}\partial_u I(\lambda,u)u-\half P(\lambda,u)=0,
    \end{equation*}
which can be rewritten as 
    \begin{equation} \label{2.4}
        \half\norm{\nabla u}_2^2 -{\frac{1}{p+1}}\norm u_{p+1}^{p+1}
        +{\frac{N}{2}}\intRN H(u)-\half h(u)u\, dx =0.  
    \end{equation}

\subsection{Least energy solutions of $-\Delta\omega+\mu\omega=\abs\omega^{p-1}\omega$} 
\label{Section:2.2}
We study a class of nonlinearities which can be regarded as a perturbation of $\abs u^{p-1}u$.
In our arguments, properties of solutions of the following unconstrained problem 
are important:
    \begin{equation} \label{2.5}
        \left\{ \begin{aligned}
        -\Delta u + \mu u &=\abs u^{p-1}u \quad \text{in}\ \R^N, \\
        &u\in E, 
        \end{aligned}
        \right. 
    \end{equation}
where $\mu>0$ is fixed.

It is well-known (see \cite{Kwo1}) that \eqref{2.5} has a unique positive solution, 
which has the least energy among all non-trivial solutions $u\in H^1(\R^N)$. 
Recalling the notations in \eqref{1.9}, $\omega_\mu$ is the unique positive solution 
of \eqref{2.5} and the set of least energy solutions of \eqref{2.5} is given by 
$\{ \pm \omega_\mu\}$. 
We consider the case $\mu=1$ and recall that 
    \begin{equation} \label{2.6}
        m_1 = \half \norm{\omega_1}_2^2.        
    \end{equation}
It follows from \eqref{2.5} with $\mu=1$ that
    \begin{equation} \label{2.7}
        \norm{\nabla\omega_1}_2^2 +\norm{\omega_1}_2^2 =\norm{\omega_1}_{p+1}^{p+1}.
    \end{equation}
By \eqref{2.4} with $h \equiv 0$ and \eqref{2.7}, we have
    \begin{equation*}
        \norm{\nabla\omega_1}_2^2 =Nm_1, \quad \norm{\omega_1}_{p+1}^{p+1}=(N+2)m_1,
        \quad \Psi_{01}(\omega_1)=m_1.
    \end{equation*}
Here $\Psi_{0\mu}$ is a functional corresponding to \eqref{2.5}:
    \begin{equation} \label{2.8}
        \Psi_{0\mu}(u) = \half\norm{\nabla u}_2^2 +{\frac{\mu}{2}}\norm u_2^2 
        -{\frac{1}{p+1}}\norm u_{p+1}^{p+1}:\, E\to\R.      
    \end{equation}
By a simple scaling, $\omega_\mu$ is expressed as 
    \begin{equation*}
            \omega_\mu(x) =\mu^{N/4}\omega_1(\mu^{1/2}x).
    \end{equation*}
We see 
    \begin{equation}\label{2.9}
    \begin{aligned}     
        &\norm{\omega_\mu}_2^2=2m_1, \quad \norm{\nabla \omega_\mu}_2^2=\mu N m_1, \quad
        \norm{\omega_\mu}_{p+1}^{p+1}=\mu (N+2)m_1, \\
        &\Psi_{0\mu}(\omega_\mu)=\mu m_1, \quad \Psi_{0\mu}'(\omega_\mu)=0.
    \end{aligned}
    \end{equation}
We note that $\mu\mapsto \omega_\mu;\, (0,\infty)\to E$ is of class $C^2$ since 
$\omega_1\in C^\infty(\R^N,\R)$ and $\omega_1(x)$ decays exponentially as 
$\abs x\to \infty$.
We also note that $\omega_\mu$ enjoys a mountain pass minimax characterization. 

Finally, we note that $m_1$ defined in \eqref{2.6} can be used to describe the best 
constant in the Gagliardo-Nirenberg inequality.  The following proposition is due to 
Weinstein \cite{Wei1} (see also Cazenave \cite{Ca0}).

\begin{Proposition}[{\cite{Wei1,Ca0}}] \label{Proposition:2.1}
    \begin{equation} \label{2.10}
        \norm v_{p+1}^{p+1} \leq \half(p+1)(2m_1)^{-2/N}\norm{\nabla v}_2^2\norm v_2^{4/N}
        \quad \text{for all}\ v\in H^1(\R^N). 
    \end{equation}
\end{Proposition}

As a corollary to \cref{Proposition:2.1}, we have the following result, which we will 
use repeatedly in this paper.

\begin{Corollary} \label{Corollary:2.2}
    \begin{equation*}
        \half\norm{\nabla u}_2^2 -{\frac{1}{p+1}}\norm u_{p+1}^{p+1} \geq 0
        \quad \text{for all}\ u\in E\ \text{with}\ \frac{1}{2}\norm u_2^2 \leq m_1.
    \end{equation*}
\end{Corollary}

\begin{proof}
By \cref{Proposition:2.1}, we have for $u\in E$
    \begin{equation*}
        \half\norm{\nabla u}_2^2 -{\frac{1}{p+1}}\norm u_{p+1}^{p+1}
        \geq \half\left(1-(2m_1)^{-2/N}\norm u_2^{4/N}\right)\norm{\nabla u}_2^2.
    \end{equation*}
Thus the conclusion of \cref{Corollary:2.2} holds. 
\end{proof}

At the end of this section, we give a proof of \eqref{1.21} for $g(s)=\abs s^{p-1}s$.

\begin{proof}[Proof of \eqref{1.21} for $g(s)=\abs s^{p-1}s$]
We note that $s\omega_\mu\in \calS_{s^2m_1}$ for all $s$, $\mu\in (0,\infty)$ and
    \[  \calI(s\omega_\mu) = {\frac{N}{2}}m_1\mu (s^2-s^{p+1}),
    \]
which follows from \eqref{2.9}.  Together with \cref{Corollary:2.2}, we deduce 
\eqref{1.21}.
\end{proof}

\subsection{Remarks on the zero mass case} \label{Section:2.3}
In this section, we prepare some notations related to \eqref{1.15} when \ref{(g0')} and
\ref{(g1)} hold. 
Equation \eqref{1.15} is important to study the situation $0 < \mu \ll 1$ and 
we first define a functional corresponding to \eqref{1.15}: 
    \begin{equation} \label{2.11}
        Z(u) = \half\norm{\nabla u}_2^2 -\intRN G(u)\, dx:\,
        F\to \R,        
    \end{equation}
where $F$ and its norm $\norm\cdot_F$ are defined as in \eqref{1.16}--\eqref{1.17}.

We first collect basic properties of $F$:

\begin{Lemma}\label{Lemma:2.3}
The following hold: 
\begin{enumerate}[label={\rm (\roman*)}]
\item $C^\infty_{0,r} (\R^N)$ is dense in $F$ and $E \subset F$. 
\item For any $q\in [p+1,2^*]$ when $N\geq 3$, for any $q\in [p+1,\infty)$ 
when $N=2$, there exists $C_q>0$ such that 
    \[  \norm u_q \leq C_q \norm u_F \quad \text{for any $u \in F$}. 
    \]
In addition, the embeddings $F \subset L^q(B_R)$ and $F \subset L^r(\R^N)$ 
are compact for any $q \in [1,2^\ast)$, $R>0$ and $r \in (p+1,2^\ast)$. 
\end{enumerate}
\end{Lemma}

\cref{Lemma:2.3} (i) is easily observed and we omit a proof.
To show \cref{Lemma:2.3} (ii), we need the following result.

\begin{Lemma} \label{Lemma:2.4} 
\begin{enumerate}[label={\rm (\roman*)}]
\item {\rm (Berestycki-Lions \cite[Radial Lemma A.III]{BeL}).}
When $N\geq 3$, there exists a constant $C_N>0$ such that for $u\in \calD_r^1(\R^N)$
    \begin{equation*}
        \abs{u(x)} \leq C_N\norm{\nabla u}_2 \abs x^{-{\frac{N-2}{2}}} \quad
            \text{for}\ \abs x\geq 1.
    \end{equation*}
\item When $N=2$, there exists a constant $C_2>0$ such that for $u\in F$
    \begin{equation} \label{2.12}
        \abs{u(x)}^3 \leq C_2\norm{\nabla u}_2\norm u_4^2 \abs x^{-1}
        \quad \text{for all} \ x\in\R^2\setminus\{ 0\}.
    \end{equation}
\end{enumerate}
\end{Lemma}

\begin{proof}
We give a proof to (ii). 
Thanks to \cref{Lemma:2.3} (i), it is enough to show \eqref{2.12} for 
$u\in C_{0,r}^\infty(\R^2)$.
We write $r=\abs x$ and compute 
    \begin{equation*}
        \begin{aligned}
        {\frac{d}{dr}}(r\abs{u(r)}^3)
        = \abs{u(r)}^3 + 3r\abs{u(r)}^2 u_r(r)\sgn(u(r)) 
        \geq -3r\abs{u(r)}^2 \abs{u_r(r)}. 
        \end{aligned}
    \end{equation*}
Integrating this inequality over $[r,\infty)$ yields 
    \begin{equation*}
        \begin{aligned}
        r\abs{u(r)}^3 
        = -\int_r^\infty {\frac{d}{ds}}(s\abs{u(s)}^3)\, ds
        \leq 3\int_r^\infty s\abs{u(s)}^2 \abs{u_r(s)}\, ds 
        \leq {\frac{3}{2\pi}}\norm u_4^2 \norm{\nabla u}_2. 
        \end{aligned}
    \end{equation*}
Thus we get \eqref{2.12}. 
\end{proof}

\begin{proof}[Proof of \cref{Lemma:2.3} (ii)]
Let $q \in [p+1, 2^\ast]$ when $N \geq 3$ and $r \in [4,\infty)$ when $N =2$. 
By the Gagliardo-Nirenberg inequality, 
    \[  \norm u_q \leq C \norm{\nabla u}_2^{\theta} \norm u_{p+1}^{1-\theta} 
        \quad \text{for any $u \in H^1(\R^N)$, where } 
        \frac{1}{q} = \theta \left( \frac{1}{2} - \frac{1}{N} \right) 
        + \frac{1-\theta}{p+1}. 
    \]
The above inequality holds for any $u \in F$ and we obtain 
$\norm{ u }_q \leq C \norm{ u }_F$. 

Since $F \subset H^1(B_R)$ holds for any $R>0$, the embedding $F \subset L^q(B_R)$ 
is compact for every $R>0$ and $q \in [1,2^\ast)$. 
Next we consider the embedding $F\subset L^r(\R^N)$.
Let $r\in(p+1, 2^*)$ and define $\theta\in (0,1)$ by 
${\frac{1}{r}}={\frac{\theta}{\infty}}+{\frac{1-\theta}{p+1}}$.
For $u\in F$ and $R\geq 1$, we have
    \begin{equation*}
        \norm u_{L^r(\R^N\setminus R_R)} 
            \leq \norm u_{L^\infty(\R^N\setminus B_R)}^\theta
            \norm u_{L^{p+1}(\R^N\setminus B_R)}^{1-\theta} 
        \leq \begin{cases}  
            C_N^\theta R^{-{\frac{N-2}{2}}\theta}\norm u_F          
                                                    &\text{for}\ N\geq 3,\\
            C_2^{\frac{\theta}{3}}R^{-{\frac{\theta}{3}}} \norm u_F 
                                                    &\text{for}\ N=2.
        \end{cases} 
    \end{equation*}
Here we used \cref{Lemma:2.4}.  
Since the coefficient $C_N^\theta R^{-{\frac{N-2}{2}}\theta}$
or $C_2^{\frac{\theta}{3}}R^{-{\frac{\theta}{3}}}$ tends to $0$ as $R\to\infty$,
we deduce the compactness of the embedding $F\subset L^r(\R^N)$ from the
the compactness of $F\subset L^r(B_R)$.
\end{proof}

By \cref{Lemma:2.3}, under \ref{(g0')} and \ref{(g1)}, we may check that 
$Z \in C^1(F,\R)$ and 
    \[  Z'(u) \varphi 
        = \left(\nabla u , \nabla \varphi \right)_2 - \intRN g(u) \varphi \, dx 
            \quad \text{for any $\varphi \in F$}. 
    \]
In particular, any solution of \eqref{1.15} in $F$ can be characterized as critical 
point of $Z$.  Moreover, we have

\begin{Lemma} \label{Lemma:2.5}
    Assume {\rm \ref{(g0')} and \ref{(g1)}} and let $u_0 \in F$ satisfy $Z'(u_0) = 0$. 
Then $u_0$ satisfies 
    \begin{align}   
    &\norm{\nabla u_0}_2^2 -\intRN g(u_0)u_0\, dx =0,               \label{2.13}\\
    &{\frac{N-2}{2}}\norm{\nabla u_0}_2^2 -N\intRN G(u_0)\, dx=0,   \label{2.14}\\
    &\intRN NG(u_0)-{\frac{N-2}{2}}g(u_0)u_0\, dx =0,               \label{2.15}\\
    &Z(u_0)\geq 0.                                                  \label{2.16}
    \end{align}
\end{Lemma}

\begin{proof}
Since $Z'(u_0)u_0=0$, we have \eqref{2.13}.  
Since $F \subset H^1_{\rm loc} (\R^N)$, \ref{(g0')}, \ref{(g1)} and elliptic regularity 
yield $u_0 \in W^{2,q}_{\rm loc} (\R^N)$ 
for any $q < \infty$. Following the argument in \cite[Proposition 1]{BeL}, 
we may show that $u_0$ satisfies the Pohozaev identity \eqref{2.14}. 
We remark that in \cite{BeL}, the Pohozaev identity \eqref{2.14} is shown
for $N\geq 3$ but their argument works also for $N=2$.
Together with \eqref{2.13}, we have \eqref{2.15}.

\eqref{2.16} follows from the Pohozaev identity \eqref{2.14} and direct computations: 
    \[  Z(u_0) = Z(u_0) - \frac{1}{N} \left( \frac{N-2}{2} \norm{\nabla u_0}_2^2 
        - N \intRN G(u_0) \, dx \right) = \frac{1}{N} \norm{ \nabla u_0 }_2^2 \geq 0. 
        \qedhere
    \]
\end{proof}

At the end of this section, we show \cref{Proposition:1.2} (ii), that is,  
\ref{(g3)} ensures \ref{(g2)}.

\begin{proof}[Proof of \cref{Proposition:1.2} (ii)]
First we show that \ref{(g3)} implies \ref{(g2)}.  
Suppose $u_0\in F$ is a solution of \eqref{1.15}. By \cref{Lemma:2.5}, $u_0$ satisfies
\eqref{2.15}.
By \ref{(g0')}, \ref{(g1)} and \ref{(g3)}, there exist $\delta_0>0$ and $c_0>0$ such that
    \begin{align}   
    &NG(s)- {\frac{N-2}{2}}g(s)s \geq c_0\abs s^{p+1} 
            \quad \text{for}\ \abs s\leq \delta_0,          \label{2.17}\\
    &NG(s)- {\frac{N-2}{2}}g(s)s \geq 0 \quad \text{for all }\ s\in\R.
                                                            \label{2.18}
    \end{align}
Since $u\in F$ satisfies $u(r)\to 0$ as $r\to \infty$ thanks to \cref{Lemma:2.4}, 
we have by \eqref{2.17} and \eqref{2.18}
    \begin{equation*}
        \intRN NG(u)-{\frac{N-2}{2}}g(u)u\, dx>0 
        \quad \text{for}\ u\in F\setminus\{ 0\}.    
    \end{equation*}
Thus \eqref{2.15} implies $u_0\equiv 0$. 
\end{proof}

Proof of \cref{Proposition:1.2} (i) will be given in the following subsection.

\subsection{Non-existence of positive solutions via a Liouville type result} \label{Section:2.4}


In this section we consider zero mass problem in a slightly general situation.
We believe that non-existence result in general situation is of interest and
it makes clear why we have restriction $N=2,3,4$ in \cref{Proposition:1.2}.

We consider the zero-mass problem \eqref{1.15} under the following condition.

\begin{enumerate}[label={\rm (A)}]
\item\label{(A)} There exists $q\in (1,2^*-1)$ such that
	$$	\lim_{s\to +0} \frac{g(s)}{s^q} =1.
    $$
\end{enumerate}

\noindent
Under \ref{(A)}, the natural space to consider the zero-mass problem \eqref{1.15} is 
    \begin{equation} \label{1.16a}
        \widetilde F= \big\{u \in L^{q+1} (\R^N)  \mid \nabla u \in L^2(\R^{N}), \
        u(x) = u(\abs{x})\big\},
    \end{equation}
which is equipped with the norm $\norm u_{\widetilde F} = \norm{\nabla u}_2 + \norm u_{q+1}$.

We have the following non-existence result.

\begin{Theorem}\label{Theorem:L}
Let $N\geq 2$ and suppose that \ref{(g0)} and \ref{(A)} hold.  Moreover assume $q\in (1,\frac{N}{N-2}]$ 
when $N\geq 3$.  Then the zero-mass problem \eqref{1.15} has no positive solutions in $\widetilde F$.
\end{Theorem}

Our \cref{Proposition:1.2}(i) can be obtained as a special case $q=p=1+\frac{4}{N}$ of
\cref{Theorem:L}.
We note that $F=\widetilde F$ when $q=p$ and $p\leq \frac{N}{N-2}$ holds if and only if $N\leq 4$.

To show \cref{Theorem:L} we recall the result of Armstrong and Sirakov \cite{AS} (c.f. \cite{AGQ16}).  
In \cite{AS}, for a positive function $f(s)\in C((0,\infty),\R)$,  they consider 
positive supersolutions of
	\begin{equation}\label{AS}
		-\Delta u \geq f(u).
	\end{equation}
They show that for any $R>0$ \eqref{AS} does not have positive supersolutions in the exterior
domain $\R^N\setminus B_R$ provided that
	\begin{equation} \label{AScond}
		\liminf_{s\to +0} \frac{f(s)}{s^{\frac{N}{N-2}}}>0 \quad \text{when}\ N\geq 3; \qquad
		\liminf_{s\to \infty} e^{as}f(s) >0 \ \text{for any $a>0$} \quad \text{when}\ N=2.
	\end{equation}

\begin{proof}[Proof of \cref{Theorem:L}]
By the assumption \ref{(A)}, there exists $\delta>0$ such that $g(s)>0$ for $s\in (0,\delta]$.
We set $\overline g:\, (0,\infty)\to\R$ by
	\begin{equation*}
		\overline g(s)=\begin{cases}
			g(s) &\text{for}\ s\in (0,\delta], \\
			g(\delta) &\text{for}\ s\in (\delta,\infty).
			\end{cases}
	\end{equation*}
Clearly, $\overline g(s)$ satisfies the condition \eqref{AScond} for $q\leq {N\over N-2}$ ($N\geq 3$)
and $q\in(1,\infty)$ for $N=2$.  Thus
$-\Delta u \geq \overline g(u)$ does not have positive solutions in $\R^N\setminus B_R$ for any $R>0$.

We argue indirectly and assume there exists a positive solution $u\in\widetilde F$.
We note that $u$ satisfies 
	\begin{equation}\label{decay}
		u(x)\to 0  \quad \text{as}\  \abs x\to\infty.
	\end{equation}
In fact, \eqref{decay} follows from 
\cref{Lemma:2.4} (i) when $N\geq 3$.  When $N=2$, modifying the argument for \cref{Lemma:2.4} (ii),
we can show for some constant $C'>0$
	\begin{equation}\label{2.12a}
		\abs{u(x)}^{q+3\over 2} \leq C'
        \norm{\nabla u}_2\, \norm u_{q+1}^{q+1\over 2}\,  \abs x^{-1} \quad \text{for all}\ u\in \widetilde F.
    \end{equation}
Thus we have \eqref{decay} also for $N=2$.

From the decay property \eqref{decay}, there exists $R>0$ such that
	$$	u(x)\in (0,\delta) \quad \text{for all}\ \abs x\geq R.
	$$
Thus $u(x)$ is a positive solution of $-\Delta u=\overline g(u)$ in $\R^N\setminus B_R$, which is a
contradiction to the result of \cite{AS} and we complete the proof of \cref{Theorem:L}.
\end{proof}

\section{Minimax properties of $I$} \label{Section:3}

\noindent
{\bf Notation.}
For sake of simplicity of notation, throughout this paper we use notation:
    \begin{equation*}
        \mu = e^\lambda \quad \text{for}\ \lambda\in \R.
    \end{equation*}
For example, we write
    \begin{equation*}
        I(\lambda,u)=\half\norm{\nabla u}_2^2 -\intRN G(u)\,dx
        +\mu\left\{\half\norm u_2^2 -m_1\right\}.
    \end{equation*}

To show the existence or non-existence of solutions of \eqref{1.1} with $m=m_1$, 
geometric properties of $I$ are quite important.  For the existence result, 
we introduce two minimax values $\underline{b}$ and $\overline{b}$, 
which are inspired by the works of Bahri and Li \cite{BaL} and Tanaka 
\cite{Tan0}. 
To define these values 
we construct a map $\zeta_0 \in C^2(\R,E)$ with the following properties: 
    \begin{align}
    &\Psi_\mu ( \zeta_0 (\lambda) ) < -2 Am_1 - 1 \quad 
            \text{for any}\ \lambda \in \R; \label{3.1}\\
    &\frac{1}{2}\norm{ \zeta_0(\lambda) }_2^2 > m_1 \quad 
            \text{for any}\ \lambda \in \R; \label{3.2}\\
    &\max_{0 \leq t \leq 1} I(\lambda, t \zeta_0 (\lambda) ) \to 0 \quad
        \text{as} \ \lambda \to \pm\infty; \label{3.3}\\
    &\zeta_0(\lambda)(x)\geq 0 \qquad \text{for all}\ \lambda\in\R \ 
        \text{and}\ x\in\R^N. \label{3.4}
    \end{align}
Here $\Psi_\mu:\, E\to\R$ is defined by
    \begin{equation}\label{3.5}
        \Psi_\mu(u) = \frac{1}{2} \norm{\nabla u}_2^2 + \frac{\mu}{2} \norm{u}_2^2 
        - \intRN G(u) \, dx.
    \end{equation}
Properties \eqref{3.1}--\eqref{3.4} are important to see geometry of $I$;
\eqref{3.2} ensures that $\zeta_0(\lambda)$ stays outside of the cylinder
    \begin{equation} \label{3.6}
        M_0= \R \times \calS_0 = \R\times \Set{u\in E | \frac{1}{2}\norm u_2^2=m_1}.
    \end{equation}
Property \eqref{3.1} implies $I(\lambda,\zeta_0(\lambda))<-2Am_1-1$ for all $\lambda\in\R$, 
which is lower the minimum height of $I$ on the cylinder $M_0$ 
(c.f. \cref{Lemma:3.2} (i)).  Property \eqref{3.3} gives
information on the behavior of $\zeta_0(\lambda)$ as $\lambda\to\pm\infty$.

We note that \eqref{3.1} implies \eqref{3.2}. In fact, $\Psi_\mu(u)\geq -2Am_1$ for 
$\half\norm u_2^2\leq m_1$ by \eqref{2.1} and \cref{Corollary:2.2}.  
In what follows we construct $\zeta_0$ with properties \eqref{3.1}, \eqref{3.3} and
\eqref{3.4}.

\medskip

\subsection{Mountain pass geometry for unconstraint problems} \label{Section:3.1}

The functional $\Psi_\mu$ defined in \eqref{3.5} corresponds to
    \begin{equation} \label{3.7}
    -\Delta u+\mu u = g(u) \quad \text{in}\ \R^N.       
    \end{equation}
It is readily seen that $\Psi_\mu$ has the mountain pass geometry and we define 
    \begin{equation}  \label{3.8}
        a(\mu)=\inf_{\gamma\in \Lambda_\mu}\max_{t\in [0,1]} \Psi_\mu(\gamma(t)),
    \end{equation}
where
    \begin{equation} \label{3.9}
        \Lambda_\mu = \Set{ \gamma\in C([0,1],E) | \gamma(0)=0, \
        \Psi_\mu(\gamma(1)) <0}.        
    \end{equation}

As a special case of the results in \cite{JT0,HIT}, we have

\begin{Proposition} \label{Proposition:3.1}
Assume that $g$ satisfies \ref{(g0')}, 
$\lim_{s\to 0}{\frac{g(s)}{s}}=0$ and $\lim_{s\to\pm\infty}{\frac{g(s)}{s}}=\infty$.
Moreover suppose that
$g$ is subcritical at $s=\pm\infty$, that is,
$g$ satisfies the condition \ref{(f1)} in \cref{Appendix:A.2}. 
Then for any $\mu>0$
\begin{enumerate}[label={\rm (\roman*)}]
\item[{\rm (0)}] 
$s\mapsto g(s)-\mu s$ satisfies the Berestycki-Lions' condition \cite{BeL};
\item 
$a(\mu)>0$ is a critical value of $\Psi_\mu$, that is,
there exists a $v\in E$ such that $\Psi_\mu(v)=a(\mu) > 0$ and $\Psi_\mu'(v)=0$.
Moreover $v>0$ in $\R^N$ and $v$ is a least energy solution of \eqref{3.7};
\item 
for any critical point $v\in E\setminus\{ 0\}$ of $\Psi_\mu$,
there exists a path $\gamma_v\in\Lambda_\mu$ such that
    \begin{equation*} 
        v\in\gamma_v([0,1]), \quad 
        \Psi_\mu(v)=\max_{t\in [0,1]} \Psi_\mu(\gamma_v(t)).    
    \end{equation*}
In particular, there is an optimal path $\gamma_\mu\in\Lambda_\mu$ such that
    \begin{equation*}
        a(\mu)=\max_{t\in [0,1]} \Psi_\mu(\gamma_{\mu}(t));
    \end{equation*}
\item 
for any $C \geq 0$, the level set $\set{u\in E | \Psi_\mu(u)<-C}$ is path-connected, 
namely for any $u_0$, $u_1\in E$ with $\Psi_\mu(u_0)<-C$ and $\Psi_\mu(u_1)<-C$, 
there exists a path $\zeta\in C([0,1],E)$ such that
    \begin{equation}  \label{3.10}
        \zeta(0)=u_0, \quad \zeta(1)=u_1, \quad \Psi_\mu(\zeta(t)) <- C 
        \quad \text{for all}\ t\in [0,1].
    \end{equation}
Moreover, when $u_0\geq 0$, $u_1\geq 0$ in $\R^N$, there exists a path 
$\zeta \in C([0,1],E)$ with \eqref{3.10} and $\zeta(t)(x)\geq 0$ for all $x\in\R^N$ and
$t\in [0,1]$.
\end{enumerate}
Moreover we have
\begin{enumerate}
\item[{\rm (iv)}] $a(\mu):\, (0,\infty)\to\R$ is strictly increasing and continuous.
\end{enumerate}
\end{Proposition}

We note that under \ref{(g0')} and \ref{(g1)} the assumption of \cref{Proposition:3.1} 
holds.

\begin{proof}
It is easily seen that $s\mapsto g(s)-\mu s$ satisfies the Berestycki-Lions' condition 
\cite{BeL} (c.f. \ref{(f0)}--\ref{(f3)} in \cref{Appendix:A.2}) for any $\mu>0$ 
under the assumption of \cref{Proposition:3.1}.

Properties (i) and (ii) are proved in \cite{JT0}.
We note that the existence of a least energy solution is obtained in \cite{BeGaKa83,BeL}
through a constraint minimization problem for $\norm{\nabla u}_2^2$  on 
$\{ u\in E\,|\, -{\frac{\mu}{2}}\norm u_2^2 + \intRN G(u)\, dx=1\}$ for $N\geq 3$ and on
$\{ u\in E\,|\, -{\frac{\mu}{2}}\norm u_2^2 + \intRN G(u)\, dx=0, \, \norm u_2=1\}$ 
for $N=2$.
In \cite{JT0}, Mountain Pass characterization of a least energy solution is given together
with an optimal path, which is given explicitly.  
See \cref{Proposition:A.2} in \cref{Appendix:A.2} for an explicit construction
of the path.

Property (iii) when $C=0$ is proved in \cite[Lemma 6.1]{HIT} and 
the argument also works for the case $C > 0$. 
We can also choose a non-negative path $\zeta\in C([0,1],E)$ when $u_0$ and $u_1$ are
non-negative in $\R^N$.

For property (iv), the strict monotonicity of $a(\mu)$ follows from the existence 
of optimal paths and the strict monotonicity of $\Psi_\mu$ on $\mu$. 
For the continuity of $a(\mu)$, first note that $a(\mu)$ is upper-semi-continuous by the
minimax definition of $a(\mu)$.
On the other hand, the lower-semi-continuity of $a(\mu)$ follows from the compactness
of $\{u\in E\,|\, \Psi_\mu(u)\leq M, \Psi_\mu'(u)=0, \mu\in [a,b]\}$ ($M>0$, 
$0<a<b<\infty$), which can be deduced from the fact that any critical point satisfies 
the Pohozaev identity $P(\lambda,u)=0$.  Thus $a(\mu):\, (0,\infty)\to\R$ is continuous. 
\end{proof}

\subsection{Geometry of $I$} \label{Section:3.2}
In this section, we observe geometric properties of $I$ in order to construct 
the map $\zeta_0 \in C^2(\R,E)$. 
In \cref{Section:3.2,Section:3.3}, we relax condition \ref{(g1)} and 
consider the behavior of $I$ under the following condition \ref{(g1*)} to 
clarify the role of the sublinear growth condition \ref{(g1)}.
\begin{enumerate}[label={\rm (g1*)}]
\item \label{(g1*)} For $g(s)=\abs s^{p-1}s +h(s)$, there exists $\alpha\in\R$ such 
that
    \begin{equation*}
        \lim_{s\to 0}{\frac{h(s)}{\abs s^{p-1}s}} =0, \quad
        \lim_{s\to \pm\infty}{\frac{h(s)}{s}} =\alpha.
    \end{equation*}
\end{enumerate}
The cylinder $M_0=\R \times \calS_0$ (c.f. \eqref{3.6})
plays an important role for the geometry of $I$.

\begin{Lemma} \label{Lemma:3.2}
Assume \ref{(g0')} and \ref{(g1*)}. Then
\begin{enumerate}[label={\rm (\roman*)}]
\item $I(\lambda,u)\geq -2Am_1$ on $M_0$, where $A>0$ is given in \eqref{2.1};
\item $I(\lambda,0)=-\mu m_1$ for all $\lambda\in\R$;
\item for any $\lambda\in\R$ and $u\in E\setminus\{0\}$,
    \begin{equation*}
            I(\lambda,T u)\to -\infty \quad \text{as}\ T\to\infty.
    \end{equation*}
\end{enumerate}
\end{Lemma}

\begin{proof}
(i) By \eqref{2.1} and \cref{Corollary:2.2}, we have for $(\lambda,u)\in M_0$
    \begin{equation*}
        \begin{aligned}
        I(\lambda,u) &= \half\norm{\nabla u}_2^2 -{\frac{1}{p+1}}\norm u_{p+1}^{p+1} 
            -\intRN H(u)\,dx 
            \geq 0 -A\norm u_2^2 = -2A m_1. 
        \end{aligned}
    \end{equation*}
Thus (i) holds. Since it is not difficult to check (ii) and (iii), we omit the details. 
\end{proof}

We next study the behavior of 
    \begin{equation*}
        (t,\lambda)\mapsto I(\lambda,t\omega_\mu);\, [0,\infty)\times \R\to\R.
    \end{equation*}
Note that 
    \begin{equation} \label{3.11}
        I(\lambda,u)=\Psi_\mu(u)-\mu m_1,\quad \mu = e^\lambda, 
    \end{equation}
where $\Psi_\mu:\, E\to\R$ is defined in \eqref{3.5}, and recall that for each 
$\lambda\in\R$
    \begin{equation*}
        t\mapsto t\omega_\mu
    \end{equation*}
is an optimal path for the mountain pass value for $\Psi_{0\mu}$.

\begin{Lemma} \label{Lemma:3.3}
Assume \ref{(g0')} and \ref{(g1*)}. 
Then there exist $T_0 > 1$ and $\lambda_0\in (-\infty, 0)$ such that 
\begin{enumerate}[label={\rm (\roman*)}]
\item 
for all $\lambda \in (-\infty, \lambda_0]$, $I(\lambda,T_0 \mu^{-1/(p+1)} \omega_\mu)
< -2Am_1 - 1 - \mu m_1$; 
\item 
as $\lambda \to - \infty$, 
$\displaystyle \max_{ 0 \leq t \leq T_0 } I(\lambda, t \mu^{-1/(p+1)} \omega_\mu )\to 0$. 
\end{enumerate}
In addition, there exists $T_1 > T_0 \mu_0^{-1/(p+1)} $ where 
$\mu_0 = e^{\lambda_0}$ such that 
\begin{enumerate}[label={\rm (\roman*)}]
\setcounter{enumi}{2}
\item 
for all $\lambda\in [\lambda_0,\infty)$, $I(\lambda,T_1\omega_\mu)< -2A m_1 -1 -\mu m_1$; 
\item 
as $\lambda \to +\infty$, $\displaystyle \max_{0 \leq t \leq T_1} I(\lambda, t\omega_\mu) 
\to - \alpha m_1$. 
\end{enumerate}
\end{Lemma}

\begin{proof}
For $\Psi_{0\mu}$ defined in \eqref{2.8}, we have by \eqref{2.9} and 
the definition of $\omega_\mu$, 
    \begin{equation}\label{3.12}
        \begin{aligned}
        I(\lambda, t \mu^{-1/(p+1)} \omega_\mu) 
        &=  \Psi_{0\mu} (t \mu^{-1/(p+1)} \omega_\mu) 
            -\intRN H(t \mu^{-1/(p+1)} \omega_\mu) \, dx - \mu m_1 \\
        &= \mu \left\{ \Psi_{01} (t \mu^{-1/(p+1)} \omega_1) 
            -\mu^{-N/2 - 1} \intRN H 
                \left( t \mu^{-1/(p+1)} \mu^{N/4} \omega_1(x) \right) dx 
            \right\} - \mu m_1.
        \end{aligned}
    \end{equation}
It follows from \ref{(g1*)} that there exists $s_0 > 0$ such that 
    \[  \abs{H(s)} \leq \frac{\abs{s}^{p+1}}{2(p+1)} \quad 
        \text{for all $s \in \R$ with $\abs{s}\leq s_0$}.
    \]
Select $T_0>1$ so that 
    \[  - \frac{T_0^{p+1}}{2(p+1)} \norm{\omega_1}_{p+1}^{p+1} < -4Am_1 - 2.
    \]
By $\frac{N}{4} - \frac{1}{p+1} = \frac{N^2}{4(N+2)}> 0$ 
and $\omega_1 \in L^\infty(\R^N)$, 
there exists $\lambda_0^* \in (-\infty,0)$ such that 
    \[  \norm{ T_0 \mu^{ N/4 - 1/(p+1) } \omega_1 }_\infty \leq s_0 
        \quad 
        \text{for all $\lambda \in \R$ with $\lambda \leq \lambda_0^*$}. 
    \]
Thus, for any $\lambda \in (-\infty, \lambda_0^*]$ and $t \in [0,T_0]$, 
    \begin{equation*}
        \begin{aligned}
        \MoveEqLeft \mu \left\{ \Psi_{01} (t \mu^{-1/(p+1)} \omega_1) 
            - \mu^{-N/2 - 1} \intRN H 
            \left( t \mu^{-1/(p+1)} \mu^{N/4} \omega_1(x) \right) dx \right\} \\
        & \leq
            \frac{t^2}{2} \mu^{(p-1)/(p+1)} \norm{\omega_1}_E^2 
            - \frac{t^{p+1}}{2(p+1)} \norm{\omega_1}_{p+1}^{p+1}.
        \end{aligned}
    \end{equation*}
From \eqref{3.12} we deduce that 
    \[  \limsup_{\lambda \to - \infty} I (\lambda , T_0 \mu^{-1/(p+1) } \omega_1 ) 
        \leq 
        - \frac{T_0^{p+1}}{2(p+1)} \norm{\omega_1}_{p+1}^{p+1} 
        < - 4 A m_1 -2.
    \]
Hence, assertion (i) holds for some $\lambda_0 \leq \lambda_0^*$.

For assertion (ii), notice that \ref{(g1*)} implies that for any $\varepsilon > 0$ 
there exists $s_\varepsilon \in (0,1) $ 
such that $\abs{H(s)} \leq \varepsilon \abs{s}^{p+1}$ when $\abs{s} \leq s_\varepsilon$. 
Thus, \eqref{3.12} and $\max_{0 \leq t} \Psi_{01} (t \omega_1) = m_1$ yield 
    \[  \begin{aligned}
        \limsup_{\lambda \to - \infty } \max_{0 \leq t \leq T_0} 
            I(\lambda, t \mu^{-1/(p+1)} \omega_1 ) 
        &\leq \limsup_{\lambda \to - \infty} 
            \max_{0 \leq t \leq T_0} - \mu^{ -N/2 } \intRN 
            H \left( t \mu^{ -1/(p+1)  } \mu^{N/4} \omega_1(x) \right) dx \\
        &\leq T_0^{p+1} \varepsilon \norm{\omega_1}_{p+1}^{p+1}.
        \end{aligned}
    \]
Since $\varepsilon > 0$ is arbitrary, 
$\limsup_{\lambda \to - \infty} \max_{ 0 \leq t \leq T_0} 
I(\lambda, t \mu^{-1/(p+1)} \omega_1 ) \leq 0$. 
On the other hand, by $I(\lambda, 0) = -\mu m_1 \to 0$ as $\lambda \to -\infty$, 
we see that assertion (ii) also holds.

Next we show (iii). It follows from \eqref{2.1} and \eqref{3.12} that 
    \begin{align}
        I(\lambda,t\omega_\mu) &= 
        \mu \Psi_{01} (t \omega_1) - \mu^{-N/2} \intRN H 
            \left( t \mu^{N/4} \omega_1(x) \right) dx - \mu m_1  \label{3.13}\\
        &\leq 
        \mu \Psi_{01} (t \omega_1) + 2 t^2 Am_1 - \mu m_1 
        = \mu \left\{ \Psi_{01} (t \omega_1) 
            + \frac{2 t^2 Am_1}{\mu} \right\} - \mu m_1. \nonumber
    \end{align}
Since we may find $c_0>0$ so that $\Psi_{01} (t \omega_1) \leq - c_0 t^{p+1}$ for all 
sufficiently large $t > 1$, there exists $T_1 > T_0 \mu_0^{-1/(p+1)}$, $\mu_0=e^{\lambda_0}$ 
such that 
    \[  \Psi_{01} (T_1 \omega_1) + \frac{2T_1^2 Am_1}{\mu} < - \frac{4Am_1 + 2}{\mu} 
        \quad \text{for all}\ \mu \geq \mu_0.
    \]
From \eqref{3.13}, we deduce that $I(\lambda, T_1 \omega_\mu) < -2 Am_1 -1 - \mu m_1 $ 
for all $\lambda \geq \lambda_0$ and (iii) holds.

To prove (iv), we first estimate the term 
    \[  -\mu^{-N/2} \intRN H \left( t \mu^{N/4} \omega_1(x) \right) dx \quad 
        \text{for $t \in [0,T_1]$}. 
    \]
By \ref{(g1*)}, for any $\epsilon>0$ there exists $C_\epsilon>0$ such that
    \begin{equation*}
    \frac{\alpha-\epsilon}{2} s^2 - C_\varepsilon \abs{s} \leq H(s) 
        \leq \frac{\alpha + \epsilon}{2} s^2 + C_\epsilon \abs{s} 
    \quad \text{for all}\ s\in\R.
    \end{equation*}
From these inequalities it follows that
    \begin{equation*}
        \begin{aligned}
        \max_{0 \leq t \leq T_1}\left|\mu^{-N/2}\intRN H(t\mu^{N/4}\omega_1)\, dx 
            - \frac{\alpha}{2} t^2 \norm{ \omega_1 }_2^2 \right|
        &\leq \frac{\epsilon}{2} T_1^2 \norm{ \omega_1 }^2_2 
            + \mu^{-N/4}C_\epsilon T_1\norm{\omega_1}_1 \\
        &\to \frac{\epsilon}{2} T_1^2 \norm{\omega_1}_2^2 
            \quad \text{as}\ \lambda\to \infty. 
        \end{aligned}
    \end{equation*}
Since $\epsilon>0$ is arbitrary,
    \begin{equation*}
    \mu^{-N/2}\intRN H(t\mu^{N/4}\omega_1)\, dx \to \alpha m_1 t^2 \quad
    \text{as}\ \lambda\to\infty \ \text{uniformly in}\ t\in [0,T_1], 
    \end{equation*}
Since the function $ t \mapsto \Psi_{01} ( t \omega_1 );\, [0,\infty)\to\R$ admits 
only one maximum point at $t=1$, 
recalling \eqref{3.13} and \eqref{2.9}, we obtain 
    \[  \begin{aligned}
        \lim_{\lambda \to \infty} \max_{0 \leq t \leq T_1} I (\lambda, t\omega_\mu) 
        &= \lim_{\lambda \to \infty} \max_{0 \leq t \leq T_1} 
            \left[ \mu \left\{ \Psi_{01} \left( t \omega_1 \right) - m_1 \right\} 
            - \alpha m_1 t^2 + o(1) \right] = - \alpha m_1. 
        \end{aligned}
    \]
Thus, (iv) holds. 
\end{proof}

\subsection{Construction of $\zeta_0$ and minimax values $\ub$ and $\ob$} 
\label{Section:3.3}
In this section, we construct the map $\zeta_0$ and define $\ub$ and $\ob$. 
Recall our choice of $\lambda_0\in(-\infty,0)$, $\mu_0=e^{\lambda_0}\in (0,1)$ and 
$T_0$, $T_1>1$.
In view of \eqref{3.11} and \cref{Lemma:3.3}, 
    \[  \begin{aligned}
            &\Psi_{\mu} (T_0 \mu^{-1/(p+1)} \omega_\mu ) < -2 A m_1 - 1 & 
                &\text{for all $\lambda \in (-\infty, \lambda_0]$}, \\
            &\Psi_{\mu} (T_1 \omega_\mu) < -2 A m_1 - 1 & 
                &\text{for all $\lambda \in [\lambda_0,\infty)$}. 
        \end{aligned}
    \]
By \cref{Proposition:3.1} (iii), we find a path $\sigma_0\in C([0,1],E)$ 
such that
    \begin{align}   
        &\sigma_0(0)= T_0\mu_0^{-{\frac{1}{p+1}}}\omega_{\mu_0}, \quad 
            \sigma_0(1)=T_1\omega_{\mu_0},      \nonumber \\
        &\Psi_{\mu_0}(\sigma_0(t)) < -2Am_1-1 \quad 
            \text{for all}\ t\in [0,1],     \label{3.14}\\
        &\sigma_0(t)(x)\geq 0\quad \text{for all}\ t\in[0,1]\ 
            \text{and}\ x\in\R^N.   \nonumber
    \end{align}
It follows from \eqref{3.14} that for $\delta>0$ small
    \begin{equation*}
        \Psi_\mu(\sigma_0(t)+\varphi) < -2Am_1-1  
    \end{equation*}
for all $t\in [0,1]$, $\lambda\in [\lambda_0,\lambda_0+2\delta]$ and $\varphi\in E$ with
$\norm\varphi_E<\delta$.  

For $\delta>0$ small we set
$\mu_1=e^{\lambda_0+2\delta}$ and $\wzeta_0\in C(\R,E)$ by
    \begin{equation*}
        \wzeta_0(\lambda)=
            \begin{cases}
            T_0\mu^{-\frac{1}{p+1}}\omega_\mu   
                &\text{for}\ \lambda\in (-\infty,\lambda_0],\\
            \sigma_0\left({\frac{\lambda-\lambda_0}{\delta}}\right)
                &\text{for}\ \lambda\in (\lambda_0,\lambda_0+\delta],\\
            \left(1-{\frac{\lambda-(\lambda_0+\delta) }{\delta}}\right)T_1\omega_{\mu_0}
                +{\frac{\lambda-(\lambda_0+\delta)}{\delta}} T_1\omega_{\mu_1}
                &\text{for}\ \lambda\in (\lambda_0+\delta,\lambda_0+2\delta],\\
                T_1\omega_\mu   &\text{for}\ \lambda\in (\lambda_0+2\delta,\infty).\\
            \end{cases}
    \end{equation*}
We note that $(1-{\frac{\lambda-(\lambda_0+\delta)}{\delta}})T_1\omega_{\mu_0}
+{\frac{\lambda-(\lambda_0+\delta)}{\delta}}T_1\omega_{\mu_1}:\, 
[\lambda_0+\delta,\lambda_0+2\delta] \to E$ is a continuous positive path joining
$\sigma_0(1)=T_1\omega_{\mu_0}$ and $T_1\omega_{\mu_1}$
and 
$\widetilde\zeta_0([\lambda_0+\delta,\lambda_0+2\delta])$ is close to $T_1\omega_{\mu_0}$
when $\delta>0$ is close to $0$.
Thus for $\delta>0$ small, we have
    \begin{align*}
        &\Psi_\mu(\wzeta_0(\lambda)) < -2Am_1 -1 \qquad 
            \text{for all}\ \lambda\in\R, \nonumber\\
        &\wzeta_0(\lambda)(x) \geq 0 \qquad 
            \text{for all}\ \lambda\in\R\ \text{and}\ x\in\R^N. \nonumber 
    \end{align*}
Modifying $\wzeta_0$ in a neighborhood of $[\lambda_0,\lambda_0+2\delta]$ slightly, 
we obtain $\zeta_0\in C^2(\R,E)$ with the desired properties \eqref{3.1}--\eqref{3.4}.

\begin{Remark} \label{Remark:3.4}    
When $G(s)={\frac{1}{p+1}}\abs s^{p+1}$ (equivalently $h \equiv 0$), we have 
    \begin{equation*}
        I(\lambda,t\omega_\mu)=\mu m_1
        \left( (N+2)\left(\half t^2 -{\frac{1}{p+1}}t^{p+1}\right)-1\right)
    \end{equation*}
and we note that
    \begin{equation} \label{3.15}
        \max_{t\geq 0} I(\lambda,t\omega_\mu)=I(\lambda,\omega_\mu)=0.  
    \end{equation}
Thus, choosing $\xi\in C^\infty(\R,(0,\infty))$ with
    \begin{equation*}
        \xi(\lambda)=\begin{cases}  
            \mu^{-{\frac{1}{p-1}}}  &\text{for}\ \lambda\leq -1,\\
            1                       &\text{for}\ \lambda\geq 0,
        \end{cases}
    \end{equation*}
we observe that for a large constant $C>1$, a map $\zeta_0(\lambda)=C\xi(\lambda)\omega_\mu$ 
satisfies the desired properties \eqref{3.1}--\eqref{3.4}.  
\end{Remark}

Using $\zeta_0$, we first define $\ub$ by 
    \begin{align*} 
    &\ub =\inf_{\gamma\in \uGamma}\max_{t\in [0,1]} I(\gamma(t)), \\
    &\uGamma= \Set{ 
        \gamma(t)\in C([0,1],\RE) | 
        \begin{aligned}
            &\gamma(0)\in\R\times\{ 0\}, \ I(\gamma(0)) < - 2 A m_1 - 1,\\
            &\gamma (1) \in ({\rm id}\times\zeta_0)(\R) 
        \end{aligned}
    }.
    \end{align*}    
By construction, we note that the condition on $\gamma (1)$ implies 
$I(\gamma(1)) < -2Am_1 - 1$ and $\frac{1}{2}\norm{\gamma_2(1)}_2^2 > m_1$.
Thus 
    \begin{equation}\label{3.16}
        \gamma([0,1])\cap (\R\times M_0)\not=\emptyset  \quad 
            \text{for all}\  \gamma\in\uGamma.
    \end{equation}

For another minimax value $\ob$, we define $\gamma_0$ by 
    \[  \gamma_0 (t,\lambda) = \left( \lambda , \ t \zeta_0 (\lambda) \right) 
        \in C^2( [0,1] \times \R , \RE )
    \]
and for $L> 1$, a collar $\calC(L) \subset [0,1] \times \R$ by 
\begin{equation*}
    \calC(L) = \left( [0,1]\times \left( (-\infty,-L ] \cup [L,\infty) \right) \bigcup
    ([0,1/L]\cup [1-1/L,1])\times \R \right) \subset [0,1] \times \R. 
\end{equation*}
Then we put 
    \begin{align*}  
    \ob 
    &=\inf_{\gamma\in\oGamma}\sup_{(t,\lambda)\in [0,1]\times\R} I(\gamma(t,\lambda)),\\
    \oGamma &=\Set{ \gamma(t,\lambda)\in C([0,1]\times\R,\RE) | 
        \begin{aligned}
            &\gamma(t,\lambda)=\gamma_0(t,\lambda) \ 
            \text{for}\ (t,\lambda)\in \calC(L_\gamma),\\
            &\text{where} \ L_\gamma>0 \
            \ \text{ possibly depends on $\gamma$}
        \end{aligned}
    }. 
    \end{align*}
We note that $\gamma_0\in\oGamma$ and $\oGamma\not=\emptyset$ under \ref{(g1*)}.
Moreover, by \ref{(g1*)}, \cref{Lemma:3.3} (ii) and (iv) 
and $\gamma = \gamma_0$ on $\calC(L_\gamma)$, we have 
    \begin{equation}\label{3.17}
        \lim_{\lambda \to - \infty} \max_{0 \leq t \leq 1} 
            I\left( \gamma(t, \lambda) \right) = 0, \quad 
        \lim_{\lambda \to +\infty} \max_{0 \leq t \leq 1} 
            I \left( \gamma (t,\lambda) \right) = -\alpha m_1.
    \end{equation}
The minimax values $\ub$ and $\ob$ enjoy the following properties.


\begin{Lemma}\label{Lemma:3.5}
Assume \ref{(g0')} and \ref{(g1*)}.  Then
the following inequalities hold: 
    \begin{align}   
        &-2Am_1 \leq \ub\leq \min\{-\alpha m_1, 0\},    \label{3.18}\\
        &\ob \geq \max\{ -\alpha m_1, 0\}.              \label{3.19}
    \end{align}
\end{Lemma}

\begin{proof}
Property \eqref{3.16} and \cref{Lemma:3.2} yield $-2Am_1 \leq \ub$. 
On the other hand, fix $\lambda_1 \in \R$ so that 
$I(\lambda_1, 0) = - e^{\lambda_1} m_1 < - 2Am_1 - 1$. 
Let $\lambda \ll -1$ and consider a path $\gamma$ obtained by joining 
the following paths: for $0 \leq \theta \leq 1$, 
    \[  \left( (1-\theta) \lambda_1 + \theta \lambda , 0 \right), \quad 
        \left( \lambda, \theta\zeta_0(\lambda)\right).
    \]
From the monotonicity of $\theta \mapsto I( (1-\theta)\lambda_1 + \theta\lambda, 0)$,
we infer that $\gamma \in \uGamma$.  
Recalling $\zeta_0(\lambda) = T_0 \mu^{-1/(p+1)} \omega_\mu$,
we have
    \[  \ub \leq \max_{0 \leq t \leq 1} I(\gamma(t)) 
        = \max_{0 \leq \theta \leq 1} I(\lambda, \theta\zeta_0(\lambda))
        = \max_{0 \leq t \leq T_0} I(\lambda, t \mu^{-1/(p+1)} \omega_{\mu} ). 
    \]
Letting $\lambda \to - \infty$, $\ub \leq 0$ follows from \cref{Lemma:3.3} (ii).
Next we consider the case $\lambda\gg 1$ and define $\gamma(\lambda)$ in a similar
way.  
Considering the limit $\lambda \to + \infty$, we get $\ub \leq -\alpha m_1$ 
by \cref{Lemma:3.3} (iv). Thus, \eqref{3.18} holds.  
\eqref{3.19} follows from \eqref{3.17}.
\end{proof}

In what follows, for the existence of solutions, we treat the balanced case 
($\alpha=0$, i.e., \ref{(g1)}) in Sections \ref{Section:4}--\ref{Section:7}.
On the other hand, the non-existence of solutions will be considered 
in \cref{Section:9} for the unbalanced case ($\alpha \neq 0$). 

\section{Mountain pass values $b(\lambda)$ and $\ub$, $\ob$}\label{Section:4}

\subsection{Properties of $b(\lambda)$ and $\ub$, $\ob$} \label{Section:4.1}

In what follows, we assume \ref{(g1)} and study the existence of a critical point 
of $I$. We will see either $\ub$ or $\ob$ is a critical value of $I$.
Under \ref{(g1)}, $\alpha = 0$ holds and  we deduce the following corollary from 
\eqref{3.18}--\eqref{3.19}.

\begin{Corollary} \label{Corollary:4.1}
Assume \ref{(g0')} and \ref{(g1)}. Then
    \begin{equation*}
        -2Am_1 \leq \ub \leq 0 \leq \ob.
    \end{equation*}
\end{Corollary}

To see further properties of $\underline{b}$ and $\overline{b}$, 
for $\lambda\in\R$ we introduce the mountain pass value $b(\lambda)$ of a functional 
$ u\mapsto I(\lambda,u);\, E\to \R$ by
    \begin{equation}\label{4.1}
        b(\lambda) = \inf_{\lambda\in\widehat{\Lambda}_\mu} \max_{t\in [0,1]}
        I(\lambda,\gamma(t)),
    \end{equation}
where 
    \begin{equation}\label{4.2}
        \widehat{\Lambda}_\mu = \Set{ \gamma \in C([0,1],E) | \gamma(0)=0, \ 
        \gamma(1) = \zeta_0(\lambda) }. 
    \end{equation}
We note that 
    \begin{equation} \label{4.3}
        b(\lambda)= a(\mu)-\mu m_1, \ \mu=e^\lambda,            
    \end{equation}
where $a(\mu)$ is the mountain pass value defined in \eqref{3.8}--\eqref{3.9}.
In fact, noting \eqref{3.11} and the connectedness of $\{ u\in E\,|\, \Psi_\mu(u)<0\}$
by \cref{Proposition:3.1}, we have \eqref{4.3}.

We give fundamental properties of $b(\lambda)$.

\begin{Proposition} \label{Proposition:4.2}
Suppose  \ref{(g0')} and \ref{(g1)}. Then 
\begin{enumerate}[label={\rm (\roman*)}]
\item $b(\lambda)\in C(\R,\R)$;
\item $\ub\leq b(\lambda)\leq \ob$ for all $\lambda\in\R$;
\item $b(\lambda)\to 0$ as $\lambda\to -\infty$;
\item $b(\lambda)\to 0$ as $\lambda\to + \infty$. 
\end{enumerate}
\end{Proposition}

Property (i) follows from \eqref{4.3} and the continuity of $a(\mu)$ obtained in 
\cref{Proposition:3.1} (iv). Here we give proofs of (ii) and (iii).  
Properties (i)--(iv) are important to see the topological properties of $I$ 
through $b(\lambda)$.
However to show \cref{Theorem:1.1} properties (i)--(iii) are enough and 
moreover a proof of (iv) is lengthy.  We will give it in \cref{Section:4.4}.
Later in \cref{Corollary:7.5} in \cref{Section:7.3}, we will show
$\ub = \inf_{\lambda\in\R} b(\lambda)$.

\begin{proof}[Proof of \cref{Proposition:4.2} (ii)]
Here we use an idea from \cite{Tan0}.
We fix $\lambda \in \R$ arbitrary and show $\ub\leq b(\lambda)$. 
Let $\gamma \in \widehat{\Lambda}_\mu$ be given and 
take $\lambda_1> \lambda$ with $- e^{\lambda_1} m_1 < -2Am_1-1$. 
As in the proof of \cref{Lemma:3.5}, consider a path $\zeta$ obtained via joining 
$((1-\theta) \lambda_1 + \theta \lambda, 0):\, [0,1]\to \RE$ and $\gamma:\, [0,1]\to\RE$. 
Then it can be verified that $\zeta \in \uGamma$ and 
    \[  \ub \leq \max_{ 0 \leq t \leq 1} I(\zeta(t)) 
        = \max_{0 \leq t \leq 1} I(\lambda, \gamma(t)). 
    \]
Since $\gamma \in \widehat{\Lambda}_\mu$ is arbitrary, $\ub \leq b(\lambda)$ holds. 

Next we show $b(\lambda)\leq \ob$. For this purpose, for each 
$\gamma \in C^2( [0,1] \times \R , \RE ) \cap \oGamma$ 
we prove the existence of a dense subset $\Xi \subset \R$ with the following property:
for any $\lambda \in \Xi$ there exists a path $\sigma_\lambda \in\Lambda_\mu$, 
$\mu=e^\lambda$ satisfying 
    \begin{equation} \label{4.4}
        \{\lambda\}\times \sigma_\lambda([0,1])\subset \gamma([0,1]\times\R).
    \end{equation}
We note that \eqref{4.4} implies
    \begin{equation*}
        b(\lambda) \leq \max_{t\in [0,1]} I(\lambda,\sigma_\lambda(t))
        \leq \sup_{(t,\lambda')\in [0,1]\times\R} I(\gamma(t,\lambda'))
        \quad \text{for}\ \lambda\in\Xi.
    \end{equation*}
Since $b(\lambda)$ is continuous by (i) and $\Xi$ is dense in $\R$,  we have
    \begin{equation} \label{4.5}
        b(\lambda) \leq \sup_{(t,\lambda')\in [0,1]\times\R} I(\gamma(t,\lambda'))
        \quad \text{for all}\ \lambda\in\R.
    \end{equation}
Now we take $\gamma\in \oGamma$ arbitrary.  We note that 
$\gamma(t,\lambda)=\gamma_0(t,\lambda)$ holds for all $(t,\lambda)$ except for 
a compact set included in $(0,1)\times \R$.  
Recalling $\gamma_0 \in C^2( [0,1] \times \R, \RE )$, it is not difficult to see
that there exists a sequence 
$(\gamma_n)_{n=1}^\infty\subset C^2( [0,1] \times \R , \RE ) \cap \oGamma$ 
such that
$\gamma_n(t,\lambda)\to\gamma(t,\lambda)$ uniformly in $[0,1]\times\R$.
Thus
    \[  \sup_{(t,\lambda) \in [0,1] \times \R} I(\gamma_n(t,\lambda)) 
        \to \sup_{(t,\lambda) \in [0,1] \times \R} I(\gamma(t,\lambda)). 
    \]
Since $\gamma\in\oGamma$ is arbitrary, \eqref{4.5} implies $b(\lambda)\leq \ob$ for all 
$\lambda\in\R$.

Let $\gamma =(\gamma_1,\gamma_2) \in\oGamma$ be of class $C^2$ and we show that for some 
dense subset $\Xi\subset\R$ there exists $\sigma_\lambda(t)\in\Lambda_\mu$
with \eqref{4.4} for each $\lambda\in\Xi$.
Since $\gamma_1(t,\lambda) = \lambda =\gamma_1(1-t,\lambda)$ hold 
for all $t \in [0, 1/ L_\gamma]$ and $\lambda \in \R$, in what follows, we identify 
$(0,\lambda)$ and $(1,\lambda)$ and regard $S^1\times\R \simeq ([0,1]/\{0, 1\})\times\R$.
Then $\gamma_1 \in C^2( S^1 \times \R,\R )$. 

We denote by $\Xi$ the set of regular values of 
$\gamma_1$, that is, $\Xi=\{v\in\R\,|\, D_{t,\lambda}\gamma_1(t,\lambda)\not=0 \ 
\hbox{if}\ \gamma_1(t,\lambda)=v\}$, where 
$D_{t,\lambda}=(\partial_t, \partial_\lambda)$.
By Sard's Theorem, the set of critical values $\R\setminus\Xi$ has Lebesgue measure 
$0$ in $\R$ and thus $\Xi$ is dense in $\R$.

For each $\lambda\in\Xi$, the pre-image $\gamma_1^{-1}(\lambda)$ is a $C^1$-curve in 
$S^1\times\R$ and
$\gamma_1^{-1}(\lambda)\cap (\{ 0\}\times\R) = \{ (0,\lambda)\}$.  Moreover 
$\gamma_1^{-1}(\lambda)$ is compact since $\gamma_1(t,\lambda) \to \pm \infty$ as 
$\lambda \to \pm \infty$. 
Thus $\gamma_1^{-1}(\lambda)$ is a finite union of closed smooth curves. 
We also note that $\gamma_1(t,\lambda)=\lambda$ in a neighborhood of $\{0\}\times\R$.
Thus $\gamma_1^{-1}(\lambda)$ has exactly one component which crosses $\{0\}\times \R$.
We denote it by $C$.
We remark that $C$ is a smooth curve joining $(0,\lambda)$ and $(1,\lambda)$.

We write $C$ as $c(t):[0,1]\to [0,1]\times \R$, where $c(t)$
satisfies
    \[  c(0)=(0,\lambda), \quad c(1)= (1,\lambda), 
        \quad c(t)\in (0,1)\times\R \quad\text{for all} \ t\in (0,1).
    \]
Then a path $\sigma_\lambda$ defined by $\sigma_\lambda(t)=\gamma_2(c(t))$ satisfies
$\gamma(c(t))=(\lambda,\sigma_\lambda(t))$.
It is easy to see $\sigma_\lambda\in \Lambda_\mu$ and \eqref{4.4} holds. 
Thus we have (ii).
\end{proof}

\begin{proof}[Proof of \cref{Proposition:4.2} (iii)]
By \cref{Lemma:3.3} (i) and (ii), for 
$\lambda \ll -1$, $\gamma_\lambda(t)=t T_0 \mu^{-1/(p+1)} \omega_\mu$ satisfies
$\gamma_\lambda\in \widehat{\Lambda}_\mu$ and
    \begin{equation} \label{4.6}
        b(\lambda)\leq \max_{t\in [0,1]} I(\lambda,\gamma_\lambda(t))\to 0
        \quad \text{as}\ \lambda\to -\infty.        
    \end{equation}
On the other hand, by the definition \eqref{4.1}--\eqref{4.2} of $b(\lambda)$, we 
observe $0\in \gamma([0,1])$ for all $\gamma\in \widehat\Lambda_\mu$, which 
implies $\max_{t\in [0,1]} I(\lambda,\gamma(t))\geq I(\lambda,0)=-\mu m_1$ for all 
$\gamma\in \widehat\Lambda_\mu$.  Thus
    \begin{equation} \label{4.7}
        b(\lambda) \geq  -\mu m_1 \quad
        \text{for all}\ \lambda\in\R.
    \end{equation}
Thus \eqref{4.6} and \eqref{4.7} imply $b(\lambda)\to 0$ as $\lambda\to -\infty$.
\end{proof}

\begin{Remark}\label{Remark:4.3}
For $L^2$ subcritical $g$ (e.g. $g(s)=\abs s^{q-1}s$, $1<q<p$), the function $b$
behaves differently and satisfies $b(\lambda)\to \infty$ as $\lambda\to \infty$.
See \cite{HT0}.
\end{Remark}

To close this section, we compute $\underline{b}$, $\overline{b}$ and $b(\lambda)$
for $G(s) = \frac{1}{p+1} \abs{s}^{p+1}$.

\begin{Lemma} \label{Lemma:4.4}
When $G(s)={\frac{1}{p+1}}\abs s^{p+1}$, we have
    \begin{equation*}
        \ub = \ob =0 \quad \text{and}\quad b(\lambda)=0 \ \text{for all}\ \lambda\in\R.
    \end{equation*}
\end{Lemma}

\begin{proof}
When $G(s)={\frac{1}{p+1}}\abs s^{p+1}$, we have $\alpha=0$ and $A=0$.  
Thus \eqref{3.18} implies $\ub=0$.  On the other hand, as stated in \cref{Remark:3.4}, 
$\zeta_0$ has a form $\zeta_0(\lambda)=C\xi(\lambda)\omega_\mu$.  
Thus $\gamma_0(t,\lambda)=(\lambda,t\zeta_0(\lambda))$
satisfies by \eqref{3.15}
    \begin{equation*}
        \sup_{(t,\lambda)\in [0,1]\times\R} I(\gamma_0(t,\lambda))=0.
    \end{equation*}
Together with \eqref{3.19}, we have $\ob=0$.
It follows from \cref{Proposition:4.2} (ii) that $b(\lambda)=0$ for 
all $\lambda\in\R$.
\end{proof}

\subsection{Least energy solutions of scalar field equations and $\ub$} 
\label{Section:4.2}
In this section we study the case where $b(\lambda_0)=\ub$ holds for some 
$\lambda_0\in\R$. In this case, we show the least energy solution of \eqref{3.7} 
with $\mu=e^{\lambda_0}$ gives a solution of \eqref{1.1} with $m=m_1$.  
Namely we show

\begin{Proposition} \label{Proposition:4.5}   
Assume that $\lambda_0\in\R$ satisfies $b(\lambda_0)=\ub$.  Then for any
critical point $u_0\in E$ of $u\mapsto I(\lambda_0,u)$ with 
$I(\lambda_0,u_0)=b(\lambda_0)=\ub$,
$(\lambda_0,u_0)\in\RE$ is a critical point of $I:\, \RE\to\R$.  That is, 
$(\mu_0,u_0)$, $\mu_0=e^{\lambda_0}$ is a solution of \eqref{1.1} with $m=m_1$.
Moreover there exists a positive critical point corresponding to $b(\lambda_0)=\ub$.
\end{Proposition}

\begin{proof}
Suppose that $b(\lambda_0)=\ub$ and let $u_0\in E$ be a critical point of 
$u\mapsto I(\lambda_0,u)$ with $I(\lambda_0,u_0)=b(\lambda_0)=\ub$.
We need to show $\partial_\lambda I(\lambda_0,u_0)=0$, which implies $(\lambda_0,u_0)$ 
is a critical point of $I$.

We argue indirectly and assume 
    \begin{equation} \label{4.8}
        \partial_\lambda I(\lambda_0,u_0)\not=0.    
    \end{equation}
By \cref{Proposition:3.1} (ii), there exists an optimal path $\gamma_0 \in
\Lambda_{\mu_0}$, $\mu_0=e^{\lambda_0}$ such that
    \begin{equation*}
        u_0 \in \gamma_0([0,1]), \quad
        \max_{t\in [0,1]} I(\lambda_0,\gamma_0(t)) = b(\lambda_0).
    \end{equation*}
Such an optimal path $\gamma_0$ is given explicitly in 
\cref{Proposition:A.2} in \cref{Appendix:A.2}.
Under assumption \eqref{4.8}, we deform the path
$t\mapsto (\lambda_0,\gamma_0(t))$ in $\RE$ in order to obtain a path $\widehat\gamma_0\in
\uGamma$ such that
    \begin{equation*}
        \ub \leq \max_{t\in [0,1]} I(\widehat\gamma_0(t)) 
        < \max_{t\in [0,1]} I(\lambda_0,\gamma_0(t))=b(\lambda_0),
    \end{equation*}
which contradicts the assumption $\ub=b(\lambda_0)$.
The explicit optimal path $\gamma_0(t)$ given in \cref{Proposition:A.2} enjoys 
the following property:
\begin{itemize}
\item[(i)] $I(\lambda_0,\gamma_0(t))\leq b(\lambda_0)$ for all $t\in [0,1]$;
\item[(ii)] there exists a compact set $K\subset E$ of a form:
    \begin{equation*}
        K=\begin{cases} 
            \{ u_0\}    &\text{for}\ N\geq 3,\\
            \{ u_{0\theta}\,|\, \theta\in [\theta_0,\theta_1]\} &\text{for}\ N=2,
        \end{cases}
    \end{equation*}
such that 
    \begin{equation*}
        I(\lambda_0,\gamma_0(t))=b(\lambda_0) \quad \text{if and only if}\ 
            \gamma_0(t)\in K.
    \end{equation*}
Here $\theta_0$, $\theta_1$ are constants such that $0<\theta_0<1<\theta_2$ and
$u_{0\theta}(x)=u_0(x/\theta)$.
\end{itemize}

\noindent
First we note that under assumption \eqref{4.8}
    \begin{equation} \label{4.9}
        D_{\lambda,u}I(\lambda_0,u)\not= 0 \quad \text{for all}\ u\in K,    
    \end{equation}
where $D_{\lambda,u}=(\partial_\lambda,\partial_u)$.
By \eqref{4.8}, clearly $D_{\lambda,u}(\lambda_0,u_0)\not= 0$ and \eqref{4.9} clearly holds
for $N\geq 3$.  For $N=2$, it suffices to show 
    \begin{equation} \label{4.10}
        \partial_u I(\lambda_0, u_{0\theta})\not=0 
        \quad \text{for}\ \theta\in [\theta_0,\theta_1]\setminus\{ 1\}. 
    \end{equation}
We compute
    \begin{equation*}
        \begin{aligned}
        \partial_u I(\lambda_0, u_{0\theta})u_{0\theta}
        &= \norm{\nabla u_{0\theta}}_2^2 +\mu_0\norm{u_{0\theta}}_2^2 
            -\intRN g(u_{0\theta} )u_{0\theta}\,dx \\
        &= \norm{\nabla u_0}_2^2 
            +\theta\left( \mu_0\norm{u_0}_2^2 -\intRN g(u_0)u_{0}\,dx\right) \\
        &= (1-\theta^2)\norm{\nabla u_0}_2^2>0. 
        \end{aligned}
    \end{equation*}
Here we used the fact that 
$\norm{\nabla u_0}_2^2+\mu_0\norm{u_0}_2^2 -\intRN g(u_0)u_0\,dx=0$.  Thus \eqref{4.10} 
holds, from which \eqref{4.9} follows.

Since $K$ is compact, we have 
$a=\inf_{\{ \lambda_0\}\times K} \norm{D_{\lambda,u}I(\lambda,u)}_{(\RE)^*}>0$.
Thus there exist small neighborhoods $N$, $N'$ of $\{ \lambda_0\}\times K$ in 
$\RE$ such that
\begin{itemize}
\item[(1)] $\{ \lambda_0\}\times K \subset N\subset N'$;
\item[(2)] $\gamma_0(0), \gamma_0(1) \not\in N'$;
\item[(3)] $\inf_{(\lambda,u)\in N} \norm{D_{\lambda,u}I(\lambda,u)}_{(\RE)^*} 
            \geq {\frac{a}{2}}$;
\item[(4)] there exists a locally Lipschitz continuous pseudo-gradient vector field
$X:\,\RE\to \RE$ such that
    \begin{equation*}
        \begin{aligned}
        &D_{\lambda,u}X(\lambda,u) \geq {\frac{a}{3}}\ \ 
                & &\text{for}\  (\lambda,u)\in N, \\
        &X(\lambda,u)=0     
                & &\text{for}\ (\lambda,u)\in (\RE)\setminus N', \\
        &\norm{X(\lambda,u)}_{\RE}\leq 1  & &\text{for}\ (\lambda,u)\in\RE. 
        \end{aligned}
    \end{equation*}
\end{itemize}
Let $\eta(\tau,\lambda,u):\, [0,1]\times\RE\to \RE$ be a flow associated to
the vector field $-X(\lambda,u)$ and define
    \begin{equation*}
            \wgamma_0(t)=\eta(1,\lambda_0,\gamma_0(t)):\, [0,1]\to \RE.
    \end{equation*}
It is easy to see 
    \begin{equation*}
        \begin{aligned}
        &\max_{t\in [0,1]} I(\wgamma_0(t)) 
            < \max_{t\in [0,1]} I(\lambda_0,\gamma_0(t))=b(\lambda_0), \\
        &\wgamma(i)=(\lambda_0,\gamma_0(i)) \quad \text{for}\ i=0,1. 
        \end{aligned}
    \end{equation*}
Joining paths $\theta\mapsto ((1-\theta)\lambda_1+\theta\lambda_0,0);\, [0,1]\to\RE$, 
where $\lambda_1\gg 1$, and $t\mapsto\wgamma_0(t):\, [0,1]\to \RE$ as in 
\cref{Proposition:4.2} (ii), we find a path $\widehat \gamma_0\in\uGamma$ such that
    \begin{equation*}
        \max_{t\in [0,1]} I(\widehat \gamma_0(t)) 
        = \max_{t\in [0,1]} I(\wgamma_0(t)) < b(\lambda_0).
    \end{equation*}
Thus we have $\ub<b(\lambda_0)$, which contradicts our assumption $\ub=b(\lambda_0)$.  
Thus \eqref{4.8} cannot take
a place and $(\lambda_0,u_0)$ is a critical point of $I$.  
The existence of a positive least energy solution corresponding to
$b(\lambda_0)$ is shown in \cite{BeL} (c.f. \cite{JT0}).
\end{proof}

\subsection{Proof of \cref{Theorem:1.1} for the case: $\ub=\ob=0$}\label{Section:4.3}
\cref{Proposition:4.5} has an interesting application; 
it may be exploited to prove \cref{Theorem:1.1} in a special case $\ub = \ob = 0$. 
It is worthwhile to note that the $(PSPC)_b$ condition does not hold in this case 
and hence our deformation results in \cref{Section:7,Section:8} are not applicable.

\begin{Proposition} \label{Proposition:4.6}
Assume \ref{(g0')}, \ref{(g1)} and $\ub=\ob=0$. Then for each $\lambda\in\R$, 
any critical point $u_\lambda\in E$ of $u\mapsto I(\lambda,u)$ with 
$I(\lambda,u_\lambda)=b(\lambda)$ is a critical point of $I:\,\RE\to\R$.
In other words, for each $\lambda\in\R$, any least energy solution
$u_\lambda$ of $-\Delta u+\mu u=g(u)$ in $\R^N$ with $\mu=e^\lambda$ 
satisfies $\half\norm{u_\lambda}_2^2=m_1$, that is, $u_\lambda$ is 
a solution of \eqref{1.1} with $m=m_1$.  
\end{Proposition}

\begin{proof}
The assumption $\ub = \ob = 0$ and \cref{Proposition:4.2} (i) yield $b(\lambda)=\ub=0$ 
for all $\lambda\in\R$. 
Thus \cref{Proposition:4.5} is applicable and \cref{Proposition:4.6} follows from 
\cref{Proposition:4.5}.
\end{proof}

\begin{Remark}\label{Remark:4.7}
A typical example, for which $\ub=\ob=0$ holds, is $g(s)=\abs s^{p-1}s$.  
It remains an open problem to find an odd $g$ which is different 
from $\abs s^{p-1}s$ and for which $\ub=\ob=0$ holds.
\end{Remark}

\subsection{Proof of \cref{Proposition:4.2} (iv)} \label{Section:4.4}

\begin{proof}[Proof of \cref{Proposition:4.2} (iv)]
We note that for each $\lambda\gg\lambda_0$ a path
$t\mapsto tT_1\omega_\mu;\ [0,1]\to E$ 
belongs to $\widehat \Lambda_\mu$, which is defined in \eqref{4.2}.
We have by \cref{Lemma:3.3} (iv) 
    \begin{equation*}
        b(\lambda) \leq \max_{t\in [0,1]} I(\lambda,tT_1\omega_\mu)\to 0
        \quad \text{as}\ \lambda\to\infty,
    \end{equation*}
which implies $\limsup_{\lambda\to \infty} b(\lambda)\leq 0$.  
Thus we need to show $\liminf_{\lambda\to\infty} b(\lambda)\geq 0$.

Suppose that $\lambda_j\to\infty$ as $j\to\infty$.  By \cref{Proposition:4.2}(ii), 
$(b(\lambda_j))_{j=1}^\infty$ is a bounded sequence.  After extracting a subsequence, we may 
assume $b(\lambda_j)\to b$ for some $b\in\R$ and we need to show $b\geq 0$.

Let $u_j\in E$ be the least energy solution corresponding to $b(\lambda_j)$.  Then
$u_j$ satisfies
    \begin{equation*}
        I(\lambda_j,u_j)=b(\lambda_j)\to b, \quad
        \partial_u I(\lambda_j,u_j)=0, \quad P(\lambda_j,u_j)=0.
    \end{equation*}
We introduce $L$, $Q:\, \RE\to\R$ by
    \begin{align}
        L(\lambda,v) &= \half\norm{\nabla v}_2^2 +\half\norm v_2^2 
            -{\frac{1}{p+1}}\norm v_{p+1}^{p+1}
            -\mu^{-1-N/2}\intRN H(\mu^{N/4}v)\, dx, \label{4.11}\\
        Q(\lambda,v) &= {\frac{N-2}{2}}\norm{\nabla v}_2^2 
            + N\left\{ \half\norm v_2^2 -{\frac{1}{p+1}}\norm v_{p+1}^{p+1}
            -\mu^{-1-N/2}\intRN H(\mu^{N/4}v)\, dx\right\} \nonumber
    \end{align}
and $v_j\in E$ by $u_j(x)=\mu_j^{N/4}v_j(\mu_j^{1/2}x)$.  We have
    \begin{align}   
        &I(\lambda_j,u_j) = \mu_j\{ L(\lambda_j,v_j)-m_1\}, \label{4.12}\\
        &\partial_v L(\lambda_j,v_j) =0,                        \label{4.13}\\
        &Q(\lambda_j,v_j) =0.                                   \label{4.14}
    \end{align}
It follows from \eqref{4.12} that
    \begin{equation} \label{4.15}
        L(\lambda_j,v_j)\to m_1 \quad \text{as} \ j\to\infty.       
    \end{equation}
Using $(v_j)_{j=1}^\infty$, we show $\liminf_{j\to\infty} b(\lambda_j)\geq 0$ in 3 steps.

\smallskip

\noindent
\textbf{Step 1:} \textsl{$(v_j)_{j=1}^\infty$ is bounded in $E$.}

\smallskip

\noindent
We show Step 1 using an idea in \cite[Lemma 5.1]{HIT}.
We give a proof just for $N\geq 3$.  It follows from \eqref{4.14} and \eqref{4.15}
that
    \begin{equation*}
        \norm{\nabla v_j}_2^2 = NL(\lambda_j,v_j)-Q(\lambda_j,v_j) = Nm_1 +o(1)
    \end{equation*}
and $\norm{\nabla v_j}_2^2$ stays bounded, from which the boundedness of $\norm{v_j}_{2^*}$ 
follows.  We also have 
    \begin{equation} \label{4.16}
        \half\norm{v_j}_2^2 -{\frac{1}{p+1}}\norm{v_j}_{p+1}^{p+1} 
        - \mu_j^{-1-N/2}\intRN H(\mu_j^{N/4} v_j)\, dx
        = L(\lambda_j,v_j)-\half\norm{\nabla v_j}_2^2
        = -{\frac{N-2}{2}} m_1 +o(1).           
    \end{equation}
By \eqref{2.1}, we have
    \begin{equation} \label{4.17}
        \pabs{ \mu_j^{-1-N/2}\intRN H(\mu_j^{N/4}v_j)\, dx} 
        \leq \mu_j^{-1} A\norm{v_j}_2^2.        
    \end{equation}
Thus it follows from \eqref{4.16}, \eqref{4.17} and 
the Gagliardo-Nirenberg inequality \eqref{2.10} that 
    \begin{equation*}
        \left(\half-\mu_j^{-1}A\right) \norm{v_j}_2^2 
            - \half (2m_1)^{-2/N}\norm{\nabla v}_2^2\norm v_2^{4/N}
        \leq -{\frac{N-2}{2}} m_1 +o(1).
    \end{equation*}
Since $\mu_j\to \infty$ as $j\to\infty$, we have boundedness of $\norm{v_j}_2$.

For $N=2$, we can modify the argument in \cite[Lemma 5.1]{HIT}.

\smallskip

\noindent
\textbf{Step 2:} \textsl{After extracting a subsequence if necessary, we have
    \begin{equation*}
        v_j\to \omega_1 \quad \text{or}\quad v_j\to -\omega_1 
        \quad \text{strongly in} \ E.
    \end{equation*}
}

\smallskip

\noindent
By Step 1, after extracting a subsequence, we may assume $v_j\wlimit v_0$ weakly in
$E$ for some $v_0\in E$.
By \eqref{2.1}, we have
    \begin{align}
        &\mu_j^{-1-N/4} \norm{h(\mu_j^{N/4} v_j)}_2 \leq \mu_j^{-1} A\norm{v_j}_2
            \to 0, \nonumber\\
        &\pabs{\mu_j^{-1-N/2} \intRN H(\mu_j^{N/4} v_j)\, dx}  
            \leq \mu_j^{-1} A\norm{v_j}_2^2
            \to 0 \quad \text{as}\ j\to\infty. \label{4.18}
    \end{align}
Thus \eqref{4.13} and \eqref{4.15} imply
    \begin{equation*}
        \Psi_{01}(v_j) \to m_1 \quad \text{and}\ \quad
        \Psi_{01}'(v_j)\to 0 \ \text{strongly in}\ E^*,
    \end{equation*}
where $\Psi_{01}$ is defined in \eqref{2.8}. 
Since $\Psi_{01}$ satisfies the Palais-Smale condition,
we have $v_j\to v_0$ strongly in $E$ and $\Psi_{01}(v_0)=m_1$. 
That is, $v_0$ is a least energy critical point of $\Psi_{01}$ and we have 
$v_0=\pm \omega_1$.

In what follows, we assume $v_j\to\omega_1$ strongly in $E$ 
since the other case may be treated similarly. 

\smallskip

\noindent
\textbf{Step 3:} \textsl{$\liminf_{j\to\infty} b(\lambda_j) \geq 0$}

\smallskip

\smallskip

\noindent
We recall that $v\mapsto L(\lambda,v)$ and $v\mapsto\Psi_{01}(v)$ enjoy the 
mountain pass geometry and least energy solutions have optimal paths, which are 
given explicitly in \cref{Proposition:A.2}. 
When $N=2$, setting $\omega_{1\theta}(x)=\omega_1(x/\theta)$, an optimal 
path $\gamma_\infty$ for $\Psi_{01}$ is given joining the following paths:
    \[
    \text{(a)}\ [0,1]\to E;\, t\mapsto t\omega_{1\theta_0}; \quad
    \text{(b)}\ [\theta_0,\theta_1]\to E;\, \theta\mapsto \omega_{1\theta}; \quad
    \text{(c)}\ [1,t_1]\to E;\, t\mapsto t\omega_{1\theta_1}.
    \]
Here $0<\theta_0\ll 1\ll \theta_1$ and $t_1>1$ is close to $1$.
Moreover the end point $\gamma_\infty(1)=t_1\omega_{1\theta_1}$ of this path 
satisfies $\Psi_{01}(t_1\omega_{1\theta_1})<0$.
Noting $v_j\to \omega_1$ strongly in $E$ and \eqref{4.18},
we can find an optimal path $\gamma_j$ for $L(\lambda_j,\cdot)$ 
replacing $\omega_1$ with $v_j$.  We note that \eqref{A.10}--\eqref{A.13} in
\cref{Proposition:A.2} holds for $\gamma_j$ replacing $L(\cdot)$ with 
$L(\lambda_j,\cdot)$ and $u_0$ with $v_j$.

For $N\geq 3$, the optimal path $\gamma_j$ for $L(\lambda_j,\cdot)$ is given by 
$\gamma_j(t)=v_j({\frac{x}{Tt}})$ for a suitably chosen $T\gg 1$, which is 
independent of $j$.

Since $\gamma_j$ is an optimal path for $I(\lambda_j,\cdot)$, 
we have $b(\lambda_j)=\max_{t\in [0,1]} I(\lambda_j,\gamma_j(t))$.
We compute that
    \begin{align*}
    b(\lambda_j) &= \max_{t\in [0,1]} I(\lambda_j,\gamma_j(t)) 
        = \max_{t\in [0,1]} \mu_j\{ L(\lambda_j,\gamma_j(t))-m_1\}  \\
        &= \max_{t\in [0,1]} \left( \mu_j\{ \Psi_{01}(\gamma_j(t))-m_1\} 
                - \mu_j^{-N/2} \intRN  H(\mu_j^{N/4}\gamma_j(t))\, dx\right)\\
        &\geq \mu_j \max_{t\in [0,1]} \{ \Psi_{01}(\gamma_j(t))-m_1\} 
            - \max_{t\in [0,1]} \mu_j^{-N/2} \intRN H(\mu_j^{N/4}\gamma_j(t))\, dx\\
        &= \mu_j (I) - (II). 
    \end{align*}
We note that $(I)\geq 0$ for large $j$.  In fact, since 
    \[  \Psi_{01}(\gamma_j(1)) = \begin{cases}
        \Psi_{01}(t_1v_j(x/\theta_1))\to \Psi_{01}(t_1\omega_{1\theta_1})<0 
            &\text{for}\ N=2,\\ 
        \Psi_{01}(v_j(x/T)) \to \Psi_{01}(\omega_1(x/T)) <0
            &\text{for}\ N\geq 3
        \end{cases}
    \]
as $j\to \infty$, we can see
that $t\mapsto \gamma_j(t)$ is a sample path for $\Psi_{01}$ for large $j$.
Recall that the mountain pass level for $\Psi_{01}$ is $m_1$, we have 
$\max_{t\in [0,1]} \Psi_{01}(\gamma_j(t))\geq m_1$ for large $j$ and thus 
$(I)\geq 0$ for large $j$.

Next we claim that $(II)\to 0$.  In fact, noting that $K=\set{\gamma_j(t) | 
t\in [0,1],\, j\in\N}$ is precompact in $L^2(\R^N)$, we have by \cref{Lemma:A.1} 
    \begin{equation*}
        (II)  \leq 
        \max_{t\in [0,1]} \mu_j^{-N/2} \intRN \abs{H(\mu_j^{N/4}\gamma_j(t))}\, dx 
        \leq \max_{u\in \overline K} \mu_j^{-N/2}  \intRN \abs{H(\mu_j^{N/4}u)}\, dx 
        \to 0 \quad \text{as}\ j\to\infty. 
    \end{equation*}
Thus we have $\liminf_{j\to\infty} b(\lambda_j) \geq 0$.
This completes the proof of (iv).
\end{proof}

\section{The Palais-Smale-Pohozaev-Cerami condition} \label{Section:5}
To show \cref{Theorem:1.1}, a Palais-Smale type compactness condition plays
an important role. In this section we introduce the \emph{Palais-Smale-Pohozaev-Cerami} 
((\emph{PSPC}) in short) condition, under
which we will develop a new deformation argument in \cref{Section:7}.

\subsection{Palais-Smale-Pohozaev-Cerami type conditions}
First we give notion of a $(PSPC)_b$ sequence and the $(PSPC)_b$ condition.

\begin{Definition}\label{Definition:5.1}
\begin{itemize}
\item[(i)] 
For $b\in\R$, we say that $(\lambda_j,u_j)_{j=1}^\infty\subset\RE$ is a 
\emph{$(PSPC)_b$ sequence} for $I $ if the following conditions are satisfied:
    \begin{align}   
        &I(\lambda_j,u_j)\to b,                                         \label{5.1}\\
        &\partial_\lambda I(\lambda_j,u_j) \to 0,                       \label{5.2}\\
        &(1+\norm{u_j}_E)\norm{\partial_u I(\lambda_j,u_j)}_{E^*}\to 0, \label{5.3}\\
        &P(\lambda_j,u_j)\to 0,                                         \label{5.4}
    \end{align}
where $P$ is the Pohozaev functional associated to $I$, i.e., a functional defined 
in \eqref{2.3}.
\item[(ii)] 
We also say that $I $ satisfies the \emph{$(PSPC)_b$ condition} if any $(PSPC)_b$ 
sequence has a strongly convergent subsequence in $\RE$.
\end{itemize}
\end{Definition}

\begin{Remark}\label{Remark:5.2}
Under the $(PSPC)_b$ condition, the following critical set $K_b$ is compact in $\RE$.
    \begin{equation*}
        K_b =\Set{ (\lambda,u)\in\RE | I(\lambda,u)=b, 
            \ \partial_\lambda I(\lambda,u)=0,\
            \partial_uI(\lambda,u)=0, \ P(\lambda,u)=0}.
    \end{equation*}
For our functional $I$, we note that the Pohozaev identity $P(\lambda,u)=0$ holds
if $(\lambda,u)$ satisfies $\partial_u I(\lambda,u)=0$.
Thus for the definition of the critical set $K_b$, it is not necessary to assume
the property $P(\lambda,u)=0$.  However our deformation argument in \cref{Section:7} 
works for more general situation, where $\partial_u I(\lambda,u)=0$
may not imply the Pohozaev identity (e.g. fractional problems).  Moreover in our 
deformation theory, the above critical set $K_b$ with the assumption $P(\lambda,u)=0$ 
plays an essential role (See \cref{Proposition:7.1}).  Thus we include $P(\lambda,u)=0$
in our definition of $K_b$.
\end{Remark}

To study the existence of a positive solution, we set
    \begin{equation} \label{5.5}
        \scrP_0=\Set{ u\in E |  u\geq 0 \ \text{in}\ \R^N}
    \end{equation}
and introduce $(PSPC)_b^*$ sequences and the $(PSPC)_b^*$ condition as follows:

\begin{Definition} \label{Definition:5.3}  
\begin{itemize}
\item[(i)] For $b\in\R$ we say that $(\lambda_j,u_j)_{j=1}^\infty\subset\RE$ is 
a \emph{$(PSPC)_b^*$ sequence} for $I$ if it satisfies
    \begin{equation*}
        \distE(u_j,\scrP_0)\to 0 \quad \text{as}\ j\to\infty
    \end{equation*}
in addition to \eqref{5.1}--\eqref{5.4}.
\item[(ii)] We also say that $I$ satisfies the \emph{$(PSPC)_b^*$ condition} if any
$(PSPC)_b^*$ sequence has a strongly convergent subsequence in $\RE$.
\end{itemize}
\end{Definition}

The functional $I $ defined in \eqref{2.2} has the following Palais-Smale type 
compactness result.

\begin{Proposition} \label{Proposition:5.4}
Assume \ref{(g0')} and \ref{(g1)}.  Then we have 
\begin{itemize}
\item[\rm (i)] for $b<0$, the $(PSPC)_b$ condition holds for $I$; 
\item[\rm (ii)] for $b>0$, if \ref{(g2)} holds in addition, then the $(PSPC)_b^*$ condition 
holds for $I$.
\end{itemize}
\end{Proposition}

\begin{Remark}\label{Remark:5.5}
For $b=0$, neither the $(PSPC)_0$ condition or the $(PSPC)_0^*$ condition holds 
under \ref{(g1)}. 
In fact, for any $(\lambda_j)_{j=1}^\infty$ with $\lambda_j \to -\infty$, 
$(\lambda_j,\omega_{\mu_j})_{j=1}^\infty$, $\mu_j=e^{\lambda_j}$ is a $(PSPC)_0^*$ 
sequence, which does not have any convergent subsequence.
This fact follows from \eqref{2.9} and 
    \[  \intRN \abs{H( \omega_{\mu_j} )} \, dx 
            \leq C \norm{\omega_{\mu_j}}_{p+1}^{p+1} \to 0, \quad 
        \intRN \abs{h(\omega_{\mu_j}) \varphi } \, dx 
            \leq C \norm{\omega_{\mu_j}}_{p+1}^p \norm{\varphi}_{p+1}\to 0.
    \]
\end{Remark}

\cref{Proposition:5.4} is one of the keys in the proof of \cref{Theorem:1.1}.
We will give a proof of \cref{Proposition:5.4} in the following 
Sections \ref{Section:5.2}--\ref{Section:5.4}.
For this purpose, we note that
for a $(PSPC)_b$ sequence $(\lambda_j,u_j)_{j=1}^\infty\subset\RE$, 
$\mu_j=e^{\lambda_j}$, it follows from \eqref{5.1}, \eqref{5.2} and \eqref{5.4} that
    \begin{align}   
    &\half\norm{\nabla u_j}_2^2-\intRN G(u_j)\,dx
        +\mu_j\left\{\half\norm{u_j}_2^2-m_1\right\}\to b,  \label{5.6}\\
    &\mu_j\left\{\half\norm{u_j}_2^2-m_1\right\}\to 0,  \label{5.7}\\
    &{\frac{N-2}{2}}\norm{\nabla u_j}_2^2 
        +N\left\{{\frac{\mu_j}{2}}\norm{u_j}_2^2-\intRN G(u_j)\,dx\right\}\to 0.
            \label{5.8}
    \end{align}
For a $(PSPC)_b^*$ sequence $(\lambda_j,u_j)_{j=1}^\infty$, we have 
in addition to \eqref{5.6}--\eqref{5.8}
    \[  \distE(u_j,\scrP_0)\to 0.
    \]
From \eqref{5.6}--\eqref{5.8}, we infer that 
    \begin{equation}\label{5.9}
        \half\norm{\nabla u_j}_2^2 -\intRN G(u_j)\, dx \to b, \quad 
        {\frac{N-2}{2}}\norm{\nabla u_j}_2^2-N\intRN G(u_j)\, dx + N\mu_j m_1\to 0. 
    \end{equation}
By \eqref{5.3}, we also have $\partial_u I(\lambda_j,u_j)u_j \to 0$, that is,
    \begin{equation} \label{5.10}
        \norm{\nabla u_j}_2^2 +\mu_j\norm{u_j}_2^2 -\intRN g(u_j)u_j \,dx\to 0.
    \end{equation}
Thus, \eqref{5.7} and \eqref{5.9}--\eqref{5.10} give 
    \begin{align}   
    &\mu_j \norm{u_j}_2^2 = 2\mu_j m_1 +o(1),       \label{5.11}\\
    &\norm{\nabla u_j}_2^2 =N\mu_j m_1 +Nb +o(1),   \label{5.12}\\
    &\intRN G(u_j)\, dx={\frac{N}{2}}\mu_j m_1+{\frac{N-2}{2}}b+o(1), \label{5.13}\\
    &\intRN g(u_j)u_j\, dx=(N+2)\mu_j m_1+Nb+o(1).      \label{5.14}
    \end{align}
In following Sections \ref{Section:5.2}--\ref{Section:5.4}, we consider 
the following three cases separately.

\begin{description}
\item[\textbf{Case A}] $\lambda_j\to \lambda_0$ as $j\to\infty$ for some $\lambda_0\in\R$;
\item[\textbf{Case B}] $\lambda_j\to+\infty$ as $j\to\infty$;
\item[\textbf{Case C}] $\lambda_j\to-\infty$ as $j\to\infty$.
\end{description}

\noindent 
Note that one of 3 cases happens after extracting a subsequence.

\subsection{Proof of \cref{Proposition:5.4}: Case A ($\lambda_j$ bounded)} 
\label{Section:5.2} 
\begin{Proposition} \label{Proposition:5.6}
Assume \ref{(g0')} and \ref{(g1)}. For $b\in \R$, suppose that 
$(\lambda_j,u_j)_{j=1}^\infty$ is 
a $(PSPC)_b$ sequence or $(PSPC)_b^*$ sequence with bounded 
$(\lambda_j)_{j=1}^\infty$. 
Then $(\lambda_j,u_j)_{j=1}^\infty$ has a strongly convergent subsequence in $\RE$.
\end{Proposition}

\begin{proof}
Since $(PSPC)_b^*$ sequences are $(PSPC)_b$ sequences, we show the compactness 
of sequences just for $(PSPC)_b$ sequences. 

Since $\mu_j=e^{\lambda_j}$ is bounded away from $0$ and $\infty$ as $j\to\infty$, 
the boundedness of $(u_j)_{j=1}^\infty$ in $E$ follows from \eqref{5.11} and 
\eqref{5.12}. 
Thus we can extract a subsequence --- still denoted by $(u_j)_{j=1}^\infty$ --- 
such that for some $u_0\in E$
    \begin{equation*}
        u_j\wlimit u_0 \quad \text{weakly in $E$ and strongly in $L^{p+1}(\R^N)$}.
    \end{equation*}
Since $\partial_u I(\lambda_j,u_j)u_j\to 0$, we have
    \begin{equation*}
        \begin{aligned}
        \norm{\nabla u_j}_2^2 + \mu_j\norm{u_j}_2^2
        &= \partial_u I(\lambda_j,u_j)u_j +\intRN g(u_j)u_j\, dx 
        \to \intRN g(u_0)u_0\, dx. 
        \end{aligned}
    \end{equation*}
Similarly, it follows from $\partial_u I(\lambda_j,u_j)u_0\to 0$ that
    \begin{equation*}
        \norm{\nabla u_0}_2^2 + \mu_0\norm{u_0}_2^2=\intRN g(u_0)u_0\, dx.
    \end{equation*}
Thus 
$\norm{\nabla u_j}_2^2+\mu_0\norm{u_j}_2^2\to \norm{\nabla u_0}_2^2+\mu_0\norm{u_0}_2^2$
as $j\to\infty$ and the strong convergence $u_j\to u_0$ in $E$ follows. 
\end{proof}

A proof of \cref{Proposition:5.4} is completed when Case A happens.

\subsection{Proof of \cref{Proposition:5.4}: Case B ($\lambda_j \to +\infty$)} 
\label{Section:5.3}
In this section we consider the situation where $\lambda_j\to+\infty$, i.e. 
$\mu_j = e^{\lambda_j} \to +\infty$. We show the following proposition.

\begin{Proposition} \label{Proposition:5.7}
Assume \ref{(g0')} and \ref{(g1)}. For $b\in\R$, suppose that 
$(\lambda_j,u_j)_{j=1}^\infty$ is 
a $(PSPC)_b$ sequence or a $(PSPC)_b^*$ sequence with $\lambda_j\to+\infty$. 
Then $b$ must be $0$.
\end{Proposition}

Clearly \cref{Proposition:5.4} in Case B follows from \cref{Proposition:5.7}.
\begin{proof}
As in the proof of \cref{Proposition:5.6}, we only treat $(PSPC)_b$ sequences. 
We divide our arguments into several steps.  Arguments for Steps 1--2 are closely
related to those in the proof of  \cref{Proposition:4.2} (iv)
in \cref{Section:4.4}.

\smallskip 

\noindent
\textbf{Step 1:} \textsl{We define $v_j$ by $u_j(x)=\mu_j^{N/4} v_j(\mu_j^{1/2}x)$. 
Then as $j\to\infty$,
    \begin{align}   
        &\text{$(v_j)_{j=1}^\infty$ is bounded in $E$,}     \label{5.15}\\
        &L(\lambda_j,v_j)\to m_1,                       \label{5.16}\\
        &\norm{\partial_v L(\lambda_j,v_j)}_{E^*} \to 0,    \label{5.17}
    \end{align}
where $L:\, \RE\to\R$ is defined by \eqref{4.11}.
}

\smallskip 

\noindent
Remark that 
    \begin{equation*}
        \norm{\nabla u_j}_2^2=\mu_j\norm{\nabla v_j}_2^2, \quad \norm{u_j}_2^2 
        = \norm{v_j}_2^2.
    \end{equation*}
Hence, \eqref{5.15} follows from \eqref{5.11}, \eqref{5.12} and $\mu_j \to \infty$. 
We also note that 
    \begin{equation*}
        b+o(1) = I(\lambda_j,u_j)
            = \mu_j \left\{ L(\lambda_j, v_j) - m_1 \right\},
    \end{equation*}
which yields \eqref{5.16}. 
Setting $\varphi_j(x) =\mu_j^{N/4}\psi(\mu_j^{1/2}x)$ for $\psi\in E$, we have
    \begin{equation} \label{5.18}
        \partial_u I(\lambda_j,u_j)\varphi_j =\mu_j\partial_v L(\lambda_j,v_j)\psi, 
        \quad \norm{ \nabla \varphi_j }_2^2 = \mu_j \norm{ \nabla \psi }_2^2, \quad 
        \norm{ \varphi_j }_2^2 = \norm{ \psi }_2^2,
    \end{equation}
which implies $\norm{\varphi_j}_E\leq \mu_j^{1/2} \norm\psi_E$. Thus it follows from 
\eqref{5.2} and \eqref{5.18} that
    \begin{equation*}
            \begin{aligned}
        \norm{\partial_v L(\lambda_j,v_j)}_{E^*} 
        &= \sup_{\norm\psi_E\leq 1} \abs{\partial_v L(\lambda_j,v_j)\psi} 
        \leq \sup_{\norm\varphi_E\leq \mu_j^{1/2}} 
            \mu_j^{-1}\pabs{\partial_u I(\lambda_j,u_j)\varphi} \\
        &=\mu_j^{-1/2}\norm{\partial_u I(\lambda_j,u_j)}_{E^*} \to 0. 
        \end{aligned}
    \end{equation*}
Therefore \eqref{5.17} holds. 

\smallskip 

\noindent
\textbf{Step 2:} \textsl{After extracting a subsequence, we have
    \begin{equation*}
        v_j\to \omega_1 \quad \text{or}\quad v_j\to -\omega_1 \quad
        \text{strongly in $E$ as $j\to\infty$.}
    \end{equation*}
}

\smallskip 

\noindent
Step 2 can be shown as in Step 2 of the proof of \cref{Proposition:4.2} (iv).

In what follows, we assume $v_j\to\omega_1$ strongly in $E$ since the argument is 
similar when $v_j \to - \omega_1$. 

\smallskip

\noindent
\textbf{Step 3:} \textsl{$\displaystyle \intRN (p+1)G(u_j)-g(u_j)u_j\, dx\to 0$ 
as $j\to\infty$.}

\smallskip

\noindent
Under \ref{(g1)}, we have
${\frac{1}{s^2}}((p+1)G(s)-g(s)s) = {\frac{1}{s^2}}((p+1)H(s)-h(s)s)$ for all 
$s\in\R$.  Thus by \cref{Lemma:A.1}, 
    \begin{align*}  
        &\intRN (p+1)G(u_j)-g(u_j)u_j\, dx 
        = \intRN (p+1)H(u_j)-h(u_j)u_j\, dx\nonumber\\
        &\qquad = \mu_j^{-N/2}\intRN (p+1)H(\mu_j^{N/4}v_j) 
            -h(\mu_j^{N/4}v_j)\mu_j^{N/4}v_j\, dx\to 0.
    \end{align*}
Hence, Step 3 holds. 

\smallskip

\noindent
\textbf{Step 4:} \textsl{Conclusion.}

\smallskip

\noindent
We show that $b$ appeared in \eqref{5.1} must be $0$. 
By \eqref{5.13} and \eqref{5.14}, we have
    \begin{equation*}
        \intRN (p+1)G(u_j)-g(u_j)u_j\, dx
        = \left( (p+1){\frac{N-2}{2}} -N\right) b +o(1)
        = -{\frac{4}{N}}b + o(1).
    \end{equation*}
Thus, Step 3 yields $b=0$ and we complete the proof. 
\end{proof}

\subsection{Proof of \cref{Proposition:5.4}: Case C ($\lambda_j \to -\infty$)} 
\label{Section:5.4}
In this section, we consider the situation when $\lambda_j\to -\infty$, i.e.,
$\mu_j = e^{\lambda_j} \to 0$. 
Here \eqref{1.15} appears to describe a limit profile of $(PSPC)_b$ or $(PSPC)_b^*$ 
sequences.  We show the following proposition.

\begin{Proposition} \label{Proposition:5.8}
Assume \ref{(g0')} and {\rm \ref{(g1)}}. For $b\in\R$, suppose that 
$(\lambda_j,u_j)_{j=1}^\infty$ is a $(PSPC)_b$ or $(PSPC)_b^*$ sequence 
with $\lambda_j\to -\infty$. Then 
there exist a subsequence of $(u_j)_{j=1}^\infty$ --- still denoted by $u_j$ --- 
and $u_0\in F$ such that
    \begin{equation*}
        u_j\to u_0 \ \text{strongly in}\ F \quad \text{and}\quad
        Z(u_0) = b, \quad Z'(u_0) = 0,      
    \end{equation*}
where $Z$ is defined in \eqref{2.11}.
For a $(PSPC)_b^*$ sequence, we also have $u_0\geq 0$.
\end{Proposition}

We first prove the boundedness of $(u_j)_{j=1}^\infty$
in $L^{p+1}(\R^N)$.

\begin{Lemma} \label{Lemma:5.9} 
Assume \ref{(g0')} and {\rm \ref{(g1)}}. For $b\in\R$, suppose that 
$(\lambda_j,u_j)_{j=1}^\infty\subset\RE$ is a $(PSPC)_b$ or $(PSPC)_b^*$ sequence with
$\lambda_j\to -\infty$. Then, $(u_j)_{j=1}^\infty$ is bounded in $F$.
\end{Lemma}

\begin{proof}
We argue for a $(PSPC)_b$ sequence.
Since $\mu_j\to 0$, it follows from \eqref{5.11} that
    \begin{equation} \label{5.19}
        \mu_j\norm{u_j}_2^2 \to 0 \quad \text{as}\ j\to\infty.  
    \end{equation}
We also see from \eqref{5.12}--\eqref{5.14} that
    \begin{equation} \label{5.20}
        \norm{\nabla u_j}_2^2 \to Nb, \quad
        \intRN G(u_j)\, dx \to {\frac{N-2}{2}}b, \quad
        \intRN g(u_j)u_j\, dx \to Nb.       
    \end{equation}
Since $\norm{\nabla u_j}_2$ is bounded, it suffices to show boundedness of 
$(u_j)_{j=1}^\infty$ in $L^{p+1}(\R^N)$.  We show it in 4 steps.
For notational convenience, we regard $Z$ as a map $E\to\R$ in this proof.

\smallskip
\noindent
\textbf{Step 1:} \textsl{ $(1+\norm{u_j}_E)\norm{Z'(u_j)}_{E^*}\to 0$ as $j\to\infty$.}

\smallskip

\noindent
For $\varphi\in E$, we have
    \begin{equation*}
        \begin{aligned}
            (1+\norm{u_j}_E)\pabs{Z'(u_j)\varphi} 
        &= (1+\norm{u_j}_E)\pabs{ \partial_u I(\lambda_j,u_j) \varphi 
            -\mu_j(u_j,\varphi)_2} \\
        &\leq (1+\norm{u_j}_E)\left\{ \norm{\partial_u I(\lambda_j,u_j)}_{E^*}
            +\mu_j\norm{u_j}_2\right\}\norm{\varphi}_E. 
        \end{aligned}
    \end{equation*}
By \eqref{5.3}, \eqref{5.19} and boundedness of $\norm{\nabla u_j}_2$,  we have Step 1.

\smallskip

By Step 1, we have
    \begin{equation} \label{5.21}
        (1+\norm{u_j}_E) \norm{Z'(u_j)}_{E^*} \leq 1 \quad \text{for large}\ j.
    \end{equation}
From \ref{(g1)}, there exist constants $C_1$, $C_2$, $\delta>0$ such that
    \begin{align}   
    &g(s)s \geq C_1\abs s^{p+1} \quad \text{for}\ \abs s\leq 2\delta, 
                                    \label{5.22}\\
    &g(s)s \geq C_1\abs s^{p+1} -C_2 \quad \text{for}\ s\in\R. 
                                    \label{5.23}
    \end{align}


\noindent
We deal with cases $N=2$ and $N\geq 3$ separately.

\smallskip

\noindent
\textbf{Step 2:} \textsl{When $N\geq 3$, $\norm{u_j}_{p+1}$ is bounded.}

\smallskip

\noindent
When $N\geq 3$, \cref{Lemma:2.4} (i) ensures the existence of $R_0>0$ independent
$j$ such that
    \begin{equation*}
        \abs{u_j(x)} \leq \delta \quad \text{for all}\ \abs x\geq R_0
        \ \text{and}\ j\in\R^N.
    \end{equation*}
By \eqref{5.22}--\eqref{5.23}, we have
    \begin{align*}  
        C_1\norm{u_j}_{p+1}^{p+1}
        &\leq \int_{\R^N\setminus B_{R_0}} g(u_j)u_j\, dx 
            + \int_{B_{R_0}} g(u_j)u_j +C_2\, dx\nonumber\\
        &= \intRN g(u_j)u_j\, dx + C_2\meas(B_{R_0})\, dx.
    \end{align*}
Hence, by \eqref{5.20}, $\norm{u_j}_{p+1}$ is bounded.  Thus Step 2 is proved.

\smallskip

When $N=2$, we note that $p+1=4$. As in \cite{HIT}, we use a scaling argument.
We argue indirectly and assume $\norm{u_j}_4\to\infty$.  
Setting
$t_j=\norm{u_j}_4^{-2}\to 0$ and $v_j(x)=u_j(x/t_j)\in E$, it follows from 
\eqref{5.20} that
    \begin{align}   
        &\norm{\nabla v_j}_2^2=\norm{\nabla u_j}_2^2\to 2b, \quad
        \norm{v_j}_4^4 =t_j^2\norm{u_j}_4^4=1, \quad
        \mu_j\norm{v_j}_2^2 = t_j^2\mu_j\norm{u_j}_2^2 \to 0, \label{5.24}\\
        &\int_{\R^2} g(v_j)v_j\, dx =t_j^2\int_{\R^2} g(u_j)u_j\, dx\to 0. \label{5.25}
    \end{align}
In particular, $(v_j)_{j=1}^\infty$ is bounded in $F$.  We may assume 
$v_j\wlimit v_0$ weakly in $F$ after extracting a subsequence.

\smallskip

\noindent
\textbf{Step 3:} \textsl{$v_0=0$.}

\smallskip

\noindent
Note that $F\subset C(\R^N\setminus\{ 0\},\R)$ and, applying \cref{Lemma:2.4} (ii), 
observe that $v_0(x)\to 0$ as $\abs x\to \infty$.
By the behavior \eqref{5.22} of $g$, to claim $v_0\equiv 0$, it suffices 
to show $g(v_0)\equiv 0$, equivalently, to show
    \begin{equation} \label{5.26}
        \int_{\R^2} g(v_0)\varphi\, dx =0 \quad \text{for all}\ \varphi\in E.
    \end{equation}
For a given $\varphi\in E$ we set $\psi_j(x)=\varphi(t_jx)$.  Then we have
    \begin{equation*}
        \norm{\psi_j}_E^2 = \norm{\nabla\psi_j}_2^2 + \norm{\psi_j}_2^2
        = \norm{\nabla\varphi}_2^2 + t_j^{-2}\norm{\varphi}_2^2.
    \end{equation*}
On the other hand by the Gagliardo-Nirenberg inequality, we have 
$t_j^{-2}=\norm{u_j}_4^4 \leq C\norm{\nabla u_j}_2^2\norm{u_j}_2^2$ and thus
$\norm{u_j}_E^2 \geq \norm{u_j}_2^2 \geq Ct_j^{-2}$, which implies
$\norm{\psi_j}_E\leq C_\varphi\norm{u_j}_E$ for some constant $C_\varphi>0$
independent of $j$.  Since $(\lambda_j,u_j)$ is a $(PSPC)_b$ sequence,
    \begin{equation*}
        \pabs{\partial_u I(\lambda_j,u_j)\psi_j} 
        \leq \norm{\psi_j}_E \norm{\partial_u I(\lambda_j,u_j)}_{E^*} 
        \leq C_\varphi(1+\norm{u_j}_E) \norm{\partial_u I(\lambda_j,u_j)}_{E^*}
        \to 0.      
    \end{equation*}
We compute that
    \begin{equation*}
        t_j^2\partial_u I(\lambda_j,u_j)\psi_j =
        t_j^2(\nabla v_j,\nabla\varphi)_2 +\mu_j(v_j,\varphi)_2 
            -\intRN g(v_j)\varphi\, dx.
    \end{equation*}
By \eqref{5.24} it implies that $\intRN g(v_j)\varphi\, dx\to 0$.
Thus we have \eqref{5.26} and $v_0\equiv 0$.

\smallskip


\noindent
\textbf{Step 4:} \textsl{Conclusion.}

\smallskip

\noindent
By \cref{Lemma:2.4} (ii), there exists $R_0>0$ such that
    \begin{equation*}
        \abs{v_j(x)} \leq \delta \quad \text{for all}\ \abs x\geq R_0 \ 
        \text{and}\ j\in\N.
    \end{equation*}
Since $v_j\wlimit v_0=0$ weakly in $F$, we have the strong convergence 
$v_j\to 0$ in $L^4(B_{R_0})$. 
On the other hand, we have by \eqref{5.22} and \eqref{5.25}
    \begin{equation*}
        C_1\norm{v_j}_{L^4(\R^2\setminus B_{R_0})}^4
        \leq \int_{\R^2\setminus B_{R_0}} g(v_j)v_j\, dx 
        = \int_{\R^2} g(v_j)v_j\, dx - \int_{B_{R_0}} g(v_j)v_j\, dx
        \to 0
    \end{equation*}
as $j\to\infty$.
Thus we have $\norm{v_j}_{4}\to 0$, which contradicts with $\norm{v_j}_4=1$, 
and we have boundedness of $(u_j)_{j=1}^\infty$ in $F$.
\end{proof}

Now we give a proof to \cref{Proposition:5.8}.

\begin{proof}[Proof of \cref{Proposition:5.8}]
First we deal with $(PSPC)_b$ sequences.
Suppose that $(\lambda_j,u_j)_{j=1}^\infty$ satisfies \eqref{5.1}--\eqref{5.4}
and $\mu_j\to 0$.
By \cref{Lemma:5.9}, $(u_j)_{j=1}^\infty$ is bounded in $F$. 
After extracting a subsequence, we may assume for some $u_0\in F$
    \begin{equation*}
        u_j\wlimit u_0 \quad \text{weakly in}\ F.
    \end{equation*}
Since $\partial_u I(\lambda_j,u_j)\varphi\to 0$ for any $\varphi\in E$, we have
by \eqref{5.19}
    \begin{equation*}
        (\nabla u_0,\nabla\varphi)_2 -\intRN g(u_0)\varphi\, dx=0
        \quad \text{for all}\ \varphi\in E.
    \end{equation*}
Hence $Z'(u_0) = 0$.

Set $v_j = u_j - u_0$. Since $\intRN H(u_j)\,dx\to \intRN H(u_0)\,dx$, 
$\intRN h(u_j)u_j\,dx\to \intRN h(u_0)u_0\,dx$, 
$\intRN \abs{H(v_j)}\,dx\to 0$ and $\intRN \abs{h(v_j)v_j}\,dx\to 0$
follow from \ref{(g1)} and \cref{Lemma:2.3},
we deduce by the Brezis-Lieb lemma that 
    \begin{align}
    &\intRN G(u_j)\, dx = \intRN G(u_0)\, dx + \intRN G(v_j) \, dx
        + o(1), \label{5.27}\\
    &\intRN g(u_j)u_j\, dx = \intRN g(u_0)u_0\, dx + \intRN g(v_j)v_j\, dx
        + o(1). \label{5.28}
    \end{align}
By \eqref{5.20},
    \begin{equation} \label{5.29}
        N\intRN G(u_j)\, dx -{\frac{N-2}{2}}\intRN g(u_j)u_j \, dx \to 0.
    \end{equation}
On the other hand, since $u_0\in F$ is a solution of \eqref{1.15}, \cref{Lemma:2.5}
implies
    \begin{equation} \label{5.30}
        \intRN NG(u_0)\, dx -{\frac{N-2}{2}}g(u_0)u_0 \, dx =0.
    \end{equation}
Properties \eqref{5.27}--\eqref{5.30} imply
    \begin{equation} \label{5.31}
        \intRN NG(v_j)\, dx -{\frac{N-2}{2}} g(v_j)v_j \, dx \to 0.
    \end{equation}
We note that
    \begin{equation*}
            NG(s)-{\frac{N-2}{2}}g(s)s = {\frac{2}{N+2}}\abs s^{s+1} +K(s),
    \end{equation*}
where $K(s)=NH(s)-{\frac{N-2}{2}}h(s)s$.  Since ${\frac{K(s)}{\abs s^{p+1}}}\to 0$ as
$s\to 0$ and $\abs s\to\infty$, it follows from \cref{Lemma:2.3} (ii) that
$\intRN K(v_j)\, dx\to 0$.  Thus \eqref{5.31} implies $\norm{v_j}_{p+1}\to 0$, that is,
$u_j\to u_0$ strongly in $L^{p+1}(\R^N)$.

By \eqref{5.3}, we have $\partial_u I(\lambda_j,u_j)u_j\to 0$. Together
with \eqref{5.19} and \eqref{2.13}, we see that 
    \begin{equation*}
        \norm{\nabla u_j}_2^2 =\intRN g(u_j)u_j\, dx +o(1)
        \to \intRN g(u_0)u_0\, dx = \norm{\nabla u_0}_2^2.
    \end{equation*}
Hence, $\nabla u_j\to \nabla u_0$ strongly in $L^2(\R^N)$. 
Therefore, $u_j\to u_0$ strongly in $F$ and \eqref{5.1} gives $Z(u_0)=b$.

Next we deal with $(PSPC)_b^*$ sequences.  We suppose that 
$(\lambda_j,u_j)_{j=1}^\infty$ is a $(PSPC)_b$ sequence with 
    \begin{equation} \label{5.32}
        \distE(u_j,\scrP_0)\to 0.       
    \end{equation}
By \eqref{5.32}, $u_j$ can be written as $u_j=v_j+\varphi_j$ with $v_j\in\scrP_0$ and
$\norm{\varphi_j}_E\to 0$, which implies $\norm{\varphi_j}_{p+1}\to 0$. 
Thus, denoting $s^-=\max\{ -s, 0\}$ for $s\in\R$, we have
    \begin{equation*}
        \norm{u_j^-}_{p+1} = \norm{(v_j+\varphi_j)^-}_{p+1} 
        \leq \norm{\varphi_j^-}_{p+1}
        \leq \norm{\varphi_j}_{p+1}\to 0.
    \end{equation*}
Thus $u_0^-\equiv 0$, that is, $u_0\geq 0$. 
\end{proof}

Recalling \cref{Lemma:2.5} and \cref{Proposition:1.2}, we have
\begin{itemize}
\item[(i)] all critical values of $Z$ are non-negative;
\item[(ii)] under \ref{(g2)}, $Z$ has no nontrivial non-negative critical point;
\item[(iii)] either $N=2,3,4$ or \ref{(g3)} implies \ref{(g2)}.
\end{itemize}

\noindent
Thus \cref{Proposition:5.4} in Case C follows from \cref{Proposition:5.8}.

\begin{Remark}\label{Remark:5.10}
In the proof of \cref{Proposition:5.8}, the properties \eqref{5.22}--\eqref{5.23}
of $g$ play a crucial role to obtain boundedness 
of $\norm{u_j}_{p+1}$ for $(PSPC)_b$-sequences $(u_j)_{j=1}^\infty$.
We note that \eqref{5.22}--\eqref{5.23} follow from \ref{(g0')} and \ref{(g1)}.

If we take an approach using the truncation of $g$ in $[0,\infty)$ as in 
\cref{Remark:1.10}, then the truncated $g$ does not satisfy \eqref{5.22}--\eqref{5.23}
for $s<0$.  Therefore, the estimate for the negative part $u_{j-}$ is an issue 
and it is difficult to show the strong convergence in $L^{p+1}$ as in 
\cref{Proposition:5.8}.

In \cref{Section:8}, using the odd extension of $g$, we introduce a new approach 
to find a positive solution.  In our new approach, we have \eqref{5.22}--\eqref{5.23}
and \cref{Proposition:5.8} is applicable directly.  We also note that
\cref{Proposition:5.8} works in the setting of \cref{Appendix:A.3}.
\end{Remark}

\section{Strategy to prove \cref{Theorem:1.1}} \label{Section:6}
To show our \cref{Theorem:1.1}, we use the two minimax values $\ub$ and $\ob$ 
defined in \cref{Section:3}.
By \cref{Proposition:4.2} (ii), we have $-\infty < \ub \leq 0 \leq \ob <\infty$ and 
at least one of
the following 3 cases occurs:

\begin{description}
\item[Case 1:] $\ub<0$,
\item[Case 2:] $\ob>0$,
\item[Case 3:] $\ub = \ob =0$.
\end{description}

\noindent
We deal with these 3 cases separately.

In Case 1, the $(PSPC)_{\ub}$ condition holds by \cref{Proposition:5.4} (i). 
In \cref{Section:7} we develop deformation argument and show 
the existence of a positive solution.  To show positivity of the critical point 
we use \cref{Proposition:4.5}.

In Case 2, the condition $(PSPC)_{\ob}^*$ condition holds.  
In \cref{Section:8}, 
we show that the $(PSPC)_{\ob}^*$ condition enables us to develop a deformation flow
in a small neighborhood of the set $\scrC_0=\set{ (\lambda,u) | u\geq 0 \ \hbox{in}\ \R^N}$
of non-negative functions.  
However this local deformation flow is not enough to show the existence of a critical point.
Also in \cref{Section:8}, we develop an iteration argument together 
with the absolute operator 
    \begin{equation*} 
        \Upsilon:\, (\lambda,u(\cdot))\mapsto (\lambda,\abs{u(\cdot)}); \,
        \RE\to \RE          
    \end{equation*}
to construct a flow which enables us to show the existence of a positive
critical point. 

Case 3 is a very degenerate case and neither the $(PSPC)_0$ condition nor
the $(PSPC)_0^*$ condition holds.
However we have already shown the existence and multiplicity of positive solutions
in \cref{Proposition:4.6}.  In particular, \eqref{1.1} with
$m=m_1$ has uncountably many solutions.

\section{Deformation argument under $(PSPC)$ and proof of \cref{Theorem:1.1}: Case 1 ($\ub<0$)} 
\label{Section:7}
In this section we give a proof of \cref{Theorem:1.1} for the case $\ub<0$. 
Here we do not use condition \ref{(g2)}.

\subsection{Deformation argument under $(PSPC)$ in augmented space}
\label{Section:7.1}

To prove \cref{Theorem:1.1}, we need an abstract deformation result. 
We believe that our deformation result (\cref{Proposition:7.1}) is useful not only 
to find positive solutions but also to find multiplicity of solutions including
sign-changing ones.

Throughout this section, suppose that a functional $I \in C^1(\RE,\R)$ is given, 
which may differ from the one in \eqref{2.2}, and we define 
    \begin{equation*}
        J(\theta,\lambda,u(x))
        =I \left( \lambda,u \left({\frac{x}{e^\theta}} \right) \right):\,
        \RRE\to \R.
    \end{equation*}
We also assume the following condition on $I$: 
\begin{enumerate}[label=(J)]
\item \label{(J)}
$J\in C^1(\RRE,\R)$.
\end{enumerate}
Remark that $I$ in \eqref{2.2} has this property. Under \ref{(J)}, we define
    \begin{equation*}
            P(\lambda,u)= \partial_\theta J(0,\lambda,u):\,\RE\to\R.
    \end{equation*}
The notion of $(PSPC)_b$ sequences and the $(PSPC)_b$ condition are defined 
as in \cref{Definition:5.1} for $I$ under \ref{(J)}.
For $b\in\R$ we can also define the critical set $K_b$ as in \cref{Remark:5.2}:
    \begin{equation}\label{7.1}
        K_b =\Set{ (\lambda,u)\in\RE | I(\lambda,u)=b, \ 
            \partial_\lambda I(\lambda,u)=0,\
            \partial_uI(\lambda,u)=0, \ P(\lambda,u)=0}.
    \end{equation}
We recall that $K_b$ is compact under the condition $(PSPC)_b$.
We also use notation for $c\in \R$
    \[  [I\leq c] = \Set{ (\lambda,u)\in\RE | I(\lambda,u)\leq c}.
    \]
The following proposition is a main result of this section:

\begin{Proposition} \label{Proposition:7.1}
Suppose that $I \in C^1(\RE,\R)$ satisfies {\rm \ref{(J)}}. Moreover for $b\in\R$
assume that $I $ satisfies the $(PSPC)_b$ condition, that is, 
any sequence $(\lambda_j,u_j)_{j=1}^\infty$ with \eqref{5.1}--\eqref{5.4}
has a strongly convergent subsequence. Then 
for any neighborhood $\calO$ of $K_b$ and for any $\oepsilon>0$ 
there exist $\epsilon\in (0,\oepsilon)$ and a continuous map 
$\eta(t,\lambda,u):\, [0,1]\times\RE\to\RE$ such that
\begin{enumerate}[label={\rm (\arabic*)}]
\item $\eta(0,\lambda,u)=(\lambda,u)$ for all $(\lambda,u)\in\RE$;
\item $\eta(t,\lambda,u)=(\lambda,u)$ for all 
$t\in [0,1]$ if $I(\lambda,u)\leq b-\oepsilon$;
\item $I(\eta(t,\lambda,u))\leq I(\lambda,u)$ for all $(t,\lambda,u)\in
[0,1]\times\RE$;
\item $\eta(1,[I\leq b+\epsilon]\setminus\calO)\subset [I\leq b-\epsilon]$;
\item 
if $K_b=\emptyset$, 
then $\eta(1, [I\leq b+\epsilon])\subset [I\leq b-\epsilon]$.
\end{enumerate}
\end{Proposition}

Related deformation results are shown in \cite{CT2,HT0,Ik20,IT0} under
the stronger condition $(PSP)_b$ in which we replace \eqref{5.3} by
    \begin{equation*}
        \norm{\partial_u I(\lambda_j,u_j)}_{E^*} \to 0.
    \end{equation*}
To prove \cref{Proposition:7.1}, we construct our deformation flow
$\eta$ through a deformation for $J$ as in \cite{HT0,IT0,Ik20,CT2}.
We write $M=\RRE$ and 
define $\Phi_\tau:\, \R\times M\to M$ by
    \begin{equation*}
        \Phi_\tau(\theta,\lambda,u(x)) = (\theta+\tau,\lambda,u(e^\tau x)).
    \end{equation*}
Then $\Phi_\tau$ is a $C^0$-group action on $M$ and we observe that $J$ is invariant
under $\Phi_\tau$, that is,
    \begin{equation*}
        J(\Phi_\tau(\theta,\lambda,u)) = J(\theta,\lambda,u)
        \qquad \text{for all}\ \tau\in\R.
    \end{equation*}
Furthermore we note that for $(\theta,\lambda,u)\in M$
    \begin{align}   
    &\partial_\theta J(\theta,\lambda,u(x))
        =P\left( \lambda, u\left( {\frac{x}{e^\theta}} \right) \right), \label{7.2}\\
    &\partial_\lambda J(\theta,\lambda,u(x))
        =\partial_\lambda I 
            \left( \lambda, u \left( {\frac{x}{e^\theta}} \right) \right), 
                                                                        \label{7.3}\\
    &\partial_u J(\theta,\lambda,u(x))h(x)
        =\partial_u I \left( \lambda, u \left({\frac{x}{e^\theta}} \right) \right) 
            h \left({\frac{x}{e^\theta}} \right). \label{7.4}
    \end{align}
We also define $\pi:\, M\to\RE$, $\iota:\,\RE\to M$ by
    \begin{equation*}
        \pi(\theta,\lambda,u(x)) 
            = \left( \lambda, \, u \left({\frac{x}{e^\theta}}\right) \right), \quad
        \iota(\lambda,u(x)) = (0, \lambda,u(x)).           
    \end{equation*}
Then
    \begin{equation*}
        I(\pi(\theta,\lambda,u)) = J(\theta,\lambda,u)  \quad 
        \text{for all}\ (\theta,\lambda,u)\in M.
    \end{equation*}
First we construct a deformation $\weta(t,\theta,\lambda,u):\, [0,1]\times M\to M$ 
for $J$.  Next we construct a desired flow $\eta$ on $\RE$ by
    \begin{equation*}
        \eta(t,\lambda,u)=\pi(\weta(t,0,\lambda,u)).
    \end{equation*}
We regard $M$ as a Hilbert manifold and introduce the following metric:
    \begin{equation*}
        \norm{ (\alpha,\nu,h)}_{(\theta,\lambda,u)}^2 
        = \abs\alpha^2 +\abs\nu^2 +\norm h_{E_\theta}^2 \quad
        \text{for}\ (\alpha,\nu,h)\in T_{(\theta,\lambda,u)}M=\RRE,
    \end{equation*}
where
    \begin{equation} \label{7.5}
        \norm h_{E_\theta}^2 = \pnorm{ h \left({\frac{x}{e^\theta}} \right)}_E^2
        = e^{(N-2)\theta}\norm{\nabla h}_2^2 +e^{N\theta}\norm h_2^2.
    \end{equation}
It follows from \eqref{7.5} that
    \begin{equation} \label{7.6}
        e^{-N\abs\theta/2}\norm u_E \leq \norm u_{E_\theta} 
            \leq e^{N\abs\theta/2}\norm u_E \quad
            \text{for all}\ \theta\in\R \ \text{and}\ u\in E.
    \end{equation}
We denote by $\norm{\cdot}_{(\theta,\lambda,u), *}$ and $\norm{\cdot}_{E_\theta^*}$ 
the dual norm on $T_{(\theta,\lambda,u)}^* M$ and $(E, \norm{\cdot}_{E_\theta})$.  
Writing $D=(\partial_\theta,\partial_\lambda,\partial_u)$, we have
    \begin{equation*}
        \begin{aligned}
        \norm{ DJ(\theta,\lambda,u)}_{(\theta,\lambda,u),*}
        &= \sup_{\norm{(\alpha,\nu,h)}_{(\theta,\lambda,u)}\leq 1} 
            DJ(\theta,\lambda,u)(\alpha,\nu,h) \\
        &= \left( \abs{P(\pi(\theta,\lambda,u))}^2 
            + \abs{\partial_\lambda I(\pi(\theta,\lambda,u))}^2
            +\norm{\partial_u I(\pi(\theta,\lambda,u))}_{E^*}^2\right)^{1/2}. 
        \end{aligned}
    \end{equation*}
Here we use the fact that
    \begin{equation*}
        \begin{aligned}
        \norm{\partial_u J(\theta,\lambda,u)}_{E_\theta^*}
        &= \sup_{\norm h_{E_\theta} \leq 1} \abs{\partial_u J(\theta,\lambda,u)h} 
            = \sup_{\norm h_{E_\theta} \leq 1} 
            \pabs{\partial_u I(\lambda,u({\frac{x}{e^\theta}}))h({\frac{x}{e^\theta}})}\\
        &= \sup_{\norm \varphi_E \leq 1} 
            \pabs{\partial_u I(\lambda,u({\frac{x}{e^\theta}}))\varphi}
        = \norm{\partial_u I(\pi(\theta,\lambda,u))}_{E^*}. 
        \end{aligned}
    \end{equation*}
Denote by $\distM(\cdot,\cdot)$ the distance on $M$ corresponding to
the Riemannian metric, that is,
    \begin{equation*}
        \begin{aligned}
        &\distM((\theta_0,\lambda_0,u_0), (\theta_1,\lambda_1,u_1)) \\
        =&\inf\left\{ \int_0^1 \norm{\dot\gamma(t)}_{\gamma(t)}\, dt \ \bigg| \ 
            \gamma \in C^1([0,1],M), \,
            \gamma(i)=(\theta_i,\lambda_i, u_i) \ \text{for}\ i=0,1\right\}. 
        \end{aligned}
    \end{equation*}
It is not difficult to see that for all $\beta\in\R$
    \begin{equation}\label{7.7}
        \distM \left( (\theta_0+\beta,\lambda_0,u_0), 
            (\theta_1+\beta,\lambda_1,u_1) \right)
        = \distM \left( \left( \theta_0,\lambda_0,
            u_0 \left({\frac{x}{e^\beta}}\right) \right), 
            \left( \theta_1,\lambda_1,
                u_1 \left( {\frac{x}{e^\beta}} \right) \right) \right).
    \end{equation}
We also set 
    \begin{equation*}
        \wKb =\pi^{-1}(K_b) = \Set{ (\theta,\lambda,u)\in M | 
                J(\theta,\lambda,u)=b,\ DJ(\theta,\lambda,u)=0},
    \end{equation*}
where $K_b$ is defined in \eqref{7.1}.
Remark that $\wKb$ is not compact even if $K_b$ is compact, since 
$\wKb$ is invariant under non-compact group action $\Phi_\tau$.

We begin with the following lemma. 

\begin{Lemma} \label{Lemma:7.2}
Suppose that $I $ satisfies the $(PSPC)_b$ condition. Then
$J$ satisfies the following $\widetilde{(PSC)_b}$ condition:
\begin{description}
\item[$\widetilde{(PSC)_b}$] If $(\theta_j,\lambda_j,u_j)_{j=1}^\infty\subset M$ 
satisfies
    \begin{align}   
        &J(\theta_j,\lambda_j,u_j)\to b, \label{7.8}\\
        &\partial_\theta J(\theta_j,\lambda_j,u_j)\to 0, \label{7.9}\\
        &\partial_\lambda J(\theta_j,\lambda_j,u_j)\to 0, \label{7.10}\\
        &(1+\norm{u_j}_{E_{\theta_j}})
            \norm{\partial_u J(\theta_j,\lambda_j,u_j)}_{E_{\theta_j}^*}\to 0,
                \label{7.11}
    \end{align}
then
    \begin{equation*}
        \distM \left( (\theta_j,\lambda_j,u_j),\wKb \right)\to 0 \quad
        \text{as}\ j\to\infty.
    \end{equation*}
\end{description}
\end{Lemma}

\begin{proof}
First we claim that for $(\theta,\lambda,u)\in M$
    \begin{equation} \label{7.12}
        \distM \left( (\theta,\lambda,u), \wKb \right)
        \leq \distRE(\pi(\theta,\lambda, u), K_b).
    \end{equation}
In fact, for any $(\lambda_0,u_0)\in K_b$, we have 
$(\theta,\lambda_0,u_0(e^\theta x)) \in \wKb$ for $\theta\in \R$ thanks to 
\eqref{7.2}--\eqref{7.4}. Thus
    \begin{equation*}
        \begin{aligned}
    &\distM((\theta,\lambda,u),\wKb) \\
    \leq& \distM((\theta,\lambda,u), (\theta,\lambda_0,u_0(e^\theta x))) 
        = \distM \left( \left( 0,\lambda, u \left({\frac{x}{e^\theta}} \right) \right), 
            (0,\lambda_0,u_0(x)) \right) \\
    \leq& \left( \abs{\lambda-\lambda_0}^2 
        +\pnorm{ u \left({\frac{x}{e^\theta}} \right)-u_0 }_E^2 \right)^{1/2}
        = \distRE \left( \left( \lambda,u \left( {\frac{x}{e^\theta}} \right) \right), 
        (\lambda_0,u_0) \right). 
        \end{aligned}
    \end{equation*}
Since $(\lambda_0,u_0)\in K_b$ is arbitrary, \eqref{7.12} holds.

Let $(\theta_j,\lambda_j,u_j)_{j=1}^\infty\subset M$ be a sequence with
\eqref{7.8}--\eqref{7.11}.
We set $\widetilde u_j(x) =u_j({\frac{x}{e^{\theta_j}}})$, equivalently 
$(\lambda_j,\widetilde u_j)=\pi(\theta_j,\lambda_j,u_j)$.
By \eqref{7.2}--\eqref{7.4}, it is easy to see that 
$(\lambda_j,\widetilde u_j)_{j=1}^\infty$ is a $(PSPC)_b$ sequence for $I$.
Since the $(PSPC)_b$ condition holds for $I $, $(\lambda_j,\widetilde u_j)_{j=1}^\infty$
is relatively compact in $\RE$ and all accumulation points belong to $K_b$.
Thus we have
    \begin{equation*}
        \distM \left( (\theta_j,\lambda_j,u_j),\wKb \right)
        \leq \distRE \left( (\lambda_j,\widetilde u_j),K_b \right)\to 0 \quad
            \text{as}\ j\to\infty. \qedhere
    \end{equation*}
\end{proof}

For $\rho>0$ and $c\in\R$, we use the following notation:
    \begin{equation*}
        \begin{aligned}
        \wN_\rho &= \Set{ (\theta,\lambda,u)\in M | 
            \distM((\theta,\lambda,u), \wKb)< \rho}, \\
        [J\leq c]_M &= \Set{ (\theta,\lambda,u)\in M | J(\theta,\lambda,u)\leq c }. 
        \end{aligned}
    \end{equation*}
We also define $[J>c]_M$ and $[\abs{J-b}<\epsilon]_M$ similarly. 
When $K_b=\emptyset$, we have  $\wKb=\emptyset$ and regard $\wN_\rho=\emptyset$
for $\rho>0$.

Now we give our deformation result for $J$.

\begin{Proposition} \label{Proposition:7.3}
Assume that $I$ satisfies the ${(PSPC)_b}$ condition.
Then for any $\oepsilon>0$ and $\rho>0$ there exist $\epsilon\in (0,\oepsilon)$ and
a continuous map $\weta:\, [0,1]\times M\to M$ such that
\begin{enumerate}[label={\rm (\arabic*)}]
\item $\weta(0,\theta,\lambda,u)=(\theta,\lambda,u)$ for all $(\theta,\lambda,u)\in M$;
\item $\weta(t,\theta,\lambda,u)=(\theta,\lambda,u)$ for all $t\in [0,1]$ if
$J(\theta,\lambda,u)\leq b-\oepsilon$;
\item $J(\weta(t,\theta,\lambda,u))\leq J(\theta,\lambda,u)$ 
for all $(t,\theta,\lambda,u) \in [0,1] \times M$; 
\item $\weta(1, [J\leq b+\epsilon]_M\setminus\wN_\rho)\subset [J\leq b-\epsilon]_M$; 
\item if $K_b=\emptyset$, then $\weta(1, [J\leq b+\epsilon]_M)\subset
[J\leq b-\epsilon]_M$.
\end{enumerate}
\end{Proposition}

We note that $K_b$ is compact since $I$ satisfies the $(PSPC)_b$ condition.  We also
note that $J$ satisfies $\widetilde{(PSP)}_b$ condition by \cref{Lemma:7.2}.

\subsection{Proof of \cref{Proposition:7.3}}
\begin{proof}[Proof of \cref{Proposition:7.3}]
A proof consists of several steps. Here we write $U=(\theta,\lambda,u)$.

\smallskip

\noindent
\textbf{Step 1:} \textsl{For any $\rho>0$ there exists $\nu>0$ such that
    \begin{equation} \label{7.13}
        \max\{ \abs{\partial_\theta J(U)}, \abs{\partial_\lambda J(U)}, 
        (1+\norm u_{E_\theta})\norm{\partial_u J(U)}_{E_\theta^*}\}
        \geq 3\nu
        \; \text{ for all}\ U\in [\abs{J-b}<\nu]_M\setminus\wN_{\rho/3}.    
    \end{equation}
}

\smallskip

\noindent
Step 1 follows from the $\widetilde{(PSC)_b}$ condition. 

\smallskip

\noindent
\textbf{Step 2:} \textsl{There exists a locally Lipschitz vector field
    \begin{equation*}
        X(U)=(X_1(U),X_2(U),X_3(U)):\, 
        [\abs{J-b}<\nu]_M\setminus \wN_{\rho/3}\to\RRE
    \end{equation*}
such that for all $U\in [\abs{J-b}<\nu]_M\setminus \wN_{\rho/3}$
    \begin{align}   
        &DJ(U)X(U) \geq \nu.                    \label{7.14}\\
        &\abs{X_1(U)}, \ \abs{X_2(U)} \leq 1,   \label{7.15}\\
        &\norm{X_3(U)}_{E_\theta} \leq 1+\norm u_{E_\theta}.    \label{7.16}
    \end{align}
}

\smallskip

\noindent
By \eqref{7.13}, for any $U\in [\abs{J-b}<\nu]_M\setminus \wN_{\rho/3}$ we have
at least one of the following three properties:
    \begin{equation*}
        (1)\ \abs{\partial_\theta J(U)}\geq 3\nu, \quad
        (2)\ \abs{\partial_\lambda J(U)}\geq 3\nu, \quad
        (3)\ (1+\norm u_{E_\theta})\norm{\partial_u J(U)}_{E_\theta^*}\geq 3\nu.
    \end{equation*}
Thus for any $U\in [\abs{J-b}<\nu]_M\setminus \wN_{\rho/3}$ there exists a vector
$Y_U=(Y_{1U}, Y_{2U}, Y_{3U}) \in\RRE$ according to the above 3 cases such that
\begin{itemize}
\item[(1)] $Y_U$ is of form $Y_U=(Y_{U1},0,0)$ with $\abs{Y_{U1}}\leq 1$ 
and $DJ(U)(Y_{U1},0,0)\geq 2\nu$; 
\item[(2)] $Y_U$ is of form $Y_U=(0,Y_{U2},0)$ with $\abs{Y_{U2}}\leq 1$ 
and $D J(U)(0,Y_{U2},0)\geq 2\nu$;
\item[(3)] $Y_U$ is of form $Y_U=(0,0,Y_{U3})$ with 
$\abs{Y_{U3}} < 1+\norm u_{E_\theta}$ and $D J(U)(0,0,Y_{U3})\geq 2\nu$.
\end{itemize}
Using a suitable partition of unity, we can construct the desired locally 
Lipschitz vector field $X:\, [\abs{J-b}<\nu]_M\setminus \wN_{\rho/3}\to \RRE$
in a standard way.

Next we show $\norm u_{E_\theta}$ stays bounded in $\wN_{\rho}$. 

\smallskip

\smallskip

\noindent
\textbf{Step 3:} \textsl{Definition of ODE and its property.}

\smallskip

\noindent
Let $\overline{\varepsilon} >0$ and $\rho > 0$ be given, and choose 
$\widetilde{\varepsilon} \in (0,\overline{\varepsilon})$ so that 
    \[  2 \widetilde{\varepsilon} < \min \Set{ \nu , \overline{\varepsilon} },
    \]
where $\nu > 0$ appears in Step 1. 
Then, we consider the following ODE in $M$:
    \begin{equation} \label{7.17}
        \left\{ \begin{aligned}
        &{\frac{d\weta}{dt}}=-\varphi_1(\weta)\varphi_2(J(\weta)) X(\weta), \\
        &\weta(0,U)=U, 
        \end{aligned}
        \right.     
    \end{equation}
where $\varphi_1:\, M\to [0,1]$ and $\varphi_2:\, \R\to [0,1]$ are locally Lipschitz
continuous functions such that
    \begin{equation*}
    \begin{aligned}
    &\varphi_1(U)=1 \quad \text{for}\ U\in M\setminus\wN_{{\frac{2}{3}}\rho}, \\
    &\varphi_1(U)=0 \quad \text{for}\ U\in \wN_{{\frac{1}{3}}\rho}, \\
    &\varphi_2(s)=1 \quad \text{for}\ s\in [ b- \widetilde{\varepsilon},b
        + \widetilde{\varepsilon} ], \\
    &\varphi_2(s)=0 \quad \text{for}\ s\not\in [b- 2\widetilde{\varepsilon},b
        + 2 \widetilde{\varepsilon}]. 
    \end{aligned}
    \end{equation*}
For a solution $\weta(t)=\weta(t,U)=(\weta_1(t,U),\weta_2(t,U),\weta_3(t,U))$ of
\eqref{7.17}, it follows from \eqref{7.14}--\eqref{7.16} that 
    \begin{align}   
        &{\frac{d}{dt}}J(\weta(t)) = -\varphi_1(\weta)\varphi_2(J(\weta))
            DJ(\weta)X(\weta)
        \leq -\varphi_1(\weta)\varphi_2(J(\weta))\nu, \nonumber\\
        &\pabs{{\frac{d}{dt}}\weta_i(t)} \leq \abs{X_i(\weta)}\leq 1
            \quad \text{for}\ i=1,2,    \label{7.18}\\
        &\pnorm{{\frac{d}{dt}}\weta_3(t)}_{E_{\weta_1(t)}} 
            \leq \norm{X_3(\weta(t))}_{E_{\weta_1(t)}} 
            \leq 1+ \norm{\weta_3(t)}_{E_{\weta_1(t)}}.
            \label{7.19}
    \end{align}
By our construction, we have for all $U\in E$
    \begin{align}   
    &{\frac{d}{dt}}J(\weta(t)) \leq 0 \quad \text{for all}\ t, \label{7.20}\\
    &{\frac{d}{dt}}J(\weta(t)) \leq -\nu \quad 
        \text{if}\ \weta(t)\in \left[\abs{J-b} 
            \leq \widetilde{\varepsilon}\right]_M\setminus \wN_{{\frac{2}{3}}\rho}. 
            \label{7.21}
    \end{align}
We note that the solution $\weta $ exists globally in time $t$. 
In fact, for any $T>0$, \eqref{7.18} gives 
    \begin{equation*}
        \abs{\weta_i(t)} \leq \abs{\weta_i(0)}+T \quad 
        \text{for $t\in [0,T]$ and $i=1,2$}.
    \end{equation*}
By \eqref{7.6}, for some constants $C_{1T}$, $C_{2T}>0$ depending only on
$T$ and the initial value $\weta(0)=U\in M$, we have 
    \begin{equation*}
            C_{1T}\norm h_E \leq \norm h_{E_{\weta_1(t)}} \leq C_{2T}\norm h_E
        \quad \text{for} \ h\in E\ \text{and}\ t \in [0,T].
    \end{equation*}
Thus by \eqref{7.19}
    \begin{align} 
    \pnorm{{\frac{d}{dt}}\weta_3(t)}_E 
    &\leq \norm{X_3(\weta_3(t))}_E\leq C_{1T}^{-1}
        \norm{X_3(\weta_3(t))}_{E_{\weta_1(t)}} 
    \leq C_{1T}^{-1} \left( 1+ \norm{\weta_3(t)}_{E_{\weta_1(t)}}\right) \nonumber\\
    &\leq C_{1T}^{-1} \left( 1+ C_{2T}\norm{\weta_3(t)}_E\right).  \label{7.22}
    \end{align}
Applying the Gronwall inequality we deduce that $\weta$ exists globally in time $t$.

\smallskip

\noindent
\textbf{Step 4:} \textsl{There exists $c_\rho>0$ such that
    \begin{equation} \label{7.23}
        \norm u_{E_\theta} \leq c_\rho \quad 
        \text{for all} \ U=(\theta,\lambda,u)\in \wN_\rho.
    \end{equation}
}

\noindent
We recall that
    \begin{equation*}
        \wKb =\Set{ (\theta,\lambda,u(e^\theta x)) | 
        (\lambda,u)\in K_b, \ \theta\in \R}.
    \end{equation*}
First we note that $\norm u_{E_\theta}$ is bounded for 
$U =(\theta,\lambda, u( e^\theta x))\in \wKb$, which follows
from the compactness of $K_b$ and the fact that
    \begin{equation*}
        \norm{u(e^\theta x)}_{E_\theta} = \norm u_E \quad
        \text{for all}\ \theta\in \R \ \text{and}\ u\in E.
    \end{equation*}
We next consider the boundedness of $\norm u_{E_\theta}$ 
for any $(\theta,\lambda,u) \in \wN_\rho$. 
In view of \eqref{7.7} and the definition of $\norm{\cdot}_{E_\theta}$, 
we may suppose $\theta = 0$. 
For arbitrary $(0,\lambda,u)\in \wN_\rho$ there exists $(\theta',\lambda',u')\in \wKb$
such that
    \begin{equation*}
        \distM((0,\lambda,u),(\theta',\lambda',u'))<\rho,
    \end{equation*}
hence there exists a path 
$\gamma(t)=(\gamma_1(t),\gamma_2(t), \gamma_3(t))\in C^1([0,1],M)$
with $\gamma(0)=(\theta',\lambda',u')$ and $\gamma(1)=(0,\lambda,u)$ such that
    \begin{equation} \label{7.24}
        \int_0^1 \left( \abs{\dot\gamma_1(t)}^2+\abs{\dot\gamma_2(t)}^2
            +\norm{\dot\gamma_3(t)}_{E_{\gamma_1(t)}}^2 \right)^{1/2}\, dt
            <\rho. 
    \end{equation}
From this we deduce that for all $t\in [0,1]$
    \begin{equation*}
        \abs{\gamma_1(t)} = \abs{\gamma_1(1)-\gamma_1(t)}
        \leq \int_t^1 \abs{\dot\gamma_1(\tau)}\, d\tau
        < \rho.
    \end{equation*}
By \eqref{7.6} and \eqref{7.24}, we have
    \begin{equation*}
        \begin{aligned}
        \norm u_E 
        &= \norm{\gamma_3(1)}_E 
        \leq \norm{\gamma_3(0)}_E + \int_0^1 \norm{\dot\gamma_3(\tau)}_E\, d\tau 
        \leq \norm{u'}_E 
            + e^{N\rho/2}\int_0^1 \norm{\dot\gamma_3(\tau)}_{E_{\gamma_1(t)}}\, d\tau \\
        &\leq e^{N\rho/2}\norm{u'}_{E_{\theta'}} + e^{N\rho/2}\rho. 
        \end{aligned}
    \end{equation*}
Since $(\theta',\lambda',u') \in \wKb$ and $\norm{u'}_{E_{\theta'}}$ is 
uniformly bounded, 
so is $\norm{u}_E$ for any $(0,\lambda, u) \in \wN_\rho$ and \eqref{7.23} holds.

\medskip

By Step 4 we note that the vector field $X$ in Step 2 satisfies for
$U=(\theta,\lambda,u)\in \widetilde N_\rho$
    \begin{equation*}
        \norm{X(U)}_U 
        = \left( \norm{X_1(U)}^2 +\abs{X_2(U)}^2 
        + \norm{X_3(U)}_{E_\theta}^2\right)^{1/2}
        \leq c_\rho'= (2+(1+c_\rho)^2)^{1/2}. 
    \end{equation*}
Thus the flow defined in \eqref{7.17} satisfies
    \begin{equation} \label{7.25}
        \pnorm{{\frac{d}{dt}}\weta(t)}_{\weta(t)} 
            \leq \norm{X(\weta(t))}_{\weta(t)} \leq c_\rho'
        \quad \text{if}\ \weta(t)\in \widetilde N_\rho.     
    \end{equation}

\smallskip 

\noindent
\textbf{Step 5:} \textsl{Conclusion.}

\smallskip 

\noindent
It follows from \eqref{7.20}, \eqref{7.21} and \eqref{7.25} that for 
$\epsilon\in (0, \widetilde{\varepsilon}]$ small, 
the solution $\weta(t,U)$ of \eqref{7.17} has the desired properties (1)--(5) of 
\cref{Proposition:7.3}.
We give an outline of the proof of (4) since (5) can be proved similarly.

Suppose $U\in [J\leq b+\epsilon]_M\setminus\wN_\rho$.
Arguing indirectly, we assume that the solution $\weta(t)=\weta(t,U)$ satisfies
    \begin{equation} \label{7.26}
        \weta(1)\not\in [J\leq b-\epsilon]_M,       
    \end{equation}
which implies $\weta(t)\in [\abs{J-b}\leq\epsilon]_M$ for all $t\in [0,1]$.

We consider the following two cases:

\begin{itemize}
\item[(a)] $\weta(t)\not\in \wN_{{\frac{2}{3}}\rho}$ for all $t\in [0,1]$; 
\item[(b)] $\weta(t_0)\in \wN_{{\frac{2}{3}}\rho}$ for some $t_0\in (0,1]$.
\end{itemize}

If (a) occurs, then we have
    \begin{equation*}
        {\frac{d}{dt}}J(\weta(t))\leq -\nu \quad \text{for all}\ t\in [0,1],
    \end{equation*}
which cannot take a place since $\varepsilon \leq \widetilde{\varepsilon}$ and 
$2 \widetilde{\varepsilon} < \nu$. 

If (b) occurs, then there exist $0<t_1<t_2\leq 1$ such that
    \begin{equation*}
        \weta(t_1)\in\partial\wN_\rho, \quad 
        \weta(t_2)\in\partial\wN_{{\frac{2}{3}}\rho}, \quad 
        \weta(t)\in \wN_\rho\setminus \wN_{{\frac{2}{3}}\rho} 
            \quad \text{for}\ t\in [t_1,t_2]. 
    \end{equation*}
Since 
    \[  {\frac{1}{3}}\rho 
        = \distM(\partial\wN_\rho,\partial\wN_{{\frac{2}{3}}\rho})
        \leq \distM \left( \weta(t_1) , \weta(t_2) \right) 
        \leq \int_{t_1}^{t_2} \pnorm{ \frac{d \weta}{dt} (t) }_{\weta(t)} dt ,
    \] 
it follows from \eqref{7.25} that $\abs{t_2-t_1}\geq {\frac{\rho}{3c_\rho'}}$.
Thus by \eqref{7.21}
    \begin{equation*}
        \begin{aligned}
        J(\weta(1))\leq J(\weta(t_2)) 
        \leq J(\weta(t_1))-{\frac{\rho}{3c_\rho'}}\nu
        \leq b+\epsilon -{\frac{\rho}{3c_\rho'}}\nu. 
        \end{aligned}
    \end{equation*}
Thus \eqref{7.26} is impossible for small $\epsilon \in (0, \widetilde{\varepsilon}]$.
\end{proof}

\cref{Proposition:7.1} can be derived from \cref{Lemma:7.2} and 
\cref{Proposition:7.3}. A key is the following lemma,
in which we use notation: for $r>0$
    \begin{equation*}
        N_r = \Set{(\lambda,u) | \distRE((\lambda,u), K_b)<r}.
    \end{equation*}

\begin{Lemma}[{\cite[Lemma 4.9]{HT0}}] \label{Lemma:7.4}
For any $\rho>0$ there exists $R(\rho)>0$ such that
    \begin{equation*}
        \begin{aligned}
        &\pi(\wN_\rho)\subset N_{R(\rho)}, \quad \iota((\RE)\setminus N_{R(\rho)}) 
            \subset M\setminus \wN_\rho. 
        \end{aligned}
    \end{equation*}
Moreover
    \begin{equation*}
            R(\rho)\to 0 \quad \text{as}\ \rho\to 0. 
    \end{equation*}
\end{Lemma}

\begin{proof}[Proof of \cref{Proposition:7.1}]
Let $\calO$ be a given neighborhood of $K_b$ and $\oepsilon>0$ a given positive number. 
By \cref{Lemma:7.4}, we can choose small $\rho>0$ such that 
$N_{R(\rho)}\subset \calO$ and 
from \cref{Lemma:7.2} and \cref{Proposition:7.3} we can find 
$\epsilon \in ( 0, \overline{\varepsilon} )$ and $\weta:\,
[0,1]\times M\to M$ with the properties (1)--(5) in \cref{Proposition:7.3}. 
Then it is easy to check that the flow defined by 
$\eta(t,\lambda,u) =\pi(\weta(t,0,\lambda,u)) :\, [0,1] \times \RE \to \RE$ 
satisfies the desired properties. 
\end{proof}

\subsection{Proof of \cref{Theorem:1.1}: Case 1 ($\ub<0$)}
\label{Section:7.3}
\begin{proof}[Proof of \cref{Theorem:1.1}: Case 1 ($\ub<0$)]
First we deal with Case 1.
By \cref{Proposition:5.4} (i),  the $(PSPC)_{\ub}$ condition holds and 
\cref{Proposition:7.1} is applicable.
We recall $\ub\geq -2Am_1$ and set $\oepsilon = \half$. Then we have
for $\gamma\in\uGamma$
    \begin{equation*}
        I(\gamma(i)) \leq -2Am_1 -1 \leq \ub-\oepsilon
        \quad \text{for}\ i=0,1
    \end{equation*}
and thus $\uGamma$ is stable under the deformation. Arguing in a standard
way, we can see that $\ub$ is a critical value of $I$. 

To show the existence of a positive solution, we need the following
\cref{Corollary:7.5}.
\end{proof}

\begin{Corollary} \label{Corollary:7.5} 
The following equality holds: 
    \begin{equation} \label{7.27}
        \ub=\inf_{\lambda\in\R} b(\lambda).     
    \end{equation}
\end{Corollary}

\begin{proof}
By \cref{Proposition:4.2} (i) and (ii), we have
    \begin{equation*}
        \ub \leq \inf_{\lambda\in\R} b(\lambda) 
        \leq \lim_{\lambda\to-\infty} b(\lambda) = 0.
    \end{equation*}
Thus \eqref{7.27} holds if $\ub=0$.

    If $\ub<0$, then there exists a critical point
$(\lambda_0,u_0)\in \RE$ of $I$ such that $I(\lambda_0,u_0)=\ub$.
Since $u_0\in E$ is a nontrivial critical point of $u\mapsto I(\lambda_0,u)$, 
\cref{Proposition:3.1} and the fact $b(\lambda_0) = a(\mu_0) - \mu_0 m_1$, 
$\mu_0=e^{\lambda_0}$ yield 
    \begin{equation*}
        b(\lambda_0) \leq I(\lambda_0,u_0) = \ub.
    \end{equation*}
Therefore \eqref{7.27} also holds for $\ub<0$. 
\end{proof}

\begin{proof}[End of the proof of \cref{Theorem:1.1}: Case 1 ($\ub<0$)]
Since $b(\lambda)\to 0$ as $\lambda\to \pm\infty$, \cref{Corollary:7.5} 
implies the
existence of $\lambda_0\in\R$ such that $b(\lambda_0) = \ub<0$.  Thus, by
\cref{Proposition:4.5}, a least energy positive solution of 
$-\Delta u +\mu_0 u=g(u)$ with $\mu_0=e^{\lambda_0}$ corresponding to 
$b(\lambda_0)$ is a positive solution of \eqref{1.1} with $m=m_1$.
\end{proof}

We also have the following corollary from \cref{Corollary:7.5} and 
\cref{Theorem:1.1} for $\ub<0$

\begin{Corollary} \label{Corollary:7.6}
Assume \ref{(g0')} and \ref{(g1)} and suppose that 
$\inf_{\lambda\in\R} b(\lambda)<0$.  Then $\inf_{\lambda\in\R} b(\lambda)$
is a critical value of $\calI$.
\end{Corollary}

We can show \cref{Corollary:1.5} as a special case of \cref{Corollary:7.6}
and \cref{Proposition:4.6}.

\begin{proof}[Proof of \cref{Corollary:1.5}]
We may assume \ref{(g0')}.  Under the assumption $g(s)\geq \abs s^{p-1}s$ for
$s\geq 0$, we have
    \[  I(\lambda,u) \leq I_0(\lambda,u) 
        \quad \text{for all}\ (\lambda,u)\in\RE,
    \]
where $I_0$ is the functional corresponding to 
$G_0(s)={\frac{1}{p+1}}\abs s^{p+1}$.  By \cref{Lemma:4.4}, the minimax values
$\ub$ and $\ob$ for $I$ satisfy
    \[  -\infty < \ub \leq \ob=0.
    \]
Thus \cref{Corollary:1.5} follows from \cref{Corollary:7.6} and 
\cref{Proposition:4.6}.
\end{proof}

\section{Deformation argument under $(PSPC)_b^*$ and proof of \cref{Theorem:1.1}: 
Case 2 ($\ob>0$)}\label{Section:8}

Aim of this section is to show $\ob$ is a critical value when $\ob>0$.

\medskip


\subsection{A gradient estimate in a neighborhood of the set of positive functions}
\label{Section:8.1}
We argue indirectly and assume 

\smallskip

\begin{enumerate}[label={\rm $(\#)$}]
\item \label{(N)}
$I$ does not have any critical point $(\lambda,u)\in\RE$ with $I(\lambda,u)=\ob$ and
$u\geq 0$.
\end{enumerate}

\smallskip

\noindent
The aim of this section is to show the following \cref{Proposition:8.1}, 
which gives a gradient estimate 
in a neighborhood $\scrC_\delta$ of the set $\scrC_0$ of non-negative functions.
Here we use notation: for $\delta>0$
    \begin{equation*}
        \begin{aligned}
        &\scrC_0 = \Set{ (\lambda,u)\in \RE | u\geq 0\ \text{in}\ \R^N}, \\
        &\scrC_\delta = \Set{ (\lambda,u)\in \RE | \distE(u,\scrP_0)<\delta},
        \end{aligned}
    \end{equation*}
where $\scrP_0$ is defined in \eqref{5.5}.

\medskip

\begin{Proposition} \label{Proposition:8.1}
Assume \ref{(g0')}, \ref{(g1)}, \ref{(g2)}, $\ob>0$ and \ref{(N)}. 
Then there exist $\delta$, $\rho$, $\nu>0$ such that
    \begin{equation*}
        \max\big\{ \abs{\partial_\lambda I(\lambda,u)}, 
            \, (1+\norm u_E)\norm{\partial_u I(\lambda,u)}_{E^*},\,
        \abs{P(\lambda,u)}\big\} \geq 3\nu
    \end{equation*}
for all $(\lambda,u)\in \scrC_\delta$ with $I(\lambda,u)\in [\ob-\rho,\ob+\rho]$. 
\end{Proposition}

\medskip

\cref{Proposition:8.1} enables us to develop a deformation argument 
in a neighborhood $\scrC_\delta$ of
non-negative functions $\scrC_0$. We will deal with such a deformation argument 
in the following sections.

\begin{proof}[Proof of \cref{Proposition:8.1}]
When $b=\ob$, there is no $(PSPC)_{\ob}^*$ sequence by \ref{(N)}. 
This means that \cref{Proposition:8.1} holds. In fact, if \cref{Proposition:8.1} 
does not hold, then for any $j\in \N$ there exists $(\lambda_j,u_j)$ such that
    \begin{equation*}
        \begin{aligned}
        &(\lambda_j,u_j)\in\scrC_{1/j}, \ 
            \text{that is,}\ \distE(u_j,\scrP_0)<{\frac{1}{j}}, \\
        &I(\lambda_j,u_j) \in \left[\ob-{\frac{1}{j}}, \ob+{\frac{1}{j}}\right] 
        \end{aligned}
    \end{equation*}
and
    \begin{equation*}
        \max\big\{ \abs{\partial_\lambda I(\lambda_j,u_j)}, \,
            (1+\norm{u_j}_E)\norm{\partial_u I(\lambda_j,u_j)}_{E^*}, \,
            \abs{P(\lambda_j,u_j)} \big\} <{\frac{3}{j}}.
    \end{equation*}
It is clear that $(\lambda_j,u_j)_{j=1}^\infty$ is a $(PSPC)_{\ob}^*$-sequence.
Since the $(PSPC)_{\ob}^*$ condition holds, there exists a critical point 
$(\lambda_0,u_0)\in\RE$ with $I(\lambda_0,u_0)=b$ and $u_0\geq 0$, which
is a contradiction to \ref{(N)}.
\end{proof}


\subsection{Deformation flow in a neighborhood of the set of positive solutions} 
\label{Section:8.2}
As in \cref{Section:7}, we develop a deformation flow on the augmented space 
$M=\RRE$.
In what follows we write $U=(\theta,\lambda,u)\in M$ and use notation: for $\delta>0$
    \begin{equation*}
        \begin{aligned}
        \wscrC_0 &=\Set{ U\in M | \pi(U)\in\scrC_0 } 
            = \Set{ (\theta,\lambda,u)\in M | u\geq 0 }, \\
        \wscrC_{\delta} &=\Set{ U\in M | \pi(U)\in\scrC_{\delta} } 
            = \Set{ (\theta,\lambda,u+\varphi)\in M | u\geq 0,\, \norm\varphi_{E_\theta} 
            = \norm{ \varphi({x/e^\theta})}_E \leq \delta }. 
        \end{aligned}
    \end{equation*}
We also define $\wscrC_{\delta/2}$ is a similar way.

We recall our ODE \eqref{7.17} in \cref{Section:7}. Here we modify 
the definition of $\varphi_1$.
Under \ref{(N)}, for $\delta$, $\rho$, $\nu>0$ given in \cref{Proposition:8.1},
we have 
    \begin{equation*}
        \max \big\{ \abs{\partial_\theta J(U)},\, \abs{\partial_\lambda J(U)},\, 
        (1+\norm u_{E_\theta})\norm{\partial_u J(U)}_{E_\theta^*} \big\} \geq 3\nu
    \end{equation*}
for all $U\in\wscrC_\delta$ with $J(U)\in [\ob-\rho,\ob+\rho]$.

We next choose a locally Lipschitz continuous vector field
    \begin{equation*}
        X(U)=(X_1(U), X_2(U), X_3(U)):\, 
        \wscrC_\delta \cap \left[\abs{J-\overline{b}} 
                \leq \rho  \right]_M  
        \to T_U(M)=\RRE
    \end{equation*}
with properties
    \begin{align}   
    &DJ(U) X(U) \geq \nu, \label{8.1}\\
    &\abs{X_1(U)}, \ \abs{X_2(U)} \leq 1,   \label{8.2}\\
    &\norm{X_3(U)}_{E_\theta} \leq 1+\norm u_{E_\theta}.    \label{8.3}
    \end{align}
For $U=(\theta,\lambda,u)\in M$, we consider the following ODE:
    \begin{equation} \label{8.4}
        \left\{ \begin{aligned}
        &{\frac{d\weta}{d\tau}} = -\varphi_1(\weta)\varphi_2(J(\weta))X(\weta), \\
        &\weta(0,U)=U, 
        \end{aligned}
        \right.
    \end{equation}
where $\varphi_1 : \, M\to [0,1]$ and $\varphi_2:\, \R\to [0,1]$ are Lipschitz 
continuous functions such that
    \begin{equation*}
        \varphi_1(U)=\begin{cases}
            1 &\text{for}\ U\in \wscrC_{\delta/2},\\ 
            0&\text{for}\ U\not\in\wscrC_\delta,
        \end{cases} \qquad
        \varphi_2(s)=\begin{cases}
            1 &\text{for}\ s\in [\ob-{\frac{\rho}{2}}, \ob+{\frac{\rho}{2}}],\\
            0 &\text{for}\ s\not\in [\ob-\rho, \ob+\rho].
        \end{cases}
    \end{equation*}
The solution $\weta(t,U):\, [0,\infty)\times M\to M$ of \eqref{8.4} has 
the following properties.

\begin{itemize}
\item[(i)] For any $U\in M$, the flow $\weta(\tau,U)$ exists for $\tau\in [0,\infty)$.
\item[(ii)] $\displaystyle {\frac{d}{d\tau}} J(\weta(\tau,U)) \leq 0$ 
for all $\tau\in [0,\infty)$ and
$U\in M$.
\item[(iii)] By property \eqref{8.1}, we have
    \begin{equation} \label{8.5}
        {\frac{d}{d\tau}} J(\weta(\tau,U)) \leq -\nu \quad \text{as long as}\ 
        \weta(\tau,U)\in   \wscrC_{\delta/2}\ \text{and} \ 
 J( \weta(\tau,U)  )\in \left[\ob-{\frac{\rho}{2}},\ob+{\frac{\rho}{2}}\right].
    \end{equation}
\item[(iv)] $\weta(\tau,U)=U$ for all $\tau\in [0,\infty)$ if $J(U)\leq \ob-\rho$.
\end{itemize}


\begin{Remark} \label{Remark:8.2} 
We can easily see that for any $U\in \wscrC_0$ there exists $h=h(U)>0$ such that
    \begin{equation*}
            \weta([0,h],U)\subset \wscrC_{\delta/2}.
    \end{equation*}
However it seems difficult to take $h=h(U)>0$ uniformly in $U$, since the third
component $X_3$ of the vector field $X$ is not bounded.
Replacing \eqref{5.3} with
    \begin{equation*}
        \norm{\partial_u I(\lambda_j,u_j)}_{E^*}\to 0,
    \end{equation*}
we may introduce $(PSP)_b^*$-sequences and the $(PSP)_b^*$ condition. 
If $I$ satisfies the $(PSP)_{\ob}^*$ condition, then 
we can choose a vector field $X(U)=(X_1(U), X_2(U), X_3(U))$ with 
properties \eqref{8.1}--\eqref{8.2} and
    \begin{equation*}
        \norm{X_3(U)}_{E_\theta} \leq 1.
    \end{equation*}
Thus the corresponding solution $\weta$ satisfies 
$\norm{{\frac{d}{d\tau}}\weta(\tau,U)}_{T_UM}\leq 3$ for all $\tau$ and $U$ 
and there exists $h_0>0$ independent of $U$ such that 
$\weta([0,h_0],U)\subset\wscrC_{\delta/2}$ for all $U\in \scrC_0$. 
This enables us to show for small $\alpha>0$
    \begin{equation*}
        \weta(h_0, \wscrC_0\cap [J\leq \ob+\alpha]) 
        \subset \wscrC_{\delta/2}\cap [J\leq \ob-\alpha],
    \end{equation*}
which is enough to show that $\ob$ is a critical value of $J$.

In our setting, we require \eqref{8.3} and we need the absolute value 
operator $\Upsilon$ defined in the following \cref{Section:8.3} to show the
existence of a positive critical point.
\end{Remark}

\subsection{The absolute value operator} 
\label{Section:8.3}
We define $\ABS:\, M\to M$ by
    \begin{equation*}
        \ABS(\theta,\lambda,u(x))=(\theta,\lambda,\abs{u(x)}).
    \end{equation*}
Since $u(x)\mapsto \abs{u(x)}; E\to E$ is continuous and satisfies 
$\norm{ \, \nabla\abs u\,}_2=\norm{\nabla u}_2$, the operator $\ABS:\, M\to M$
has the following properties:

\begin{itemize}
\item[(i)] $\ABS:\,M\to M$ is continuous;
\item[(ii)] $J(\ABS(U))=J(U)$ for all $U\in M$;
\item[(iii)] $\ABS(U) \in \wscrC_0$ for all $U\in M$;
\item[(iv)] $\ABS(U)=U$ for all $U=(\theta,\lambda,u)$ with $u\geq 0$.
\end{itemize}

We define the following sub-class of $\oGamma$ and the corresponding minimax 
value $\ob_0$ by
    \begin{equation*}
        \begin{aligned}
        &\oGamma_0 =\Set{ \gamma(t,\lambda)
            =(\gamma_1(t,\lambda), \gamma_2(t,\lambda))\in\oGamma | 
            \gamma_2(t,\lambda) \in \scrC_0 \ 
            \text{for all}\ (t,\lambda)\in [0,1]\times \R }, \\
        &\ob_0 = \inf_{\gamma\in \oGamma_0}\sup_{(t,\lambda)\in [0,1]\times\R} 
                I(\gamma(t,\lambda)).
        \end{aligned}
    \end{equation*}
We note that $\oGamma_0$ is a family of paths which consist of non-negative 
functions. 
By the properties of
$\ABS$, we can see for any $\gamma \in\oGamma$
    \begin{equation*}
        \ABS(\gamma(t,\lambda))\in \oGamma_0, \quad
        \sup_{(t,\lambda)\in [0,1]\times \R} J(\ABS(\gamma(t,\lambda)))
            = \sup_{(t,\lambda)\in [0,1]\times \R} J(\gamma(t,\lambda)). 
    \end{equation*}
In particular, we have $\ob_0=\ob$.

For $h>0$, the composition 
$\Upsilon\circ \weta(h,\cdot):\, U\mapsto \Upsilon(\weta(h,U))$ is well-defined 
as a map $\wscrC_0\to\wscrC_0$.  
In the next section, we use an iteration of $\Upsilon\circ \weta(h,\cdot)$ 
to deform paths.

\subsection{Iteration for positive paths} \label{Section:8.4}
Let $\delta$, $\rho$, $\nu>0$ be numbers given in \cref{Proposition:8.1}. 
We may assume $\ob-\rho>0$.
We set
    \begin{equation*}
        T={\frac{\rho}{\nu}}. 
    \end{equation*}
Here $T>0$ is chosen so that
    \begin{equation*}
        J(\weta(T,U)) \leq \ob -{\frac{\rho}{2}}
    \end{equation*}
holds under the assumption
    \begin{equation} \label{8.6}
        \weta(\tau,U) \in \wscrC_{\delta/2}
                \cap\left[ J\leq \ob+{\frac{\rho}{2}}\right]_M
        \quad \text{for}\ \tau\in [0,T].    
    \end{equation}
We note that \eqref{8.6} seems difficult to verify for 
$U\in \wscrC_0\cap\left[ J\leq \ob+{\frac{\rho}{2}}\right]_M$ in general.
See \cref{Remark:8.2}.  However for small $h=h(U)>0$ depending on 
$U\in\wscrC_0\cap\left[ J\leq \ob+{\frac{\rho}{2}}\right]_M$ it is natural to expect
that $\weta([0,h],U) \subset \wscrC_{\delta/2}$ and
$J(\weta(h,U))\leq J(U)-\nu h$ hold. Recall that $\Upsilon(M)\subset\wscrC_0$
and for $n\in\N$ we consider an iteration of the following maps:
    \begin{equation*}
        S_n = \Upsilon\circ\weta(h_n,\cdot):\, \wscrC_0\to \wscrC_0, \quad 
        h_n = {\frac{T}{n}}.
    \end{equation*}
More precisely, for $\gamma\in\oGamma_0$ with 
$\sup_{(t,\lambda)\in [0,1]\times\R} I(\gamma(t,\lambda))<\ob+{\frac{\rho}{2}}$, 
we will find a suitable $n\in\N$ and apply 
$S_n\circ S_n\circ\cdots \circ S_n$ ($n$-times) to $\gamma(t,\lambda)$ to find 
a path $\widehat\gamma\in \oGamma$ with 
$\sup_{(t,\lambda)\in[0,1]\times\R} I(\widehat\gamma(t,\lambda))<\ob-{\frac{\rho}{2}}$.

To find such an $n\in\N$, we introduce 
$\whetan:\, [0,T]\times M\to M$ by
    \begin{equation*}
        \whetan(\tau,U)=\begin{cases}
            \weta(\tau,U_0) &\text{for}\ \tau\in [0,h_n),\\
                \ \ \vdots \\
            \weta(\tau-(i-1)h_n,U_{i-1})
                &\text{for}\, \tau\in [(i-1)h_n,ih_n)\ (i=1,2,\cdots,n), \\
            \ \ \vdots \\
        U_n &\text{for}\ \tau=T.
        \end{cases}
    \end{equation*}
Here
    \begin{equation*}
        U_0=\ABS(U), \quad
        U_i= S_n(U_{i-1})=\ABS(\weta(h_n, U_{i-1})) 
        \quad (i=1,2,\cdots, n).
    \end{equation*}
We write $\whetan(\tau,U)=(\theta(\tau,U),\lambda(\tau,U),u(\tau,U))$ and observe

\begin{itemize}
\item[(0)] for $U\in \wscrC_0$, $\whetan(T,U)=(S_n\circ S_n\circ\cdots\circ S_n)(U)$ 
($n$-times);
\item[(i)] $\tau\mapsto \whetan(\tau,U)$ satisfies the first equation of \eqref{8.4} in 
intervals $[0,h_n)$, $[h_n,2h_n)$, $\dots$, $[(n-1)h_n,T]$;
\item[(ii)] $\theta(\tau,U)$, $\lambda(\tau,U)$ are continuous on $[0,T]\times M$. 
Moreover by \eqref{8.2}
    \begin{equation} \label{8.7}
        \pabs{ \frac{d}{ d\tau} \theta(\tau,U) }, \ 
        \pabs{ \frac{d}{d\tau} \lambda(\tau,U) }  \leq 1 
        \quad \text{except for}\ \tau=h_n, 2h_n, \cdots, (n-1)h_n;
    \end{equation}
\item[(iii)] $u(\tau,U):\, [0,T]\times M\to E$ is continuous except sets 
$\set{ h_n, 2h_n, \cdots, (n-1)h_n}\times M$;
\item[(iv)] 
$\abs{u(\tau,U)}:\, [0,T]\times M\to E$, $\norm{u(\tau,U)}_E: [0,T] \times M \to \R$ and 
$J(\whetan(\tau,U)):\, [0,T]\times M\to\R$ are continuous, that is,
for $i\in \{ 1,2,\cdots, n-1\}$ and $U\in M$,
    \begin{equation*}
        \begin{aligned} 
        &\abs{u(ih_n-0,U)(x)} = \abs{u(ih_n,U)(x)} \ \text{for all}\ x\in\R^N,\\
        &\norm{u(ih_n-0,U)}_E=\norm{u(ih_n,U)}_E, \quad
        J(\whetan(ih_n-0,U)) = J(\whetan(ih_n,U));
        \end{aligned}
    \end{equation*}
\item[(v)] $\tau \mapsto J(\whetan(\tau,U));\, [0,T]\to \R$ is non-increasing for 
$U \in M$;
\item[(vi)] for $U=(\theta,\lambda,u)\in M$, $\whetan(\tau,U)=U$ holds provided 
$J(U)\leq \ob-\rho$ and $u\geq 0$. Thus $\oGamma$ and $\oGamma_0$ are stable under 
$(\lambda,u)\mapsto \pi(\whetan(T,0,\lambda,u))$, that is, 
$(t,\lambda)\mapsto \pi(\whetan(T,0,\gamma(\lambda,u)))\in\oGamma_0$ 
for $\gamma\in \oGamma$ (resp. $\oGamma_0$). 
\end{itemize}

We have the following estimate for $\whetan$.

\begin{Lemma} \label{Lemma:8.3}
Suppose $U = (0,\lambda, u) \in \wscrC_0$. 
Then for any $L>0$ there exists $C_L>0$ independent of $n\in\N$ such that 
if $\norm u_E\leq L$, then 
$\whetan(\tau,U)=(\theta(\tau,U), \lambda(\tau,U), u(\tau,U))$ satisfies

\begin{enumerate}[label={\rm (\roman*)}]
\item $\abs{\theta(\tau,U)} \leq T$ for $\tau\in [0,T]$;
\item $\abs{\lambda(\tau,U)} \leq \abs{\lambda}+T$ for $\tau\in [0,T]$;
\item there exist $C_1$, $C_2>0$ such that
    \begin{equation*}
        C_1\norm\varphi_E \leq \norm\varphi_{E_{\theta(\tau,U)}} 
            \leq C_2\norm \varphi_E
        \quad \text{for} \ \varphi\in E;
    \end{equation*}
\item $\norm{u(\tau,U)}_E \leq C_L$, $\norm{{\frac{d}{d\tau}}u(\tau,U)}_E \leq C_L$ 
for $\tau\in [0,T]\setminus\set{ h_n,2h_n,\cdots, (n-1)h_n}$.
\end{enumerate}

\end{Lemma}

\begin{proof}
Properties (i) and (ii) follow from \eqref{8.7}.  Property (iii) follows 
from \eqref{7.6} and (i).

(iv) As in \eqref{7.22}, we have in each interval $((i-1)h_n,ih_n)$ ($i=1,2,\cdots$)
    \begin{equation} \label{8.8}
        \pabs{{\frac{d}{d\tau}} \norm{u(\tau,U)}_E} 
            \leq \pnorm{ {\frac{d}{d\tau}}u(\tau,U) }_E
        \leq C_1^{-1}(1+C_2\norm{u(\tau,U)}_E).         
    \end{equation}
Since $\tau\mapsto \norm{u(\tau,U)}_E$ is continuous and \eqref{8.8} holds a.e.\!
 in $[0,T]$, we apply the Gronwall inequality and observe that for any $L>0$ 
there exists $C_L>0$ such that
    \begin{equation*}
        \max_{\tau\in [0,T]}\norm{u(\tau,U)}_E\leq C_L \quad 
        \text{for}\ \norm{U}_E\leq L.
    \end{equation*}
The boundedness of
$\norm{{\frac{d}{d\tau}}u(\tau,U)}_E$ also follows from \eqref{8.8}. 
\end{proof}

\begin{Lemma} \label{Lemma:8.4}
For any $L>0$ there exists $n_L\in\N$ such that for $n\geq n_L$ and 
$U=(0,\lambda,u)$ with $\norm u_E\leq L$
    \begin{equation*}
        \whetan(\tau,U)\in \wscrC_{\delta/2} \quad \text{for}\ \tau\in [0,T].
    \end{equation*}
\end{Lemma}

    \begin{proof}
We write $\whetan(\tau,U)=(\theta(\tau), \lambda(\tau), u(\tau))$. 
For $\tau\in [(i-1)h_n,ih_n)$, we have
    \begin{equation*}
    \begin{aligned}
        \norm{u(\tau)-U_{i-1}}_{E_{\theta(\tau)}} 
        &\leq C_2 \norm{u(\tau)-u((i-1)h_n)}_E \leq C_2\int_{(i-1)
        h_n}^\tau \pnorm{{\frac{d}{ds}}u(s) }_E\, ds \\
        &\leq C_2 C_L(\tau-(i-1)
        h_n) \leq C_2 C_Lh_n= C_2 C_L{\frac{T}{n}}. 
    \end{aligned}
    \end{equation*}
Choosing $n_L\in\N$ with $C_2C_L {\frac{T}{n_L}}<\delta/2$ and 
noting $U_{i-1}\in\wscrC_0$, 
we have $\whetan(\tau,U)\in \wscrC_{\delta/2}$ for $\tau\in [(i-1)h_n,ih_h)$. 
Since $i\in\{ 1,2,\cdots, n\}$ is arbitrary, \cref{Lemma:8.4} holds.
\end{proof}

\begin{Corollary} \label{Corollary:8.5}
For $U=(0,\lambda,u)$ with $\norm u_E\leq L$ and $J(U)\leq \ob+{\frac{\rho}{2}}$,
it holds that for $n\geq n_L$
    \begin{equation*}
        J(\whetan(T,U)) \leq \ob-{\frac{\rho}{2}}.
    \end{equation*}
\end{Corollary}

\begin{proof}
By \cref{Lemma:8.4}, $\whetan(\tau,U) \in \wscrC_{\delta/2}$ holds 
for $n\geq n_L$ and $\tau \in [0,T]$.
We argue indirectly and suppose that $J(\whetan(T,U))> \ob-{\frac{\rho}{2}}$. 
Then by monotonicity of $J \circ \weta_n$, we have 
$\ob-{\frac{\rho}{2}}< J(\whetan(\tau,U)) \leq \ob+{\frac{\rho}{2}}$ 
for all $\tau\in [0,T]$. \
Applying \eqref{8.5} in intervals $[0,h_n]$, $[h_n, 2h_n]$, $\cdots$, $[(n-1)h_n,T]$, 
we deduce
    \begin{equation*}
        J(\whetan(ih_n,U)) \leq J(\whetan((i-1)h_n, U) -\nu h_n 
            \quad \text{for}\ i=1,2,\cdots, n.
    \end{equation*}
Thus, noting $nh_n= T={\frac{\rho}{\nu}}$, we observe
    \begin{equation*}
        J(\whetan(T,U)) = J(\whetan(nh,U)) \leq J(\whetan(0, U)) -n \nu h_n 
        \leq \ob+{\frac{\rho}{2}} -\rho = \ob -{\frac{\rho}{2}}.
    \end{equation*}
This is a contradiction. 
\end{proof}

\subsection{Proof of \cref{Theorem:1.1}: Case 2 ($\ob>0$)}
 \label{Section:8.5}
We are now ready to find the desired deformation of $\gamma_*\in \oGamma$ and 
get a contradiction with \ref{(N)}. 

We may assume $\ob-\rho>0$ and suppose that $\gamma=(\gamma_1,\gamma_2)\in\oGamma_0$
satisfies
    \begin{equation*}
        \sup_{(t,\lambda)\in [0,1]\times\R} I(\gamma(t,\lambda)) 
        \leq \ob+{\frac{\rho}{2}}.
    \end{equation*}
By the definition of $\oGamma_{0}$ there exists $R>0$ such that if 
$\abs\lambda\geq R$ or $t\sim 0, 1$, then 
    \begin{equation*}
        I(\gamma(t,\lambda)) \leq \ob-\rho \quad 
        \text{and}\quad \gamma_2(t,\lambda)\geq 0.
    \end{equation*}
We set
    \begin{equation*}
        L=\max_{t\in [0,1],\,\abs\lambda\leq R} \norm{\gamma_2(t,\lambda)}_E. 
    \end{equation*}
We choose $n_L\in\N$ by \cref{Lemma:8.4} and define 
$\gamma_*:\, [0,1]\times\R\to M$ by
    \begin{equation}\label{8.9}
        \gamma_*(t,\lambda)
        =\pi(\wheta_{n_L}(T,0,\gamma_1(t,\lambda),\gamma_2(t,\lambda)))
    \end{equation}
Noting 
    \begin{equation*}
        \wheta_{n_L}(T,0,\gamma_1(t,\lambda),\gamma_2(t,\lambda)) 
            = (0,\gamma_1(t,\lambda),\gamma_2(t,\lambda))
        \quad \text{if}\ \abs\lambda\geq R\ \text{or}\ t\sim 0,1,
    \end{equation*}
we have by \cref{Corollary:8.5}
    \begin{align*}
        &\gamma_*\in\oGamma_0, \\
        &I(\gamma_*(t,\lambda))
            =J(\wheta_{n_L}(T,0,\gamma_1(t,\lambda),\gamma_1(t,\lambda)))
            \leq \ob-{\frac{\rho}{2}} \quad 
            \text{for all}\ (t,\lambda)\in [0,1]\times\R.
    \end{align*}
This contradicts
the definition of $\ob$. Thus \ref{(N)} cannot hold and $I$ has a critical point 
$(\lambda,u)$ with $I(\lambda,u) = \ob$ and $u\geq 0$. 
Since $u$ is nontrivial, the strong maximum principle yields $u>0$ in $\R^N$. 
\qed

\begin{Remark} \label{Remark:8.6}
From the above arguments in Sections \ref{Section:7}--\ref{Section:8}, 
if $\underline{b} < 0 < \overline{b}$, then 
both of $\ub$ and $\ob$ are critical values of $I$ and hence 
there are at least two positive solutions of \eqref{1.1} with $m=m_1$.  
\end{Remark}

\begin{Remark}\label{Remark:8.7}
For a given $\gamma\in\oGamma_0$, the path $\gamma_*$ defined in \eqref{8.9} is 
a continuous deformation of $\gamma$.  That is, there is a continuous map 
$\widetilde\gamma(\tau,t,\lambda):\, [0,1]\times M\to M$ such that
    \begin{align*}
        &\widetilde\gamma(0,t,\lambda)=\gamma(t,\lambda), \quad 
         \widetilde\gamma(1,t,\lambda)=\gamma_*(t,\lambda) \quad 
            \text{for}\ (t,\lambda)\in [0,1]\times \R, \\
        &\widetilde\gamma(\tau,t,\lambda)=\gamma_0(t,\lambda) \quad 
            \text{for $t\sim 0$, $t\sim 1$ or $\abs{\lambda}\gg 1$ 
            and for all $\tau\in [0,1]$}.
    \end{align*}
Recall that 
$\gamma_*(t,\lambda)=(\pi\circ S_{n_L}\circ S_{n_L}\circ\cdots 
\circ S_{n_L})(0,\gamma(t,\lambda))$ 
($n_L$-times), $S_{n_L}=\Upsilon\circ\weta(h_{n_L}, \cdot)$.  Thus, 
to show $\gamma_*$ is a continuous deformation of $\gamma$, it suffices to see that 
$\pi(0,\lambda,u)=(\lambda,u)$ and 
$\weta(h_{n_L},\cdot):\, M\to M$, $\Upsilon:\, M\to M$
are homotopic to $id:\, M\to M; (\theta,\lambda,u)\mapsto (\theta,\lambda,u)$.
Continuous homotopies between these maps are given by
    \begin{align*}
        &\gamma_1:\, (\tau,\theta,\lambda,u)\mapsto \weta(\tau h_{n_L},\theta,\lambda,u);
            &&[0,1]\times M\to M, \\
        &\gamma_2:\, (\tau,\theta,\lambda,u)\mapsto (\theta,\lambda, (1-\tau) u + \tau\abs u);
            &&[0,1]\times M\to M.
     \end{align*}
We note that $\gamma_2$ satisfies $J(\gamma_2(1,\theta,\lambda,u))=J(\gamma_2(0,\theta,\lambda,u))
=J(0,\theta,\lambda,u)$ for $(\theta,\lambda,u)\in M$ but we do not have in general
    \[  J(\gamma_2(\tau,\theta,\lambda,u))\leq J(\gamma_2(0,\theta,\lambda,u))=J(0,\theta,\lambda,u)
        \quad \text{for}\ \tau\in (0,1)\ \text{and}\ (\theta,\lambda,u)\in M.
    \]

\end{Remark}

\section{Non-existence results} \label{Section:9}
In previous sections we mainly study the existence of positive solutions of 
\eqref{1.1} under the condition \ref{(g1)}, in particular we assume that $h$ is 
sublinear as $s\sim\pm\infty$.
In this section we show the sublinear condition cannot be removed for the existence 
and give a proof to \cref{Theorem:1.7}.  Here the function
    \begin{equation*}
        \rho(s)={\frac{H(s)}{s^2/2}}\in C(\R,\R)\cap C^1(\R\setminus\{ 0\},\R)
    \end{equation*}
plays a role.  We consider under the situation where
    \begin{equation*}
        \lim_{s\to 0} {\frac{\rho(s)}{\abs s^{p-1}}}\to 0 \quad 
        \text{and} \quad \rho(s)<0 \ \text{for all}\ s\in (0,\infty).
    \end{equation*}
Under condition \ref{(rho1)}, we note that 
    \begin{equation*}
            \rho(s)<0 \quad \text{for all}\ s \in (0,\infty)
    \end{equation*}
and by \cref{Remark:1.8} (i)
    \begin{equation} \label{9.1}
        \intRN H(u)-\half h(u)u\, dx >0 \quad 
        \text{for all $u\in E\setminus \set{0}$ with $u \geq 0$ in $\R^N$}.
    \end{equation}

\begin{proof}[Proof of \cref{Theorem:1.7}]
We argue by contradiction and suppose that
$(\lambda,u)\in\RE$ is a positive solution of \eqref{1.1} with $m=m_1$ 
under \eqref{1.18} and \ref{(rho1)}. 
We use an argument similar to Soave \cite{So20}.

Since $(\lambda,u)$ is a solution of \eqref{1.1}, we have
    \begin{equation} \label{9.2}
        \norm{\nabla u}_2^2 +\mu\norm u_2^2 -\intRN g(u)u\, dx=0. 
    \end{equation}
We also have the Pohozaev identity: 
    \begin{equation} \label{9.3}
        {\frac{N-2}{2}}\norm{\nabla u}_2^2 
        +N\left\{{\frac{\mu}{2}}\norm u_2^2 -\intRN G(u)\, dx\right\}=0.
    \end{equation}
Computing ${\frac{N}{4}}\times\eqref{9.2} - \half\times\eqref{9.3}$, we see 
    \begin{equation*}
        \half\norm{\nabla u}_2^2-{\frac{1}{p+1}}\norm u_{p+1}^{p+1}
        +{\frac{N}{2}}\intRN H(u)-\half h(u)u\, dx=0.
    \end{equation*}
Since $\norm u_2^2=2m_1$, \cref{Corollary:2.2} yields 
    \begin{equation*}
        \half\norm{\nabla u}_2^2-{\frac{1}{p+1}}\norm u_{p+1}^{p+1}\geq 0.
    \end{equation*}
Thus by \eqref{9.1}, we have $u=0$, which is a contradiction. 
\end{proof}

\appendix

\section{Appendix: technical lemmas}\label{Appendix:A}

\subsection{Estimates for $\displaystyle \intRN \pabs{H(\mu^{N/4}u)}\, dx$ and 
$\norm{h( \mu^{N/4} u )}_2$} \label{Appendix:A.1}

The following lemma is used repeatedly in this paper:

\begin{Lemma} \label{Lemma:A.1}
Assume
    \begin{equation} \label{A.1}
        \lim_{s\to 0}{\frac{h(s)}{s}}=0, \quad
        \lim_{\abs{s} \to \infty}{\frac{h(s)}{s}}=0
    \end{equation}
and suppose that $(\mu_j)_{j=1}^\infty\subset(0,\infty)$ satisfies $\mu_j\to 0$ or 
$\mu_j\to\infty$.  Then for any compact set $K\subset L^{2}(\R^N)$,
    \begin{equation*}
        \begin{aligned}
        &\max_{u\in K}\mu_j^{-N/2}\intRN \pabs{H(\mu_j^{N/4}u)} dx\to 0, 
        \quad \mu_j^{-N/4}  \norm{h(\mu_j^{N/4}u)}_2 \to 0
        \quad \text{as}\ j\to\infty.
        \end{aligned}
    \end{equation*}
In particular, for a strongly convergent sequence $(u_j)_{j=1}^\infty
\subset L^2(\R^N)$, 
    \begin{equation}\label{A.2}
        \mu_j^{-N/2} \intRN \pabs{H(\mu_j^{N/4}u_j)}\, dx, \quad
        \mu_j^{-N/4} \norm{h(\mu_j^{N/4}u_j)}_2 \to 0 \quad
        \text{as}\ j\to\infty.          
    \end{equation}
\end{Lemma}

\begin{proof}
We only deal with the case $\mu_j\to\infty$ and we show the first statement\\
$\max_{u\in K}\mu_j^{-N/2}\intRN \pabs{H(\mu_j^{N/4}u)} dx\to 0$.
We can show other statements in a similar way. In particular, to show \eqref{A.2},
consider a compact set $K=\overline{\set{u_j | j\in\N}}$.

Choose $u_j\in K$ such that $\intRN \pabs{H(\mu_j^{N/4}u_j)} dx
=\max_{u\in K}\intRN \pabs{H(\mu_j^{N/4}u)} dx$.
Since $K$ is compact in $L^2(\R^N)$, after extracting a subsequence, 
we may assume that $u_j(x)\to u_0(x)$ a.e. in $\R^N$ and for some 
$W(x)\in L^2(\R^N)$
    \begin{equation}\label{A.3}
        \abs{u_j(x)} \leq W(x) \quad \text{a.e. in}\ \R^N.
    \end{equation}
By \eqref{A.1}, there exists a constant $C>0$, independent of $s$, such that
$\abs{H(s)} \leq C\abs s^2$ for all $s\in\R$. Thus we have
    \begin{equation} \label{A.4}
        \mu_j^{-N/2} \pabs{H ( \mu_j^{N/4} u_j (x) ) } 
        \leq C \abs{u_j(x)}^2  \quad \text{a.e. in $\R^N$}.
    \end{equation}
Thus we have by \eqref{A.3} and \eqref{A.4}
    \begin{equation} \label{A.5}
        \mu_j^{-N/2} \pabs{ H ( \mu_j^{N/4} u_j (x) ) }
        \leq W(x)^2 \quad \text{a.e. in $\R^N$}.
    \end{equation}
In addition, we have
    \begin{equation} \label{A.6}
        \mu_j^{-N/2} \pabs{H(\mu_j^{N/4} u_j(x))} \to 0 \quad \text{a.e. in}\ \R^N
        \ \text{as}\ j\to\infty.    
    \end{equation}
To show \eqref{A.6}, for $x\in\R^N$ we consider two cases $u_0(x)=0$ and $u_0(x)\not=0$
separately.  If $x\in\R^N$ satisfies $u_0(x)=0$, then $u_j(x)\to u_0(x)=0$.  
Thus \eqref{A.4} implies \eqref{A.6}.
On the other hand, if $u_0(x)\not=0$, then $u_j(x)\not=0$
for large $j$ and $\mu_j^{N/4} \abs{u_j(x)}\to \infty$ as $j\to\infty$.
Thus by \eqref{A.1}, 
    \begin{equation*}
        \mu_j^{-N/2} \pabs{ H ( \mu_j^{N/4} u_j (x) ) }
        = \frac{ \pabs{ H ( \mu_j^{N/4} u_j (x) ) } }{(\mu_j^{N/4}\abs{u_j(x)} )^2}
      \abs{u_j(x)}^2 \to 0
    \end{equation*}
and hence \eqref{A.6} holds. 
Together with \eqref{A.5}, the dominated convergence theorem yields the 
desired statement.
\end{proof}

\subsection{Optimal paths for nonlinear scalar equations} \label{Appendix:A.2}
We assume that $f\in C(\R,\R)$ satisfies
\begin{enumerate}[label={\rm (f\arabic*)}]
    \setcounter{enumi}{-1}
\item \label{(f0)}
$f(0)=0$;
\item \label{(f1)}
$f$ has subcritical growth at $s=\pm\infty$, that is,
\begin{itemize}
\item[(1)] for $N\geq 3$,
     $\displaystyle   \lim_{s\to\pm\infty} {\frac{f(s)}{\abs s^{2^*-1}}} = 0$;
\item[(2)] for $N=2$
     $\displaystyle   \lim_{s\to\pm \infty}{\frac{f(s)}{e^{\alpha s^2}}} = 0 \quad
        \text{for any}\ \alpha>0$;
\end{itemize}
\item \label{(f2)}
    $\displaystyle    \lim_{s\to 0}{\frac{f(s)}{s}}<0$;
\item \label{(f3)}
there exists $s_0>0$ such that $F(s_0) > 0$ where 
$\displaystyle F(s) =\int_0^s f(\tau)\,d\tau$. 
\end{enumerate}

A closely related situation is studied in \cite{BeL,BeGaKa83,HIT}. We set 
    \begin{equation*}
        L(u) = \half\norm{\nabla u}_2^2 -\intRN F(u)\,dx:\, E\to \R.
    \end{equation*}
In \cite{JT0}, for a non-zero critical point $u_0\in E$ of $L$, 
a path $\gamma_0 \in C([0,1] , E ) $ is constructed with
    \begin{align}   
        &\gamma_0(0)=0, \quad L(\gamma_0(1))<0,     \label{A.7}\\
        &u_0 \in \gamma_0([0,1]),                   \label{A.8}\\
        &\max_{t\in [0,1]} L(\gamma_0(t)) =L(u_0).  \label{A.9}
    \end{align}
The following proposition gives more detailed information on the path, which is
a key of our argument in \cref{Section:4.2} and \cref{Section:4.4}.

\begin{Proposition} \label{Proposition:A.2}
Assume $N\geq 2$ and {\rm \ref{(f0)}--\ref{(f3)}}. Suppose that 
$u_0\in E\setminus\{ 0\}$ is a critical point of $L$.
Then there exists a path $\gamma_0 \in C([0,1],E)$ with properties
\eqref{A.7}--\eqref{A.9}.
Moreover,
\begin{enumerate}
\item[\rm (i)] when $N\geq 3$, for a large $T>1$ the path $\gamma$ is of the form 
    \[  \gamma(t)(x)=\begin{cases}
            u_0({\frac{x}{Tt}}) &\text{for}\ t\in (0,1],\\
            0        &\text{for}\ t=0.
        \end{cases}
    \]
Setting $t_0= 1/T\in (0,1)$, the path $\gamma$ satisfies 
$\gamma(t_0)=u_0$ and $t\mapsto L(\gamma(t))$ is strictly increasing 
in $[0,t_0)$ and strictly decreasing in $(t_0,1]$; 
\item[\rm (ii)] when $N=2$, setting $u_{0\theta}(x)=u_0(x/\theta)$ and for some small 
$\theta_0\in (0,1)$, large $\theta_1\in (1,\infty)$ and 
some $t_1>1$ close to $1$, joining the following three paths, we obtain the desired
path $\gamma_0:\,[0,1]\to E$ after re-parametrization:
	\[	{\rm (a)}\ [0,1]\to E;\, t\mapsto tu_{0\theta_0}, \quad 
		{\rm (b)}\ [\theta_0,\theta_1]\to E;\, \theta\mapsto u_{0\theta}, \quad
		{\rm (c)}\ [1,t_1]\to E;\, t\mapsto tu_{0\theta_1},
	\]
\noindent
here the path in {\rm (a)--(c)} satisfies for some constant $C>0$
    \begin{align}
    &{\frac{d}{dt}}L(tu_{0\theta_0})\geq Ct \quad  \text{for}\  t\in (0,1], 
                                    \label{A.10}\\
    &L(u_{0\theta})\equiv L(u_0)\quad \text{for} \ \theta\in [\theta_0,\theta_1], 
                                    \label{A.11}\\
    &{\frac{d}{dt}}L(tu_{0\theta_0})\leq -C \quad \text{for}\  t\in [1,t_1],
                                    \label{A.12}\\
    &L(t_1u_{0\theta_1})<0.     \label{A.13}
    \end{align}
\end{enumerate}
\end{Proposition}

The proof of \cref{Proposition:A.2} is given in \cite{JT0} and we omit it here.

\subsection{Deformation argument for possibly sign-changing solutions}\label{Appendix:A.3}
In this appendix, we give a variant of deformation results given in Sections \ref{Section:5} 
and \ref{Section:7},
which may be useful to find the existence and multiplicity of possibly sign-changing
solutions.

In this appendix, we do not assume oddness of $g:\,\R\to\R$ and assume 

\smallskip

{\parindent=2\parindent

\begin{description}
 \item[$(g0\#)$] $g\in C(\R,\R)$.
\end{description}

\smallskip

\noindent
As variants of conditions (g1)--(g3), we may consider the following conditions:

\smallskip

\begin{description}
\item[$(g1\#)$]  $g(s)=\abs s^{p-1}s+h(s)$ and
    \begin{equation*}
        \lim_{s\to 0}{\frac{h(s)}{\abs s^{p-1}s}}=0, \quad
        \lim_{s\to \pm\infty}{\frac{h(s)}{\abs s}}=0;
    \end{equation*}
\item[$(g2\#)$] if $u_0\in F$ satisfies \eqref{1.13}, then $u_0\equiv 0$;
\item[$(g3\#)$] $NG(s) -{\frac{N-2}{2}}g(s)s\geq 0$ for all $s\in\R$.
\end{description}

}
\smallskip

\noindent
Condition $(g2\#)$ ensures that \eqref{1.13} has no non-trivial solutions and $(g2\#)$ is
stronger than $(g2)$.  As in Proposition \ref{Proposition:1.2}, we have

\medskip

\begin{Lemma} \label{Lemma:A.3}  
$(g2\#)$ follows from $(g3\#)$.
\end{Lemma}

\medskip

We also have

\medskip

\begin{Proposition} \label{Proposition:A.4} 
Assume $(g0\#)$--$(g2\#)$.  Then for $b\not=0$, the $(PSPC)_b$ condition
holds for $I$.
\end{Proposition}

\medskip

In fact, under $(g2\#)$, $Z$ has no non-trivial critical points including 
sign-changing ones.  Modifying the argument in the proof of \cref{Proposition:5.4} 
in Case C (\cref{Section:5.4})
slightly and repeating the arguments for Cases A and B, we can show 
\cref{Proposition:A.4}.

Thus the argument for \cref{Proposition:7.1} is applicable and we have 
the following deformation result.

\medskip

\begin{Proposition} \label{Proposition:A.5}
Assume $(g0\#)$--$(g2\#)$.  Then for $b\not=0$ the conclusion of 
\cref{Proposition:7.1} holds.
\end{Proposition}

\section*{Acknowledgments}
S.C. is partially supported by INdAM-GNAMPA and by PRIN PNRR P2022YFAJH \lq\lq Linear and 
Nonlinear PDEs:  New directions and applications''. 
S.C. thanks acknowledge financial support from PNRR MUR project PE0000023 NQSTI 
- National Quantum Science and Technology Institute (CUP H93C22000670006).
M.G. is supported by INdAM-GNAMPA Project \lq\lq Mancanza di regolarit\`a e spazi non lisci: 
studio di autofunzioni e autovalori'', codice CUP E53C23001670001.
N.I. is partially supported by JSPS KAKENHI Grant Numbers JP 19H01797, 19K03590 and 24K06802. 
K.T. is partially supported by JSPS KAKENHI Grant Numbers JP18KK0073, JP19H00644 and 
JP22K03380.

\end{document}